\documentclass[12pt]{article}
\usepackage{amsthm,amsfonts,amssymb,epsfig,graphics,amsmath,amsbsy,color,enumerate,tikz}

\theoremstyle{plain}
\newtheorem{theorem}{Theorem}
\newtheorem{lemma}{Lemma}

\newtheorem{remark}{Remark}
\newtheorem{definition}{Definition}

\newtheorem{claim}{Claim}

\newcommand{\mas}{\operatorname{Mas}}
\newcommand{\mor}{\operatorname{Mor}}
\newcommand{\dom}{\operatorname{dom}}
\newcommand{\ran}{\operatorname{ran}}
\newcommand{\colspan}{\operatorname{colspan}}
\newcommand{\loc}{\operatorname{loc}}

\newcommand{\AC}{\operatorname{AC}}
\newcommand{\Span}{\operatorname{Span}}
\newcommand{\rank}{\operatorname{rank}}
\newcommand{\ess}{\operatorname{ess}}

\numberwithin{equation}{section}
\numberwithin{lemma}{section}
\numberwithin{theorem}{section}
\numberwithin{remark}{section}
\numberwithin{claim}{section}
\numberwithin{corollary}{section}
\numberwithin{proposition}{section}
\numberwithin{definition}{section}
\numberwithin{condition}{section}
\numberwithin{figure}{section}

\title{Renormalized Oscillation Theory for Singular Linear 
Hamiltonian Systems}

\author{Peter Howard and Alim Sukhtayev}

\begin{document}

\maketitle

\begin{abstract} 
Working with a general class of linear Hamiltonian systems 
with at least one singular boundary condition, we show that 
renormalized oscillation results can be obtained in a 
natural way through consideration of the Maslov index 
associated with appropriately chosen paths of Lagrangian 
subspaces of $\mathbb{C}^{2n}$. This extends previous work
by the authors for regular linear Hamiltonian systems.     
\end{abstract}

\section{Introduction}\label{introduction}

We consider linear Hamiltonian systems  
\begin{equation} \label{linear-hammy}
J y' = (B_0 (x) + \lambda B_1 (x)) y; \quad y (x; \lambda) \in \mathbb{C}^{2n},
\quad n \in \{1, 2, \dots \},
\end{equation}
where $J$ denotes the standard symplectic matrix 
\begin{equation*}
J 
=
\begin{pmatrix}
0_n & - I_n \\
I_n & 0_n
\end{pmatrix}.
\end{equation*}
We specify (\ref{linear-hammy}) on intervals $(a, b)$, with 
$-\infty \le a < b \le +\infty$, and we assume throughout that 
$B_0, B_1 \in L^1_{\loc} ((a, b), \mathbb{C}^{2n \times 2n})$,
and additionally that $B_0 (x)$ and $B_1 (x)$ are both self-adjoint
for a.e. $x \in (a, b)$. For convenient reference, we refer to these 
assumptions as Assumptions {\bf (A)}. In addition, we make the following 
Atkinson-type positivity assumption.

\medskip
{\bf (B)} If $y(\cdot; \lambda) \in \AC_{\loc} ((a, b), \mathbb{C}^{2n})$ 
is any non-trivial solution of (\ref{linear-hammy}), then 
\begin{equation*}
\int_c^d (B_1 (x) y(x; \lambda), y(x; \lambda)) dx > 0,
\end{equation*}
for all $[c, d] \subset (a, b)$. (Here, $\AC_{\loc}$ denotes local absolute
continuity, and $(\cdot, \cdot)$ denotes the usual inner product
on $\mathbb{C}^{2n}$.)
\medskip

Our goal is to associate (\ref{linear-hammy}) with one or more
self-adjoint operators $\mathcal{L}$ (see Lemma 
\ref{self-adjoint-operator-lemma} below), and to use renormalized 
oscillation theory to count the number of eigenvalues 
$\mathcal{N} ([\lambda_1, \lambda_2))$ that each 
such operator has on a given interval $[\lambda_1, \lambda_2) 
\subset \mathbb{R}$ for which the closure $[\lambda_1, \lambda_2]$
has empty intersection with the essential spectrum of 
the operator.
We will formulate our results for two cases: (1) when $x = a$ 
is a regular boundary point for (\ref{linear-hammy}); and 
(2) when $x = a$ is a singular boundary point for 
(\ref{linear-hammy}). (We take (\ref{linear-hammy}) to 
be singular at $x = b$ in both cases; the case in which 
(\ref{linear-hammy}) is regular at both endpoints has 
been analyzed in \cite{HS2}.) The case in which 
(\ref{linear-hammy}) is regular at $x = a$ corresponds 
with the following additional assumption. 

\medskip
{\bf (A)$^\prime$} The value $a$ is finite, and for any 
$c \in (a, b)$, we have $B_0, B_1 \in L^1 ((a, c), \mathbb{C}^{2n \times 2n})$.
\medskip
 
Our starting point will be to specify an appropriate Hilbert 
space to work in, and for this we follow \cite{Krall2002}.
We denote by $\tilde{L}^2_{B_1} ((a, b), \mathbb{C}^{2n})$
the set of all Lebesgue measureable functions $f$ defined on $(a, b)$
so that 
\begin{equation*}
\|f\|_{B_1} := \Big(\int_a^b (B_1 (x) f(x), f(x)) dx\Big)^{1/2} 
< \infty.
\end{equation*}
Correspondingly, we denote by $\mathcal{Z}_{B_1}$ the subset of 
$\tilde{L}^2_{B_1} ((a, b), \mathbb{C}^{2n})$ comprising 
elements $f \in \tilde{L}^2_{B_1} ((a, b), \mathbb{C}^{2n})$ so that 
$\|f\|_{B_1} = 0$. Our Hilbert space will be the quotient
space, 
\begin{equation*}
L^2_{B_1} ((a, b), \mathbb{C}^{2n})
= \tilde{L}^2_{B_1} ((a, b), \mathbb{C}^{2n})/\mathcal{Z}_{B_1}.
\end{equation*}
I.e., two functions $f, g \in L^2_{B_1} ((a, b), \mathbb{C}^{2n})$ 
are equivalent if and only if $\|f - g\|_{B_1} = 0$. With 
this specification, $\| \cdot \|_{B_1}$ is a norm on 
$L^2_{B_1} ((a, b), \mathbb{C}^{2n})$. We equip 
$L^2_{B_1} ((a, b), \mathbb{C}^{2n})$ with the inner product
\begin{equation*}
\langle f, g \rangle_{B_1} := \int_a^b (B_1 (x) f(x), g(x)) dx.
\end{equation*}
In all of these specifications, we emphasize that $B_1 (x)$
need not be an invertible matrix. 

We now introduce a maximal operator associated with 
(\ref{linear-hammy}).

\begin{definition} \label{maximal-operator}
(i) We denote by $\mathcal{D}_M$ the collection of 
all 
\begin{equation*}
    y \in \AC_{\loc} ((a, b), \mathbb{C}^{2n}) \cap L^2_{B_1} ((a, b), \mathbb{C}^{2n})
\end{equation*} 
for which there exists some $f \in L^2_{B_1} ((a, b), \mathbb{C}^{2n})$ so that 
\begin{equation*}
Jy' - B_0 (x) y
= B_1 (x) f,
\end{equation*}
for a.e. $x \in (a, b)$. We will refer to $\mathcal{D}_M$ 
as the maximal domain, and we note that $f$ is uniquely determined 
in $L^2_{B_1} ((a, b), \mathbb{C}^{2n})$. (If $f$ and $g$ are two 
functions associated with the same $y \in \mathbf{\mathcal{D}}_M$, then 
$B_1 (x) (f - g) = 0$ for a.e. $x \in (a, b)$,
so that $f = g$ in $L^2_{B_1} ((a, b), \mathbb{C}^{2n})$.)

(ii) We define the maximal operator 
$\mathcal{L}_{M}: L^2_{B_1} ((a, b), \mathbb{C}^{2n}) 
\to L^2_{B_1} ((a, b), \mathbb{C}^{2n})$ as the map taking 
a given $y \in \mathcal{D}_M$ to the unique 
$f \in L^2_{B_1} ((a, b), \mathbb{C}^{2n})$ guaranteed by 
the definition of $\mathcal{D}_M$. We note 
particularly that $y (\cdot; \lambda) \in \mathcal{D}_M$
solves (\ref{linear-hammy}) iff and only if 
$\mathcal{L}_M y = \lambda y$ a.e. in $(a, b)$. 
\end{definition}

The following terminology will be convenient for the discussion.

\begin{definition} \label{left-right-definition}
We say that a solution $y (\cdot;\lambda) \in \AC_{\loc} ((a,b),\mathbb{C}^{2n})$
of (\ref{linear-hammy}) {\it lies left} in $(a, b)$ if for any $c \in (a,b)$, 
the restriction of $y(\cdot; \lambda)$ to $(a,c)$ is in 
$L^2_{B_1} ((a, c), \mathbb{C}^{2n})$. Likewise, 
we say that a solution $y (\cdot;\lambda) \in \AC_{\loc} ((a,b),\mathbb{C}^{2n})$
of (\ref{linear-hammy}) {\it lies right} in $(a, b)$ if for any $c \in (a,b)$, 
the restriction of $y(\cdot; \lambda)$ to $(c,b)$ is in 
$L^2_{B_1} ((c,b), \mathbb{C}^{2n})$. For each fixed
$\lambda \in \mathbb{C}$ we will denote by 
$m_a (\lambda)$ the dimension of the space of solutions to 
(\ref{linear-hammy}) that lie left in $(a, b)$, and 
we will denote by $m_b (\lambda)$ the dimension of the 
space of solutions to (\ref{linear-hammy}) that lie right 
in $(a, b)$.  
\end{definition}

We will show in Section \ref{operator-section} that if Assumptions {\bf (A)} and 
{\bf (B)} hold, then for any $\lambda \in \mathbb{C} \backslash \mathbb{R}$, 
(\ref{linear-hammy}) admits at least $n$ solutions that lie
left in $(a, b)$ and at least $n$ solutions that lie right in 
$(a, b)$. According to Theorem V.2.2 in \cite{Krall2002}, $m_a (\lambda)$
and $m_b (\lambda)$ are both constant for all $\lambda$ 
with $\textrm{Im}\,\lambda > 0$, and the same statement is 
true for $\textrm{Im}\,\lambda < 0$. In the event that 
$B_0 (x)$ and $B_1 (x)$ have real-valued entries for a.e. $x \in (a, b)$,
it is furthermore the case that $m_a (\lambda)$
and $m_b (\lambda)$ are both constant for all 
$\lambda \in \mathbb{C} \backslash \mathbb{R}$. (See our 
Remark \ref{assumption-c-remark}.) We will
allow $B_0 (x)$ and $B_1 (x)$ to have complex-valued entries, but 
we will make the following consistency assumption: 

\medskip
{\bf (C)} The values $m_a (\lambda)$ and $m_b (\lambda)$ are 
both constant for all $\lambda \in \mathbb{C} \backslash \mathbb{R}$.
We denote these common values $m_a$ and $m_b$. 
\medskip

In the event that Assumption {\bf (A)$^\prime$} also holds, it's 
clear that $m_a (\lambda) = 2n$ for all $\lambda \in \mathbb{C}$.
In the terminology of our next definition, this means that under
Assumption {\bf (A)$^\prime$}, (\ref{linear-hammy}) 
is in the limit circle case at $x = a$. In this case, Assumption 
{\bf (C)} holds immediately for $x = a$, with $m_a = 2n$.

\begin{definition} If $m_a = n$, we say that (\ref{linear-hammy}) 
is in the limit point case at $x = a$, and if $m_a = 2n$, we say 
that (\ref{linear-hammy}) is in the limit circle case at $x = a$.
If $m_a \in (n, 2n)$, we say that (\ref{linear-hammy}) 
is in the limit-$m_a$ case at $x = a$. Analogous specifications
are made at $x = b$.
\end{definition}

Under Assumptions {\bf (A)}, {\bf (B)}, and {\bf (C)},
we will show that by taking an appropriate
selection of solutions that lie left in $(a, b)$, 
$\{u^a_j (x; \lambda)\}_{j=1}^n$, and an appropriate 
selection of solutions that lie right in $(a, b)$,
$\{u^b_j (x; \lambda)\}_{j=1}^n$, we can specify the domain
of a self-adjoint restriction of $\mathcal{L}_M$, which we will 
denote $\mathcal{L}$. For the purposes of this introduction, we
will sum this development up in the following lemma, for which 
we denote by $U^a (x; \lambda)$ the matrix comprising the 
vector functions $\{u^a_j (x; \lambda)\}_{j=1}^n$ as its columns, 
and by $U^b (x; \lambda)$ the matrix comprising the 
vector functions $\{u^b_j (x; \lambda)\}_{j=1}^n$ as its columns.
The selection process is described in detail in 
Section \ref{operator-section}; see especially the summary
in Remark \ref{selection-remark}. 

\begin{lemma} \label{self-adjoint-operator-lemma}
(i) Let Assumptions {\bf (A)}, {\bf (B)}, and {\bf (C)} hold, and let 
$\lambda_0 \in \mathbb{C} \backslash \mathbb{R}$ be fixed. Then there
exists a selection of solutions $\{u^a_j (x; \lambda_0)\}_{j=1}^n$
to (\ref{linear-hammy}) that lie left in $(a, b)$, along with 
a selection of solutions $\{u^b_j (x; \lambda_0)\}_{j=1}^n$
to (\ref{linear-hammy}) that lie right in $(a, b)$ so 
that the restriction of $\mathcal{L}_M$ to the domain 
\begin{equation*}
    \mathcal{D} := \{y \in \mathcal{D}_M: \lim_{x \to a^+} U^a (x; \lambda_0)^* J y(x) = 0,
    \quad \lim_{x \to b^-} U^b (x; \lambda_0)^* J y(x) = 0\}
\end{equation*}
is a self-adjoint operator. We will denote this operator $\mathcal{L}$.

\medskip
(ii) Let Assumptions {\bf (A)}, {\bf (A)$^\prime$}, {\bf (B)}, and 
{\bf (C)} hold, and let 
$\lambda_0 \in \mathbb{C} \backslash \mathbb{R}$ be fixed. In addition, 
let $\alpha \in \mathbb{C}^{n \times 2n}$ denote any fixed matrix
satisfying $\rank \alpha = n$ and $\alpha J \alpha^* = 0$. Then there
exists a selection of solutions $\{u^b_j (x; \lambda_0)\}_{j=1}^n$
to (\ref{linear-hammy}) that lie right in $(a, b)$ so 
that the restriction of $\mathcal{L}_M$ to the domain 
\begin{equation*}
    \mathcal{D}^{\alpha} := \{y \in \mathcal{D}_M: \alpha y(a) = 0,
    \quad \lim_{x \to b^-} U^b (x; \lambda_0)^* J y(x) = 0\}
\end{equation*}
is a self-adjoint operator. We will denote this operator $\mathcal{L}^{\alpha}$.
\end{lemma}

In order to set some notation and terminology for this discussion, we make the following standard
definitions.

\begin{definition} \label{spectrum-definition}
We denote by $\rho (\mathcal{L})$ the usual resolvent set 
\begin{equation*}
    \begin{aligned}
    \rho (\mathcal{L}) &:= 
    \{\lambda \in \mathbb{C}: (\mathcal{L} - \lambda I)^{-1}: 
    L^2_{B_1} ((a, b), \mathbb{C}^{n}) \to L^2_{B_1} ((a, b), \mathbb{C}^{n}) \\
    & \quad \quad \quad \textrm{ is a bounded linear operator}\},
    \end{aligned}
\end{equation*}
and we denote by $\sigma (\mathcal{L})$ the spectrum of $\mathcal{L}$, 
$\sigma (\mathcal{L}) := \mathbb{C} \backslash \rho (\mathcal{L})$. In addition, 
we define
the point spectrum of $\mathcal{L}$ to the be collection of eigenvalues,
\begin{equation*}
    \sigma_p (\mathcal{L}) := \{\lambda \in \mathbb{C}: \mathcal{L}y = \lambda y 
    \textrm{ for some } y \in \mathcal{D} \backslash \{0\}\},
\end{equation*}
and we define the essential spectrum of $\mathcal{L}$, denoted 
$\sigma_{\ess} (\mathcal{L})$ to be the collection of all $\lambda \in \mathbb{C}$
so that $\lambda \notin \rho(\mathcal{L})$ and $\lambda$ is not an isolated
eigenvalue of $\mathcal{L}$ with finite multiplicity. Finally, we 
define the discrete spectrum of $\mathcal{L}$ to 
be $\sigma_{\rm discrete} (\mathcal{L}) = \sigma (\mathcal{L}) \backslash \sigma_{\ess} (\mathcal{L})$.
We will use precisely 
the same definitions for $\mathcal{L}^{\alpha}$, with $\mathcal{D}$ replaced
by $\mathcal{D}^{\alpha}$. 
\end{definition}

Our primary tool for this analysis will be the 
Maslov index, and as a starting point for a discussion of this object, we 
define what we will mean by a Lagrangian subspace of $\mathbb{C}^{2n}$. 

\begin{definition} \label{lagrangian_subspace}
We say $\ell \subset \mathbb{C}^{2n}$ is a Lagrangian subspace of $\mathbb{C}^{2n}$
if $\ell$ has dimension $n$ and
\begin{equation} 
(J u, v) = 0, 
\end{equation} 
for all $u, v \in \ell$. In addition, we denote by 
$\Lambda (n)$ the collection of all Lagrangian subspaces of $\mathbb{C}^{2n}$, 
and we will refer to this as the {\it Lagrangian Grassmannian}. 
\end{definition}

\begin{remark} Following the convention of Arnol'd's foundational 
paper \cite{Arnold67}, the notation $\Lambda (n)$ is often used to denote the 
Lagrangian Grassmannian associated with $\mathbb{R}^{2n}$. Our expectation 
is that it can be used in the current setting of $\mathbb{C}^{2n}$ without
confusion. We note that the Lagrangian Grassmannian associated with 
$\mathbb{C}^{2n}$ has been considered by a number of authors, including
(ordered by publication date)
Bott \cite{Bott56}, Kostrykin and Schrader \cite{KS99}, Arnol'd
\cite{Arnold00}, and Schulz-Baldes \cite{S-B07, S-B12}. It is shown
in all of these references that $\Lambda (n)$ is 
homeomorphic to the set of $n \times n$ unitary matrices $U (n)$,
and in \cite{S-B07, S-B12} the relationship is shown to be 
diffeomorphic. It is also shown in \cite{S-B07} that the fundamental
group of $\Lambda (n)$ is isomorphic to the integers $\mathbb{Z}$.
\end{remark}

Any Lagrangian subspace of $\mathbb{C}^{2n}$ can be
spanned by a choice of $n$ linearly independent vectors in 
$\mathbb{C}^{2n}$. We will generally find it convenient to collect
these $n$ vectors as the columns of a $2n \times n$ matrix $\mathbf{X}$, 
which we will refer to as a {\it frame} for $\ell$. Moreover, we will 
often coordinatize our frames as $\mathbf{X} = {X \choose Y}$, where $X$ and $Y$ are 
$n \times n$ matrices. Following \cite{F} (p. 274), we specify 
a metric on $\Lambda (n)$ in terms of appropriate orthogonal projections. 
Precisely, let $\mathcal{P}_i$ 
denote the orthogonal projection matrix onto $\ell_i \in \Lambda (n)$
for $i = 1,2$. I.e., if $\mathbf{X}_i$ denotes a frame for $\ell_i$,
then $\mathcal{P}_i = \mathbf{X}_i (\mathbf{X}_i^* \mathbf{X}_i)^{-1} \mathbf{X}_i^*$.
We take our metric $d$ on $\Lambda (n)$ to be defined 
by 
\begin{equation*}
d (\ell_1, \ell_2) := \|\mathcal{P}_1 - \mathcal{P}_2 \|,
\end{equation*} 
where $\| \cdot \|$ can denote any matrix norm. We will say 
that a path of Lagrangian subspaces 
$\ell: \mathcal{I} \to \Lambda (n)$ is continuous provided it is 
continuous under the metric $d$. 

Suppose $\ell_1 (\cdot), \ell_2 (\cdot)$ denote continuous paths of Lagrangian 
subspaces $\ell_i: \mathcal{I} \to \Lambda (n)$, $i = 1,2$, for some parameter interval 
$\mathcal{I}$ (not necessarily closed and bounded). 
The Maslov index associated with these paths, which we will 
denote $\mas (\ell_1, \ell_2; \mathcal{I})$, is a count of the number of times
the paths $\ell_1 (\cdot)$ and $\ell_2 (\cdot)$ intersect, counted
with both multiplicity and direction. (In this setting, if we let 
$t_*$ denote the point of intersection (often referred to as a 
{\it conjugate point}), then multiplicity corresponds with the dimension 
of the intersection $\ell_1 (t_*) \cap \ell_2 (t_*)$; a precise definition of what we 
mean in this context by {\it direction} will be
given in Section \ref{maslov-section}.) 

In order to formulate our results for the case in which (\ref{linear-hammy}) 
is regular at $x = a$, we introduce the $2n \times n$ matrix solution
$\mathbf{X}_{\alpha} (x; \lambda)$ to the initial value problem
\begin{equation} \label{frame-alpha}
    \begin{aligned}
    J \mathbf{X}_{\alpha}' &= (B_0 (x) + \lambda B_1 (x)) \mathbf{X}_{\alpha} \\
    \mathbf{X}_{\alpha} &(a; \lambda) = J \alpha^*.
    \end{aligned}
\end{equation}
Under our assumptions {\bf (A), (A)$^\prime$}, 
we can conclude that for each $\lambda \in \mathbb{C}$, 
$\mathbf{X}_{\alpha} (\cdot; \lambda) \in AC_{\loc} ([a,b), \mathbb{C}^{2n \times n})$. 
In addition, $\mathbf{X}_{\alpha} \in C([a,b) \times \mathbb{C}, \mathbb{C}^{2n \times n})$,
and $\mathbf{X}_{\alpha} (x; \cdot)$ is analytic in 
$\lambda$. (See, for example, \cite{Weidmann1987}.)
As shown in \cite{HJK}, for each pair $(x, \lambda) \in [a,b) \times \mathbb{R}$,
$\mathbf{X}_{\alpha} (x; \lambda)$ is the frame for a Lagrangian subspace
of $\mathbb{C}^{2n}$, which we will denote $\ell_{\alpha} (x; \lambda)$. 
(In \cite{HJK}, the authors make slightly 
stronger assumptions on $B_0 (x)$ and $B_1 (x)$, but their proof
carries over immediately into our setting.)

For the frame associated with the right endpoint, 
we let $[\lambda_1, \lambda_2]$, $\lambda_1 < \lambda_2$,
be such that $[\lambda_1, \lambda_2] \cap \sigma_{\ess} (\mathcal{L}^{\alpha}) = \emptyset$.
In Section \ref{operator-section}, we will show that 
for each $\lambda \in [\lambda_1, \lambda_2]$, there exists a
$2n \times n$ matrix solution $\mathbf{X}_b (x; \lambda)$
to the ODE 
\begin{equation} \label{frame-b}
\begin{aligned}
J \mathbf{X}_b' =& (B_0 (x) + \lambda B_1 (x)) \mathbf{X}_b \\
\lim_{x \to b^-} &U^b (x; \lambda_0)^* J \mathbf{X}_b (x; \lambda) = 0,  
\end{aligned}
\end{equation}
where the matrix $U^b (x; \lambda_0)$ is described in 
Lemma \ref{self-adjoint-operator-lemma} (and the paragraph
leading into that lemma).
In addition, we will check that for each pair $(x, \lambda) \in [a,b) \times [\lambda_1, \lambda_2]$,
$\mathbf{X}_{b} (x; \lambda)$ is the frame for a Lagrangian subspace
of $\mathbb{C}^{2n}$, which we will denote $\ell_b (x; \lambda)$, 
and we will also check that 
$\ell_b \in C([a, b) \times [\lambda_1, \lambda_2], \Lambda (n))$.

In Section \ref{theorems-section}, we will establish the following 
theorem. 

\begin{theorem} \label{regular-singular-theorem}
Let Assumptions {\bf (A)}, {\bf (A)$^\prime$}, {\bf (B)}, 
and {\bf (C)} hold,
and assume that for some pair $\lambda_1, \lambda_2 \in \mathbb{R}$,
$\lambda_1 < \lambda_2$, we have 
$\sigma_{\ess} (\mathcal{L}^{\alpha}) \cap [\lambda_1, \lambda_2] = \emptyset$.
If $\ell_{\alpha} (\cdot; \lambda_1)$ and $\ell_b (\cdot; \lambda_2)$
denote the paths of Lagrangian subspaces of $\mathbb{C}^{2n}$ constructed just above, 
and $\mathcal{N}^{\alpha} ([\lambda_1, \lambda_2))$
denotes a count of the number of eigenvalues $\mathcal{L}^{\alpha}$
has on the interval $[\lambda_1, \lambda_2)$, then 
\begin{equation} \label{regular-singular-theorem-inequality}
  \mathcal{N}^{\alpha} ([\lambda_1, \lambda_2))
  \ge \mas (\ell_{\alpha} (\cdot; \lambda_1), \ell_b (\cdot; \lambda_2); [a, b)).
\end{equation}
If additionally $\lambda_1, \lambda_2 \notin \sigma_p (\mathcal{L}^{\alpha})$,
then we have equality in (\ref{regular-singular-theorem-inequality}).
\end{theorem}

In the case that {\bf (A)$^\prime$} doesn't hold, so that (\ref{linear-hammy})
is singular at $x = a$, we let $[\lambda_1, \lambda_2]$, $\lambda_1 < \lambda_2$,
be such that $[\lambda_1, \lambda_2] \cap \sigma_{\ess} (\mathcal{L}) = \emptyset$.
We will show in Section \ref{operator-section} 
that for each $\lambda \in [\lambda_1, \lambda_2]$ there exists a
$2n \times n$ matrix solution $\mathbf{X}_a (x; \lambda)$
to the ODE 
\begin{equation} \label{frame-a}
\begin{aligned}
J \mathbf{X}_a' =& (B_0 (x) + \lambda B_1 (x)) \mathbf{X}_a \\
\lim_{x \to a^+} &U^a (x; \lambda_0)^* J \mathbf{X}_a (x; \lambda) = 0,  
\end{aligned}
\end{equation}
where the matrix $U^a (x; \lambda_0)$ is described in 
Lemma \ref{self-adjoint-operator-lemma} (and the paragraph
leading into that lemma). In addition, we will check that for each 
pair $(x, \lambda) \in [a,b) \times [\lambda_1, \lambda_2]$,
$\mathbf{X}_a (x; \lambda)$ is the frame for a Lagrangian subspace
of $\mathbb{C}^{2n}$, which we will denote $\ell_a (x; \lambda)$, 
and that 
$\ell_a \in C((a, b) \times [\lambda_1, \lambda_2], \Lambda (n))$.

In Section \ref{theorems-section}, we will establish the following 
theorem. 

\begin{theorem} \label{singular-theorem}
Let Assumptions {\bf (A)}, {\bf (B)}, and {\bf (C)} hold,
and assume that for some pair $\lambda_1, \lambda_2 \in \mathbb{R}$,
$\lambda_1 < \lambda_2$, we have $\sigma_{\ess} (\mathcal{L}) \cap [\lambda_1, \lambda_2] = \emptyset$.
If $\ell_a (\cdot; \lambda_1)$ and $\ell_b (\cdot; \lambda_2)$
denote the paths of Lagrangian subspaces of $\mathbb{C}^{2n}$ constructed just above, 
and $\mathcal{N} ([\lambda_1, \lambda_2))$
denotes a count of the number of eigenvalues $\mathcal{L}$
has on the interval $[\lambda_1, \lambda_2)$, then 
\begin{equation} \label{singular-theorem-inequality}
  \mathcal{N} ([\lambda_1, \lambda_2))
  \ge \mas (\ell_a (\cdot; \lambda_1), \ell_b (\cdot; \lambda_2); (a, b)).
\end{equation}
If additionally $\lambda_1, \lambda_2 \notin \sigma_p (\mathcal{L})$,
then we have equality in (\ref{singular-theorem-inequality}).
\end{theorem}

In order to relate our results to previous work on renormalized
oscillation theory, we observe that in some cases the Maslov
index can be expressed as a sum of nullities for certain 
evolving matrix Wronskians. To understand this, we first 
specify the following terminology: for two paths of Lagrangian subspaces 
$\ell_1, \ell_2: \mathcal{I} \to \Lambda (n)$, we say 
that the evolution of the pair $\ell_1, \ell_2$
is {\it monotonic} provided all intersections occur in the same
direction. If the intersections all correspond with the 
positive direction, then we can compute 
\begin{equation*}
    \mas (\ell_1, \ell_2; \mathcal{I})
    = \sum_{t \in \mathcal{I}} \dim (\ell_1 (t) \cap \ell_2 (t)).
\end{equation*}
Suppose $\mathbf{X}_1 (t) = {X_1 (t) \choose Y_1 (t)}$ and 
$\mathbf{X}_2 (t) = {X_2 (t) \choose Y_2 (t)}$ respectively 
denote frames for Lagrangian subspaces of $\mathbb{C}^{2n}$,
$\ell_1 (t)$ and $\ell_2 (t)$. Then we can express
this last relation as 
\begin{equation*}
\mas (\ell_1, \ell_2; \mathcal{I}) 
= \sum_{t \in \mathcal{I}} \dim \ker (\mathbf{X}_1 (t)^* J \mathbf{X}_2 (t)).
\end{equation*}
(See Lemma 2.2 of \cite{HS2}.)

In the current setting, the necessary monotonicity follows 
from Claims 4.1 and 4.2 of \cite{HS2} (with $(0, 1)$ replaced
by $(a, b)$). With this observation, we obtain the following 
theorem.

\begin{theorem} \label{nullity-theorem}
Under the assumptions of Theorem \ref{regular-singular-theorem} (without 
the requirement $\lambda_1, \lambda_2 \notin \sigma_p (\mathcal{L}^{\alpha})$), 
we can write 
\begin{equation*}
    \mas (\ell_{\alpha} (\cdot; \lambda_1), \ell_b (\cdot; \lambda_2); [a, b))
    = \sum_{x \in [a, b)} \dim \ker \mathbf{X}_{\alpha} (x; \lambda_1)^* J \mathbf{X}_b (x; \lambda_2),
\end{equation*}
and under the assumptions of Theorem \ref{singular-theorem} (without 
the requirement $\lambda_1, \lambda_2 \notin \sigma_p (\mathcal{L})$), 
we can write 
\begin{equation*}
    \mas (\ell_a (\cdot; \lambda_1), \ell_b (\cdot; \lambda_2); (a, b))
    = \sum_{x \in (a, b)} \dim \ker \mathbf{X}_a (x; \lambda_1)^* J \mathbf{X}_b (x; \lambda_2).
\end{equation*}
\end{theorem}

In the remainder of this section, we briefly review the origins of 
renormalized oscillation theory, placing our result in the broader
context, and we also set out a plan for the paper and summarize
our notational conventions. 
For the first, renormalized oscillation theory was introduced 
in \cite{GST1996} in the context of single Sturm-Liouville
equations, and subsequently it was developed in 
\cite{Teschl1996, Teschl1998} for Jacobi operators and 
Dirac operators. Most recently, Gesztesy and Zinchenko
have extended these early results to the setting of 
(\ref{linear-hammy}) in the limit point case \cite{GZ2017}, 
though with a set-up and approach substantially different from 
the ones employed in the current analysis. See also \cite{Simon2005} 
for an expository discussion. 

In order to understand the motivation behind this approach, 
we can contrast it with standard oscillation theory, 
exemplified by Sturm's oscillation theorem for Sturm-Liouville
operators \cite{Sturm}. As a specific point of comparison, 
we will use a (standard) oscillation result that the authors
have obtained for 
Sturm-Liouville equations on the half-line, $(a, b) = (0, \infty)$,
where $x = 0$ is a regular boundary point
(see \cite{HS3}). If we focus on the case of Dirichlet boundary conditions 
at $x = 0$ (i.e., $\alpha = (I \,\,\, 0)$), then Theorem 1.1 of 
\cite{HS3} asserts (under fairly strong assumptions 
on the coefficient matrices associated with the Sturm-Liouville
operator), that the number of eigenvalues that the Sturm-Liouville
operator has below some $\lambda_* \in \mathbb{R}$ can be expressed as 
\begin{equation} \label{morse-dirichlet}
\mor (\mathcal{L}; \lambda_*) = \sum_{x > 0} \dim \ker X_b (x; \lambda_*),
\end{equation}
where $X_b$ denotes 
the first $n \times n$ coordinate in the frame 
$\mathbf{X}_b$. We see immediately, that the 
number of eigenvalues between $\lambda_1$ and $\lambda_2$
can be computed in this case as 
\begin{equation} \label{dirichlet-difference}
\mathcal{N} ([\lambda_1, \lambda_2))
= \sum_{x > 0} \dim \ker X_b (x; \lambda_2)
- \sum_{x > 0} \dim \ker X_b (x; \lambda_1).
\end{equation}
The difficulty with this approach is twofold. First, 
for conditions other than Dirichlet, the right-hand
side of (\ref{morse-dirichlet}) becomes a count of {\it signed}
intersections between $\ell_b (x; \lambda_*)$ and 
$\ell_{\alpha} (0; \lambda_*)$, and so cannot be expressed 
as a sum of nullities; and second, if the strong coefficient
conditions of \cite{HS3} are dropped, the right-hand side 
of (\ref{morse-dirichlet}) can become infinite, even in the 
Dirichlet case. Consequently, (\ref{dirichlet-difference})
can take the form $\infty - \infty$, even in cases 
for which $\mathcal{N} ([\lambda_1, \lambda_2))$ is finite. 
Indeed, this latter observation seems to have been the 
primary motivation for the approach \cite{GST1996, Simon2005}.
(See Section \ref{applications-section} for a specific 
implementation of our theory in this setting.) 

{\it Plan of the paper}. In Section \ref{operator-section},
we prove Lemma \ref{self-adjoint-operator-lemma}, establishing
the existence and nature of the family of self-adjoint operators
$\mathcal{L}$ and $\mathcal{L}^{\alpha}$ that will be the 
objects of our study. In Section \ref{maslov-section}, we 
provide some background on the Maslov index, along with 
some results we'll need for the subsequent analysis. In 
Section \ref{theorems-section}, we prove Theorems 
\ref{regular-singular-theorem} and \ref{singular-theorem},
and in Section \ref{applications-section}
we conclude with two specific illustrative applications.

{\it Notational conventions}. Throughout the analysis, we 
will use the notation $\|\cdot\|_{B_1}$ and 
$\langle \cdot, \cdot \rangle_{B_1}$ respectively for our 
weighted norm and inner product. In the case that 
(\ref{linear-hammy}) is regular at $x = a$, we will denote 
the associated map of Lagrangian subspaces by $\ell_{\alpha}$,
and we will denote by $\mathbf{X}_{\alpha}$ a specific
corresponding map of frames. Likewise, if (\ref{linear-hammy})
is singular at $x = a$, we will use $\ell_a$ and $\mathbf{X}_a$,
and for $x = b$ (always assumed singular), we will use 
$\ell_b$ and $\mathbf{X}_b$. In order to accommodate limits
associated with our bilinear form, we will adopt the notation 
\begin{equation*}
    (Jy,z)_a := \lim_{x \to a^+} (J y(x), z(x)); 
    \quad (Jy,z)_b := \lim_{x \to b^-} (J y(x), z(x)),
\end{equation*}
along with 
\begin{equation*}
 (Jy, z)_a^b := (Jy,z)_b - (Jy,z)_a.   
\end{equation*}
Here and throughout, we use 
$(\cdot, \cdot)$ to denote the usual inner product 
in $\mathbb{C}^{2n}$.

\section{The Self-Adjoint Operators $\mathcal{L}$ and $\mathcal{L}^{\alpha}$} 
\label{operator-section}

In this section, we adapt the approach of \cite{Niessen70, Niessen71, Niessen72}
(as developed in Chapter VI of \cite{Krall2002}) to the setting 
of (\ref{linear-hammy}). 
\subsection{Niessen Spaces} 
\label{niessen-section} 

We begin by fixing some $c \in (a, b)$, and letting 
$\Phi (x; \lambda)$ denote the fundamental matrix specified 
by
\begin{equation} \label{phi-specified}
J \Phi' = (B_0 (x) + \lambda B_1 (x)) \Phi; \quad \Phi (c; \lambda) = I_{2n}.
\end{equation}
We define
\begin{equation*}
\mathcal{A} (x;\lambda) := \frac{1}{2 \rm{Im }\lambda} 
\Phi (x; \lambda)^* (J/i) \Phi (x; \lambda),
\end{equation*}
on $(a,b) \times \mathbb{C} \backslash \mathbb{R}$. It's clear from 
this definition that for each $\lambda \in \mathbb{C} \backslash \mathbb{R}$, we have
$\mathcal{A} (\cdot; \lambda) \in \AC_{\loc} ((a, b), \mathbb{C}^{2n \times 2n})$,  
with $\mathcal{A} (x;\lambda)$ self-adjoint 
for all $(x, \lambda) \in (a,b) \times \mathbb{C} \backslash \mathbb{R}$.
It follows that the eigenvalues $\{\mu_j (x; \lambda)\}_{j=1}^{2n}$ 
of $\mathcal{A} (x;\lambda)$ can be 
ordered so that $\mu_j (x; \lambda) \le \mu_{j+1} (x; \lambda)$
for all $j \in \{1, 2, \dots, 2n-1\}$.

Since $\mathcal{A} (c; \lambda) = \frac{1}{2 \rm{Im }\lambda} (J/i)$,
we see that $\mathcal{A} (c; \lambda)$ has an eigenvalue with multiplicity
$n$ at $- \frac{1}{2 |\rm{Im }\lambda|}$ and an eigenvalue with 
multiplicity $n$ at $+ \frac{1}{2 |\rm{Im }\lambda|}$. According to 
Theorem II.5.4 in \cite{Kato}, we can understand the motion of 
the eigenvalues $\{\mu_j (x; \lambda)\}_{j=1}^{2n}$
as $x$ increases by evaluating the matrix $\mathcal{A}' (x;\lambda)$,
where prime denotes differentiation with respect to $x$. To this 
end, we find by direct calculation that 
\begin{equation} \label{A-prime}
\mathcal{A}' (x;\lambda)
= \Phi (x; \lambda)^* B_1 (x) \Phi (x; \lambda) 
\end{equation}  
for all $(x, \lambda) \in (a,b) \times \mathbb{C} \backslash \mathbb{R}$.
We can conclude from Assumption {\bf (B)} that each eigenvalue
$\mu_j (x; \lambda)$ must be continuous and non-decreasing as a function of 
$x$. In addition, since the fundamental matrix $\Phi (x; \lambda)$
is invertible for all $(x, \lambda) \in (a, b) \times \mathbb{C} \backslash \mathbb{R}$,
we see that $\mathcal{A} (x; \lambda)$ is likewise invertible, and 
so none of its eigenvalues can cross 0 for any $x \in (a, b)$.
We conclude that for all $(x, \lambda) \in (a,b) \times \mathbb{C} \backslash \mathbb{R}$,
we have the ordering 
\begin{equation} \label{eig-A-order}
\mu_1 (x; \lambda) \le \mu_2 (x; \lambda) \le \dots \le \mu_n (x; \lambda)
< 0 < \mu_{n+1} (x; \lambda) \le \mu_{n+2} (x; \lambda) \le ... \le \mu_{2n} (x; \lambda).
\end{equation}

As $x$ decreases toward $x = a$, these eigenvalues are all non-increasing,
and so in particular the limits 
\begin{equation*}
\mu_j^a (\lambda) := \lim_{x \to a^+} \mu_j (x; \lambda)
\end{equation*}
exist for each $j \in \{n+1, n+2, \dots, 2n\}$. Moreover, for each 
$j \in \{1, 2, \dots, n\}$, these same limits either exist or diverge
to $- \infty$. Likewise, as $x$ increases toward $x = b$, the 
eigenvalues $\{\mu_j (x; \lambda)\}_{j=1}^{2n}$ are all non-decreasing,
and so in particular the limits 
\begin{equation*}
\mu_j^b (\lambda) := \lim_{x \to b^-} \mu_j (x; \lambda)
\end{equation*}
exist for each $j \in \{1, 2, \dots, n\}$. Moreover, for each 
$j \in \{n+1, n+2, \dots, 2n\}$, these same limits either exist or diverge
to $+ \infty$. 

\begin{lemma} \label{subspace-dimensions-lemma}
Let Assumptions {\bf (A)} and {\bf (B)} hold,
and let $\lambda \in \mathbb{C} \backslash \mathbb{R}$ be fixed.
Then the dimension $m_a (\lambda)$ of the subspace of solutions to 
(\ref{linear-hammy}) that lie left in $(a, b)$ is precisely 
the number of eigenvalues $\mu_j (x; \lambda) \in \sigma (\mathcal{A} (x; \lambda))$
that approach a finite limit as $x \to a^+$. Likewise, 
the dimension $m_b (\lambda)$ of the subspace of solutions to 
(\ref{linear-hammy}) that lie right in $(a, b)$ is precisely 
the number of eigenvalues $\mu_j (x; \lambda) \in \sigma (\mathcal{A} (x; \lambda))$
that approach a finite limit as $x \to b^-$.
\end{lemma}

\begin{proof} 
We will carry out the proof for $m_b (\lambda)$; the proof for 
$m_a (\lambda)$ is similar. Integrating (\ref{A-prime}), 
we see that $\mathcal{A} (x;\lambda)$ can alternatively be 
expressed as 
\begin{equation} \label{A-alternative}
\mathcal{A} (x; \lambda) = \frac{1}{2 \rm{Im }\lambda} (J/i) 
+ \int_c^x \Phi (\xi; \lambda)^* B_1 (\xi) \Phi (\xi; \lambda) d\xi. 
\end{equation} 
We temporarily let $\tilde{m}_b (\lambda)$ denote the 
number of eigenvalues of $\mathcal{A} (x; \lambda)$ that have a finite limit as 
$x \to b^-$; precisely, this will be the set 
$\{\mu_j (x; \lambda)\}_{j=1}^{\tilde{m}_b (\lambda)}$.
Let $\{v_j (x; \lambda)\}_{j=1}^{\tilde{m}_b (\lambda)}$ denote an orthonormal 
basis of eigenvectors associated with these eigenvalues, noting that
these elements may not be continuous in $x$. We can 
take any element $v_j (x; \lambda)$ from this collection and multiply 
(\ref{A-alternative}) on the left by $v_j(x; \lambda)^*$ and on 
the right by $v_j(x; \lambda)$ to obtain 
\begin{equation} \label{ensures-boundedness} 
v_j (x; \lambda)^* \{\mathcal{A} (x; \lambda) -  \frac{1}{2 \rm{Im }\lambda} (J/i) \} v_j (x; \lambda)
= \int_c^x v_j(x; \lambda)^* \Phi (\xi; \lambda)^* B_1 (\xi) \Phi (\xi; \lambda) v_j (x; \lambda) d\xi.
\end{equation}
The left-hand side of this last relation is 
\begin{equation*}
   \mu_j (x; \lambda) -  \frac{1}{2i \rm{Im } \lambda} v_j (x; \lambda)^* J v_j (x; \lambda),
\end{equation*}
and so is bounded above for all $x \in (c, b)$. Now, consider
any sequence of values $\{x_k\}_{k=1}^{\infty}$ so that 
$x_k$ increases to $b$ as $k \to \infty$. The corresponding sequence 
$\{v_j (x_k; \lambda)\}_{k=1}^{\infty}$ lies on the unit sphere
in $\mathbb{C}^{2n}$ (a compact set), so there exists a 
subsequence $\{x_{k_i}\}_{i=1}^{\infty}$ so that 
$\{v_j (x_{k_i}; \lambda)\}_{i=1}^{\infty}$ converges to 
some $v^b_j (\lambda)$ on the unit sphere
in $\mathbb{C}^{2n}$. We claim that it follows 
that the functions $\{\Phi (x; \lambda) v^b_j (\lambda)\}_{j=1}^{\tilde{m}_b (\lambda)}$
lie right in $(a, b)$. To see this, we assume to the contrary 
that for some $j \in \{1, 2, \dots, \tilde{m}_b (\lambda)\}$,
\begin{equation*}
    \int_c^b v_j^b(\lambda)^* \Phi (\xi; \lambda)^* B_1 (\xi) \Phi (\xi; \lambda) v_j^b (\lambda) d\xi
    = \infty.
\end{equation*}
In this case, if we are given any constant $K > 0$, we can take $b' \in (c, b)$
sufficiently close to $b$ (sufficiently large if $b = \infty$) so that 
\begin{equation} \label{temp-int}
    \int_c^{b'} v_j^b(\lambda)^* \Phi (\xi; \lambda)^* B_1 (\xi) \Phi (\xi; \lambda) v_j^b (\lambda) d\xi
    > K.
\end{equation}
By a straightforward calculation, we can check that by taking 
$x_{k_i}$ sufficiently close to $b$ (sufficiently large if $b = \infty$),
we can make 
\begin{equation*}
    \int_c^{b'} v_j(x_{k_i}; \lambda)^* \Phi (\xi; \lambda)^* B_1 (\xi) \Phi (\xi; \lambda) v_j (x_{k_i}; \lambda) d\xi
\end{equation*}
as close as we like to the integral in (\ref{temp-int}). In particular, 
we can find a positive integer $N$ sufficiently large so that 
for all $i \ge N$, we have 
\begin{equation*}
    \int_c^{b'} v_j(x_{k_i}; \lambda)^* \Phi (\xi; \lambda)^* B_1 (\xi) \Phi (\xi; \lambda) v_j (x_{k_i}; \lambda) d\xi
    \ge K.
\end{equation*}
Possibly by taking $N$ even larger, we can ensure that 
$x_{k_i} > b'$, and it follows from our Assumption {\bf (B)}
that 
\begin{equation*}
    \begin{aligned}
    \int_c^{x_{k_i}} &v_j(x_{k_i}; \lambda)^* \Phi (\xi; \lambda)^* B_1 (\xi) \Phi (\xi; \lambda) v_j (x_{k_i}; \lambda) d\xi \\
    &> \int_c^{b'} v_j(x_{k_i}; \lambda)^* \Phi (\xi; \lambda)^* B_1 (\xi) \Phi (\xi; \lambda) v_j (x_{k_i}; \lambda) d\xi
    \ge K,
    \end{aligned}
\end{equation*}
Since $K$ can be taken as large as we like, this contradicts the 
boundedness ensured by (\ref{ensures-boundedness}). 

The set $\{v^b_j (\lambda)\}_{j=1}^{\tilde{m}_b (\lambda)}$
retains orthonormality in the limit, ensuring that the 
functions $\{\Phi (x; \lambda) v^b_j (\lambda)\}_{j=1}^{\tilde{m}_b (\lambda)}$
are linearly independent as solutions of (\ref{linear-hammy}). We 
conclude that this set comprises a basis for the $\tilde{m}_b (\lambda)$-dimensional 
subspace of solutions to (\ref{linear-hammy}) that lie right in $(a, b)$.  
In particular, we see that $\tilde{m}_b (\lambda) = m_b (\lambda)$.

If we allow $\{v_j (x; \lambda)\}_{j=m_b (\lambda)+1}^{2n}$ to denote 
an orthonormal basis of eigenvectors associated with the eigenvalues 
of $\mathcal{A} (x; \lambda)$ that do not have finite limits as 
$x \to b^-$, then we find that the functions 
$\{\Phi (x; \lambda) v^b_j (\lambda)\}_{m_b (\lambda)+1}^{2n}$ form 
a basis for a $(2n-m_b (\lambda))$-dimensional subspace of solutions
of (\ref{linear-hammy}) that do not lie right in $(a, b)$.
\end{proof}

Lemma \ref{subspace-dimensions-lemma} suggests that we need to 
better understand the nature of the eigenvalues of $\mathcal{A} (x; \lambda)$.
As a starting point, we observe the relation 
\begin{equation} \label{conj-no-conj}
\Phi(x; \bar{\lambda})^* (J/i) \Phi (x; \lambda) = (J/i), 
\end{equation}
for all $x \in (a, b)$, which can be verified by showing
that the quantity on the left is independent of $x$
(its derivative is zero) and evaluating at $x = c$, 
where $\Phi (c; \lambda) = I_{2n}$. (Although we are currently
working with the case $\textrm{Im }\lambda \ne 0$, 
(\ref{conj-no-conj}) holds for $\lambda \in \mathbb{R}$
as well.) Since $(J/i)$ is 
self-adjoint, we likewise have (by taking an adjoint
on both sides of (\ref{conj-no-conj}))
\begin{equation} \label{no-conj-conj}
\Phi(x; \lambda)^* (J/i) \Phi (x; \bar{\lambda}) = (J/i), 
\end{equation}
and this relation allows us to write
\begin{equation*}
\Phi (x; \bar{\lambda}) = (J/i) (\Phi(x; \lambda)^*)^{-1} (J/i).
\end{equation*}
In this way, we see that we can write 
\begin{equation*}
\begin{aligned}
\mathcal{A} (x; \bar{\lambda}) &= 
- \frac{1}{2 \rm{Im }\lambda} 
\Phi (x; \bar{\lambda})^* (J/i) \Phi (x; \bar{\lambda}) \\
&= - \frac{1}{2 \rm{Im }\lambda} 
(J/i) (\Phi(x; \lambda))^{-1} (J/i) (J/i) (J/i)
(\Phi(x; \lambda)^*)^{-1} (J/i) \\
&= - \frac{1}{(2 \rm{Im }\lambda)^2} 
(J/i) \mathcal{A} (x; \lambda)^{-1} (J/i). 
\end{aligned}
\end{equation*}
Upon subtracting a term $\rho I$ from both sides of 
this last relation (for any $\rho \in \mathbb{R}$), 
we obtain the relation 
\begin{equation} \label{eig-relation}
\mathcal{A} (x; \bar{\lambda}) - \rho I
= - \rho (J/i) \mathcal{A} (x; \lambda)^{-1} 
\{\mathcal{A} (x; \lambda) + \frac{1}{\rho (2 \rm{Im }\lambda)^2} I \} (J/i).
\end{equation}

These considerations allow us to conclude the following lemma, 
adapted from Theorem VI.2.1 of \cite{Krall2002}.

\begin{lemma} \label{cal-A-eigs}
Let Assumption {\bf (A)} hold (not necessarily Assumption {\bf (B)}). 
A value $\rho \in \mathbb{R}$ is an eigenvalue of 
$\mathcal{A} (x; \bar{\lambda})$ if and only if 
the value $-\frac{1}{\rho (2 \rm{Im }\lambda)^2}$
is an eigenvalue of $\mathcal{A} (x; \lambda)$. It 
follows immediately that if we order the eigenvalues
of $\mathcal{A} (x; \lambda)$ according to (\ref{eig-A-order}),
and order the eigenvalues of $\mathcal{A} (x; \bar{\lambda})$
similarly, then we have 
\begin{equation*}
\begin{aligned}
\mu_j (x; \bar{\lambda}) &= - \frac{1}{(2 {\rm Im }\lambda)^2 \mu_{n+j} (x; \lambda)};
\quad j = 1, 2, \dots, n; \\
\mu_j (x; \bar{\lambda}) &= - \frac{1}{(2 {\rm Im }\lambda)^2 \mu_{j-n} (x; \lambda)};
\quad j = n+1, n+2, \dots, 2n. 
\end{aligned}
\end{equation*}
Moreover, for $j = 1, 2, \dots, n$, if $v_j (x; \bar{\lambda})$ is an 
eigenvector of $\mathcal{A} (x; \bar{\lambda})$ associated with 
eigenvalue $\mu_j (x; \bar{\lambda})$, then 
\begin{equation*}
v_{n+j} (x; \lambda) = (J/i) v_j (x; \bar{\lambda})
\end{equation*} 
is an eigenvector of $\mathcal{A} (x; \lambda)$ associated with 
eigenvalue $\mu_{n+j} (x; \lambda)$. Likewise, for
$j = n+1, n+2, \dots, 2n$, if $v_j (x; \bar{\lambda})$ is an 
eigenvector of $\mathcal{A} (x; \bar{\lambda})$ associated with 
eigenvalue $\mu_j (x; \bar{\lambda})$, then 
\begin{equation*}
v_{j-n} (x; \lambda) = (J/i) v_j (x; \bar{\lambda})
\end{equation*} 
is an eigenvector of $\mathcal{A} (x; \lambda)$ associated with 
eigenvalue $\mu_{j-n} (x; \lambda)$. 
\end{lemma}

Similarly as in the proof of Lemma \ref{subspace-dimensions-lemma}, 
we can use compactness of the unit sphere in $\mathbb{C}^{2n}$
to associate limiting vectors $\{v_j^b (\lambda)\}_{i=1}^{2n}$
and $\{v_j^b (\bar{\lambda})\}_{i=1}^{2n}$ respectively 
with the eigenvectors $\{v_j (x; \lambda)\}_{i=1}^{2n}$
and $\{v_j (x; \bar{\lambda})\}_{i=1}^{2n}$. These limiting
vectors naturally inherit both orthonormality and 
the relations of Lemma \ref{cal-A-eigs},
\begin{equation} \label{v-bar-relation}
\begin{aligned}
v_{n+j}^b (\lambda) &= (J/i) v_j^b (\bar{\lambda}); \quad j=1,2,\dots,n \\
v_{j-n}^b (\lambda) &= (J/i) v_j^b (\bar{\lambda}); \quad j=n+1, n+2, \dots, 2n,
\end{aligned}
\end{equation} 
with precisely the same statements holding for the limit $x \to a^+$ with the 
superscript $b$ replaced by the superscript $a$. 

We note for later use that for any indices
$j \in \{1, 2, \dots, n\}$,
$k \in \{1, 2, \dots, 2n\}$,
we can use (\ref{v-bar-relation}) to see that 
\begin{equation} \label{niessen-relations1}
\begin{aligned}
v_j^b (\bar{\lambda})^* J v_k^b (\lambda) 
&= ((J/i) v_{n+j}^b (\lambda))^* J v_k^b (\lambda) 
= v_{n+j}^b (\lambda)^* (J/i) J v_k^b (\lambda) \\
&= i v_{n+j}^b (\lambda)^* v_k^b (\lambda)  
= i \delta_{n+j}^k,
\end{aligned}
\end{equation}
where $\delta_{n+j}^k$ is a Kroenecker delta function, 
and the final equivalence is due to orthonormality.
Likewise, for any indices
$j \in \{n+1, n+2, \dots, 2n\}$,
$k \in \{1, 2, \dots, 2n\}$, we see from (\ref{v-bar-relation}) that 
\begin{equation} \label{niessen-relations2}
\begin{aligned}
v_j^b (\bar{\lambda})^* J v_k^b (\lambda) 
&= ((J/i) v_{n+j}^b (\lambda))^* J v_k^b (\lambda) 
= v_{j-n}^b (\lambda)^* (J/i) J v_k^b (\lambda) \\
&= i v_{j-n}^b (\lambda)^* v_k^b (\lambda)  
= i \delta_{j-n}^k.
\end{aligned}
\end{equation}

For $j = 1, 2, \dots, n$, we set 
\begin{equation} \label{niessen-components}
\begin{aligned}
y_j^b (x; \lambda) &= \Phi (x; \lambda) v_j^b (\lambda) \\
z_j^b (x; \lambda) &= \Phi (x; \lambda) v_{n+j}^b (\lambda). 
\end{aligned}
\end{equation} 
It's clear from our construction that 
$y_j^b (\cdot; \lambda)$ lies right in $(a, b)$
for each $j \in \{1, 2, \dots, n\}$, while 
$z_j^b (\cdot; \lambda)$ lies right in $(a, b)$
if and only if $\mu_{n+j}^b (\lambda)$ is finite. 
We have seen that the total number of the values 
$\{\mu_j^b (\lambda)\}_{j=1}^{2n}$ that are finite
is $m_b (\lambda)$, and we will also find it convenient
to introduce the value $r_b (\lambda) := m_b (\lambda) - n$. 
Following 
\cite{Niessen70, Niessen71, Niessen72}, for each $j \in \{1, 2, \dots, n\}$,
we define the two-dimensional space
\begin{equation} \label{niessen-spaces}
N_j^b (\lambda) := \Span \{y_j^b (\cdot; \lambda), z_j^b (\cdot; \lambda)\},
\end{equation}
and following \cite{Krall2002} we refer to the collection 
$\{N_j^b (\lambda)\}_{j=1}^n$ as the {\it Niessen subspaces} at $b$.
According to our labeling convention, the Niessen spaces
$\{N_j^b (\lambda)\}_{j=1}^{r_b (\lambda)}$ all satisfy 
$\dim N_j^b (\lambda) \cap L^2_{B_1} ((c, b), \mathbb{C}^{2n}) = 2$,
while the remaining Niessen spaces $\{N_j^b (\lambda)\}_{r_b (\lambda) + 1}^n$ 
satisfy $\dim N_j^b (\lambda) \cap L^2_{B_1} ((c, b), \mathbb{C}^{2n}) = 1$. 
(Here, $c$ continues to be any value $c \in (a, b)$.)

We see from Lemma \ref{cal-A-eigs} that as $x$ increases to $b$, we will 
have $\mu_j (x; \bar{\lambda}) \to + \infty$ if and only if 
$\mu_{j-n} (x; \lambda) \to 0$. In this way, the values 
$m_b (\lambda)$ and $m_b (\bar{\lambda})$ are both determined
by the eigenvalues of $\mathcal{A} (x; \lambda)$ as 
$x \to b^-$. A similar statement holds at $x = a$. We emphasize, however,
that the values $m_b (\lambda)$ and $m_b (\bar{\lambda})$ do 
not necessarily agree. This is precisely why we need 
our consistency Assumption {\bf (C)}. As noted in the 
Introduction, under Assumption {\bf (C)} we will denote 
the mutual value of $m_b (\lambda)$ and $m_b (\bar{\lambda})$ by $m_b$,
and we will also denote the mutual value of 
$r_b (\lambda)$ and $r_b (\bar{\lambda})$ by $r_b$. 

\begin{remark} \label{assumption-c-remark}
We note that if the matrices
$B_0 (x)$ and $B_1 (x)$ have real entries so that 
$\overline{B_0 (x) + \lambda B_1 (x)} = B_0 (x) + \bar{\lambda} B_1 (x)$, 
then we will have $\overline{\Phi (x; \lambda)} = \Phi (x; \bar{\lambda})$,
and correspondingly $\overline{\mathcal{A} (x; \lambda)} = \mathcal{A} (x; \bar{\lambda})$.
In this case, for each $j \in \{1, 2, \dots, 2n\}$, 
\begin{equation}
\mu_j (x; \lambda) = \overline{\mu_j (x; \lambda)} = \mu_j (x; \bar{\lambda}).
\end{equation}
In particular, $m_a (\lambda) = m_a (\bar{\lambda})$ and 
$m_b (\lambda) = m_b (\bar{\lambda})$, and so our 
Assumption {\bf (C)} will hold.
\end{remark}  

In the next part of our development, the ratios 
$\{\mu_j (x; \lambda)/\mu_{n+j} (x; \lambda)\}_{j=1}^n$
will have an important role, and we emphasize that 
Assumption {\bf (C)} becomes crucial at this point. 
To see this, we first observe from Lemma \ref{cal-A-eigs}
the relation
\begin{equation} \label{ratios-with-x}
    \frac{\mu_j (x; \bar{\lambda})}{\mu_{n+j} (x; \bar{\lambda})}
    = - \frac{\frac{1}{(2 \mathrm{Im }\lambda)^2 \mu_{n+j} (x; \lambda)}}
    {\frac{1}{(2 \mathrm{Im }\lambda)^2 \mu_{j} (x; \lambda)}}
    = \frac{\mu_j (x; \lambda)}{\mu_{n+j} (x; \lambda)}.
\end{equation}
For $j = r_b (\lambda) + 1, \dots, n$, we have 
\begin{equation*}
    \lim_{x \to b^-} \mu_{n+j} (x; \lambda) = \infty;
    \quad \implies \quad
    \lim_{x \to b^-} \mu_j (x; \bar{\lambda}) = 0,
\end{equation*}
and so both sides of (\ref{ratios-with-x}) approach 0
as $x \to b^-$. On the other hand, for $j = 1, \dots, r_b (\lambda)$,
we have 
\begin{equation*}
    \lim_{x \to b^-} \mu_{n+j} (x; \lambda) = \mu_{n+j}^b (\lambda);
    \quad \implies \quad
    \lim_{x \to b^-} \mu_j (x; \bar{\lambda}) = \mu_j^b (\bar{\lambda}),
\end{equation*}
where the values $\mu_{n+j}^b (\lambda)$ and $\mu_j^b (\bar{\lambda})$ 
are both non-zero real numbers, and so do not fully determine the limits of 
(\ref{ratios-with-x}) as $x \to b^-$. In particular, in order to determine
these limits, we require either the limit of $\mu_{n+j} (x; \bar{\lambda})$
or the limit of $\mu_j (x; b)$ as $x \to b^-$.  Precisely the same statements
hold with $\lambda$ replaced by $\bar{\lambda}$, so for  
$j = 1, \dots, r_b (\bar{\lambda})$,
we have 
\begin{equation*}
    \lim_{x \to b^-} \mu_{n+j} (x; \bar{\lambda}) = \mu_{n+j}^b (\bar{\lambda});
    \quad \implies \quad
    \lim_{x \to b^-} \mu_j (x; \lambda) = \mu_j^b (\lambda),
\end{equation*}
where the values $\mu_{n+j}^b (\bar{\lambda})$ and $\mu_j^b (\lambda)$ 
are both non-zero real numbers. We can conclude that if $r_b (\lambda) = r_b (\bar{\lambda})$,
then the ratios 
$\{\mu_j (x; \lambda)/\mu_{n+j} (x; \lambda)\}_{j=1}^{r_b (\lambda)}$ 
will all have real non-zero limits as $x \to b^-$.  

Working now under Assumption {\bf (C)},  we choose $n$ 
solutions of (\ref{linear-hammy}) that lie right in 
$(a, b)$, taking precisely 
one from each Niessen subspace $N_j^b (\lambda)$ in the following way. 
First, for each $j \in \{1, 2, \dots, r_b\}$,
we let $\beta_j (\lambda)$ be any complex
number on the circle 
\begin{equation*}
    |\beta_j^b (\lambda)| = \sqrt{-\mu_j^b (\lambda)/\mu_{n+j}^b (\lambda)},
\end{equation*}
where as described just above, these ratios cannot be 0, and we set 
\begin{equation*}
    u^b_j (x; \lambda) = y^b_j (x; \lambda) + \beta_j^b (\lambda) z^b_j (x; \lambda). 
\end{equation*}
Next, for each $j \in \{r_b + 1, r_b + 2, \dots, n\}$,
we set 
\begin{equation*}
    u^b_j (x; \lambda) = y^b_j (x; \lambda). 
\end{equation*}
Correspondingly, we will denote by $\{r_j^b (\lambda)\}_{j=1}^n$ 
the vectors specified so that $u^b_j (x; \lambda) = \Phi (x; \lambda) r_j^b (\lambda)$
for each $j \in \{1, 2, \dots, n\}$. Precisely, this means that 
\begin{equation*}
    \begin{aligned}
    r^b_j (\lambda) &= v_j^b (\lambda) + \beta^b_j (\lambda) v_{n+j}^b (\lambda),
    \quad j \in \{1, 2, \dots, r_b\}, \\
    r^b_j (\lambda) &= v_j^b (\lambda), 
    \quad \quad \quad \quad \quad \quad  j \in \{r_b + 1, r_b + 2, \dots, n\}. 
    \end{aligned}
\end{equation*}
We can now collect the vectors $\{r_j^b (\lambda)\}_{j=1}^n$ 
into a frame 
\begin{equation} \label{b-frame}
\mathbf{R}^b (\lambda) =
\begin{pmatrix}
r_1^b (\lambda) & r_2^b (\lambda) & \dots & r_n^b (\lambda)
\end{pmatrix}.
\end{equation}

In addition to the above specifications, for the Niessen spaces 
$\{N_j^b (\lambda)\}_{j=1}^{r_b}$, it will be useful to introduce
notation for elements linearly independent to the 
$\{u_j^b (\lambda)\}_{j = 1}^{r_b}$. For each 
$j \in \{1, 2, \dots, r_b\}$, we take any 
complex number $\gamma_j (\lambda)$ so that 
$|\gamma_j (\lambda)| = |\beta_j (\lambda)|$ 
but  $\gamma_j (\lambda) \ne \beta_j (\lambda)$,
and we define the Niessen complement to $u_j^b (x; \lambda)$
to be 
\begin{equation} \label{niessen-complements}
v_j^b (x; \lambda) 
= y^b_j (x; \lambda) + \gamma_j^b (\lambda) z^b_j (x; \lambda). 
\end{equation}

With this notation in place, we can adapt Theorem VI.3.1
from \cite{Krall2002} to the current setting.

\begin{lemma} \label{krall-niessen-lemma} Let 
Assumptions {\bf (A)}, {\bf (B)} and {\bf (C)} hold, and 
let the Niessen elements $\{u_j^b (x; \lambda)\}_{j=1}^{n}$
and the Niessen complements $\{v_j^b (x; \lambda)\}_{j=1}^{r_b}$
be specified as above. Then the following hold: 

\medskip
(i) For each $j, k \in \{1, 2, \dots, n\}$, 
\begin{equation*}
    (J u_j^b (\cdot; \lambda), u_k^b (\cdot; \lambda))_b = 0.
\end{equation*}

\medskip
(ii) For each $j, k \in (1, 2, \dots, r_b)$, 
\begin{equation*}
    (J u_j^b (\cdot; \lambda), v_k^b (\cdot; \lambda))_b 
    = \begin{cases}
    0 & j \ne k \\
    \kappa_j^b = 2 i \mathrm{Im }\lambda (\mu_j^b (\lambda) 
    + \gamma^b_j (\lambda) \beta^b_j (\lambda) \mu_{n+j}^b (\lambda)) \ne 0 & j = k. 
    \end{cases}
\end{equation*}
\end{lemma}

\begin{proof}
See Theorem VI.3.1 in \cite{Krall2002}. We note here only two key points: (1) 
We require Assumption {\bf (C)} in order to ensure that $\kappa_j^b \ne 0$;
and (2) in anticipation of Lemma \ref{green-lemma}, we are introducing the notation 
\begin{equation*}
    (Ju, v)_b := \lim_{x \to b^-} (J u(x), v(x)).
\end{equation*}
\end{proof}

\begin{claim} \label{niessen-claim1}
Let Assumptions {\bf (A)}, {\bf (B)}, and {\bf (C)} hold, and suppose
the Niessen elements for (\ref{linear-hammy}) are chosen to be
\begin{equation*}
    \begin{aligned}
    u_j^b (x; \lambda) &= \Phi (x; \lambda) (v_j^b (\lambda) + \beta_j^b (\lambda) v^b_{n+j} (\lambda)),
    \quad j \in \{1, 2, \dots, r_b\} \\
    v_j^b (x; \lambda) &= \Phi (x; \lambda) (v_j^b (\lambda) + \gamma_j^b (\lambda) v^b_{n+j} (\lambda)),
    \quad j \in \{1, 2, \dots, r_b\} \\
    u_j^b (x; \lambda) &= \Phi (x; \lambda) v_j^b (\lambda), 
    \quad \quad \quad \quad \quad \quad j \in \{r_b + 1, r_b + 2, \dots, n\},
    \end{aligned}
\end{equation*}
with $\beta_j^b (\lambda)$ and $\gamma_j^b (\lambda)$ specified just above (in particular, 
as real non-zero values).
Then the Niessen elements for (\ref{linear-hammy}) with $\lambda$ 
replaced by $\bar{\lambda}$ can be chosen to be
\begin{equation*}
    \begin{aligned}
    u_j^b (x; \bar{\lambda}) &= \Phi (x; \bar{\lambda}) (v_j^b (\bar{\lambda}) + \beta_j^b (\bar{\lambda}) v^b_{n+j} (\bar{\lambda})),
    \quad j \in \{1, 2, \dots, r_b\} \\
    v_j^b (x; \bar{\lambda}) &= \Phi (x; \bar{\lambda}) (v_j^b (\bar{\lambda}) + \gamma_j^b (\bar{\lambda}) v^b_{n+j} (\bar{\lambda})),
    \quad j \in \{1, 2, \dots, r_b\} \\
    u_j^b (x; \bar{\lambda}) &= \Phi (x; \bar{\lambda}) v_j^b (\bar{\lambda}), 
    \quad \quad \quad \quad \quad \quad j \in \{r_b + 1, r_b + 2, \dots, n\},
    \end{aligned}
\end{equation*}
with $\beta_j^b (\bar{\lambda}) =  - \overline{\beta_j^b (\lambda)}$
and $\gamma_j^b (\bar{\lambda}) = - \overline{\gamma_j^b (\lambda)}$ 
for all $j \in \{1, 2, \dots r_b\}$.
\end{claim}

\begin{proof}
This statement follows almost entirely from our labeling conventions, 
and the only part that we will explicitly check is the final 
assertion that we can take $\beta_j^b (\bar{\lambda}) = - \overline{\beta_j^b (\lambda)}$
and $\gamma_j^b (\bar{\lambda}) = - \overline{\gamma_j^b (\lambda)}$. For this, 
we observe from (\ref{ratios-with-x}) that 
\begin{equation*}
    \frac{\mu_j^b (\bar{\lambda})}{\mu_{n+j}^b (\bar{\lambda})}
    = - \frac{\frac{1}{(2 \mathrm{Im }\lambda)^2 \mu_{n+j}^b (\lambda)}}
    {\frac{1}{(2 \mathrm{Im }\lambda)^2 \mu_{j}^b (\lambda)}}
    = \frac{\mu_j^b (\lambda)}{\mu_{n+j}^b (\lambda)},
\end{equation*}
and consequently
\begin{equation*}
    |\beta_j^b (\bar{\lambda})|
    = \sqrt{-\mu_j^b (\bar{\lambda})/\mu_{n+j}^b (\bar{\lambda})}
    = |\beta_j^b (\lambda)|.
\end{equation*}
Since we can take $\beta_j^b (\bar{\lambda})$ to be any complex 
number with this norm, we can set $\beta_j^b (\bar{\lambda}) = - \overline{\beta_j^b (\lambda)}$,
and subsequently we're justified in choosing 
$\gamma_j^b (\bar{\lambda}) = - \overline{\gamma_j^b (\lambda)}$.
\end{proof}

\begin{claim} \label{niessen-claim2}
Let the Assumptions and notation of Claim \ref{niessen-claim1} hold,
and let $\mathbf{R}^b (\lambda)$ denote the matrix defined in (\ref{b-frame}).
If $\mathbf{R}^b (\bar{\lambda})$ denotes the matrix defined in 
(\ref{b-frame}) with $\lambda$ replaced by $\bar{\lambda}$
and the Niessen elements described in Claim \ref{niessen-claim1},
then 
\begin{equation*}
    \mathbf{R}^b (\bar{\lambda})^* J \mathbf{R}^b (\lambda) = 0.
\end{equation*}
\end{claim}

\begin{proof}
First, for $j, k \in \{1, 2, \dots, r_b (\lambda)\}$, we have 
\begin{equation*}
    \begin{aligned}
    r^b_j (\bar{\lambda})^* J r^b_k (\lambda) 
    &= (v_j^b (\bar{\lambda})^* + \overline{\beta_j^b (\bar{\lambda})} v^b_{n+j} (\bar{\lambda})^*)
    J (v_k^b (\lambda) + \beta_k^b (\lambda) v^b_{n+k} (\lambda)) \\
    &= v_j^b (\bar{\lambda})^* J v_k^b (\lambda) 
    + \beta_k^b (\lambda) v_j^b (\bar{\lambda})^* J v^b_{n+k} (\lambda) \\
    &+ \overline{\beta_j^b (\bar{\lambda})} v^b_{n+j} (\bar{\lambda})^* J v_k^b (\lambda)
    + \overline{\beta_j^b (\bar{\lambda})} \beta_k^b (\lambda)  v^b_{n+j} (\bar{\lambda})^* v^b_{n+k} (\lambda) \\
    &= \begin{cases}
    0 & j \ne k \\
    i (\beta_k^b (\lambda) + \overline{\beta_k^b (\bar{\lambda})}) & j=k,
    \end{cases}
    \end{aligned}
\end{equation*}
where in obtaining the final inequality we've used the 
relations (\ref{niessen-relations1}) and (\ref{niessen-relations2}). Recalling 
our convention from Claim \ref{niessen-claim1}, we see that we in fact have
\begin{equation*}
 r^b_j (\bar{\lambda})^* J r^b_k (\lambda) = 0, 
 \quad \forall \, j, k \in \{1, 2, \dots, r_b (\lambda)\}. 
\end{equation*}
Next, for $j \in \{1, 2, \dots, r_b (\lambda)\}$, 
$k \in \{r_b (\lambda) + 1, r_b (\lambda) + 2, \dots, n\}$, 
we have 
\begin{equation*}
  r^b_j (\bar{\lambda})^* J r^b_k (\lambda)
  = (v_j^b (\bar{\lambda})^* + \overline{\beta_j^b (\bar{\lambda})} v^b_{n+j} (\bar{\lambda})^*)
  J v_k^b (\lambda) 
  = 0
\end{equation*}
where again we've used the 
relations (\ref{niessen-relations1}) and (\ref{niessen-relations2}). The 
cases  $j \in \{r_b (\lambda) + 1, r_b (\lambda) + 2, \dots, n\}$, 
$k \in \{1, 2, \dots, r_b (\lambda)\}$ and 
$j, k \in \{r_b (\lambda) + 1, r_b (\lambda) + 2, \dots, n\}$ can be 
handled similarly.
\end{proof}

With appropriate labeling, statements analogous to 
Lemma \ref{krall-niessen-lemma} and Claims 
\ref{niessen-claim1} and \ref{niessen-claim2}
can be established with $b$ replaced by $a$.

\subsection{Properties of $\mathcal{L}$ and $\mathcal{L}^{\alpha}$} 
\label{properties-section}

Turning now to consideration of the operators 
$\mathcal{L}$ and $\mathcal{L}^{\alpha}$, 
we will take as our starting point
the following formulation of Green's identity for 
our maximal operator $\mathcal{L}_M$. 

\begin{lemma}[Green's Identity] \label{green-lemma}
For any $y, z \in \mathcal{D}_M$,
we have 
\begin{equation} \label{green1}
    \langle \mathcal{L}_M y, z \rangle_{B_1} 
    - \langle  y, \mathcal{L}_M z \rangle_{B_1}
    = (Jy,z)_a^b,
\end{equation}
where 
\begin{equation*}
     (Jy,z)_a^b =  (Jy,z)_b -  (Jy,z)_a,
\end{equation*}
with 
\begin{equation*}
    \begin{aligned}
    (Jy,z)_a &:= \lim_{x \to a^+} (J y(x),z(x)), \\
    (Jy,z)_b &:= \lim_{x \to b^-} (J y(x),z(x))
    \end{aligned}
\end{equation*} 
(for which the limits are well-defined).
In particular, if $y$ and $z$ satisfy 
$\mathcal{L}_M y = \lambda y$ and $\mathcal{L}_M z = \lambda z$
then 
\begin{equation} \label{green2}
    2 i \mathrm{Im }\lambda \langle y, z \rangle_{B_1} = (Jy,z)_a^b.
\end{equation}
\end{lemma}

\begin{proof}
To begin, we take any $y, z \in \mathcal{D}_M$, and we let $f, g \in L^2_{B_1} ((a, b), \mathbb{C}^{2n})$
respectively denote the uniquely defined functions so that $\mathcal{L}_M y = f$ and 
$\mathcal{L}_M z = g$. By definition of $\mathcal{D}_M$, this means that we have
the relations
\begin{equation*}
    \begin{aligned}
    J y' - B_0 (x)y &= B_1 (x) f \\
    J z' - B_0 (x)z &= B_1 (x) g,
    \end{aligned}
\end{equation*}
for a.e. $x \in (a, b)$. We compute the $\mathbb{C}^{2n}$ inner
product
\begin{equation*}
    (B_1 \mathcal{L}_M y, z) 
    = (B_1 f, z) 
    = (J y' - B_0 y, z) 
    = (Jy', z) - (y, B_0 z),
\end{equation*}
where in obtaining the final equality we have used our assumption that 
$B_0 (x)$ is self-adjoint for a.e. $x \in (a, b)$. Likewise, 
\begin{equation*}
    (B_1 y, \mathcal{L}_M z) 
    = (B_1 y, g) 
    = (y, B_1 g)
    = (y, J z' - B_0 (x) z)
    = (y, Jz') - (y, B_0 z).
\end{equation*}
Subtracting the latter of these relations from the former, 
we see that 
\begin{equation*}
    \frac{d}{d x} (Jy,z)
    =  (B_1 \mathcal{L}_M y, z)
    -  (B_1 y, \mathcal{L}_M z). 
\end{equation*}
For any $c, d \in (a, b)$, $c < d$, we can integrate this 
last relation to see that 
\begin{equation*}
    (J y(d), z(d)) - (Jy (c),z (c))
    = \int_c^d (B_1 (x) \mathcal{L}_M y (x), z (x)) dx
    - \int_c^d (B_1 (x) y (x), \mathcal{L}_M z (x)) dx.
\end{equation*}
If we allow $d$ to remain fixed, then since $y, z \in L^2_{B_1} ((a, b), \mathbb{C}^{2n})$
we see that the limit 
\begin{equation*}
    (Jy, z)_a := \lim_{c \to a^+} (Jy (c),z (c))
\end{equation*}
is well-defined. In particular, we can write 
\begin{equation*}
    (J y(d), z(d)) - (Jy,z)_a
    = \int_a^d (B_1 (x) \mathcal{L}_M y (x), z (x)) dx
    - \int_a^d (B_1 (x) y (x), \mathcal{L}_M z (x)) dx.
\end{equation*}
If we now take $d \to b_-$, we obtain precisely 
(\ref{green1}). Relation (\ref{green2})
is an immediately consequence of (\ref{green1}).
\end{proof}

We turn next to the identification of appropriate domains 
$\mathcal{D}$ and $\mathcal{D}^{\alpha}$ on which the 
respective restrictions of $\mathcal{L}_M$ are self-adjoint.
This development is adapted from Chapter 6 in \cite{Pearson1988}, 
and we begin by making some preliminary definitions. 
We set 
\begin{equation*}
\mathcal{D}_c := \{y \in \mathcal{D}_M: y \textrm{ has compact support in } (a,b)\},
\end{equation*}
and we denote by $\mathcal{L}_c$ the restriction of $\mathcal{L}_M$
to $\mathcal{D}_c$. We can show, as in Theorem 3.9 of \cite{Weidmann1987}
that $\mathcal{L}_c^* = \mathcal{L}_M$, and from Theorem 3.7 of 
that same reference (adapted to the current setting) we know that 
$\mathcal{D}_c$ is dense in $L^2_{B_1} ((a, b), \mathbb{C}^{2n})$. 

\begin{remark}
The {\it minimal operator} associated with $\mathcal{L}_M$ 
is the closure of $\mathcal{L}_c$. We know from Theorem 8.6
in \cite{Weidmann1980} that $\overline{\mathcal{L}_c}$ has 
a self-adjoint extension if and only if its {\it defect indices}
$\gamma_{\pm} (\overline{\mathcal{L}_c})$ agree, where
\begin{equation*}
    \gamma_{\pm} (\overline{\mathcal{L}_c})
    := \dim \ran (\overline{\mathcal{L}_c} \mp i I)^{\perp}
    = \dim \ker (\mathcal{L}_M \pm i I).
\end{equation*}
In addition, we know from Theorem 7.1 of \cite{Weidmann1987}
that 
\begin{equation*}
    \dim \ker (\mathcal{L}_M \pm i I)
    = m_a (\mp i) + m_b (\mp i) - 2n.
\end{equation*}
Our Assumption {\bf (C)} assures us that 
$m_a (i) = m_a (-i)$ and $m_b (i) = m_b (-i)$
so that 
$\gamma_{-} (\overline{\mathcal{L}_c}) = \gamma_{+} (\overline{\mathcal{L}_c})$.
I.e., under Assumption {\bf (C)} the defect indices 
agree, so $\overline{\mathcal{L}_c}$ has 
a self-adjoint extension. 
\end{remark}

For any $\lambda \in \mathbb{C} \backslash \mathbb{R}$, 
we let $\{u_j^b (x; \lambda)\}_{j=1}^n$ denote a selection 
of Niessen elements as described in Claim 
\ref{niessen-claim1}, and we denote by $U^b (x; \lambda)$
the $2n \times n$ matrix comprising the vectors 
$\{u_j^b (x; \lambda)\}_{j=1}^n$ as its columns. 
Likewise we let $\{u_j^a (x; \lambda)\}_{j=1}^n$
denote a collection of Niessen elements that can 
similarly be specified in association with $x = a$,
and we denote by $U^a (x; \lambda)$
the $2n \times n$ matrix comprising the vectors 
$\{u_j^a (x; \lambda)\}_{j=1}^n$ as its columns. 
Next, we introduce functions 
$\rho_a, \rho_b \in C^{\infty} ((a, b), \mathbb{R})$ so that 
\begin{equation*}
    \rho_a (x) 
    = \begin{cases}
    1 & \mathrm{near }\,\, x = a \\
    0 & \mathrm{near }\,\, x = b
    \end{cases}; \quad \quad
    \rho_b (x) 
    = \begin{cases}
    0 & \mathrm{near }\,\, x = a \\
    1 & \mathrm{near }\,\, x = b
    \end{cases}, 
\end{equation*}
and we define 
\begin{equation*}
   \begin{aligned}
       \tilde{u}_j^a (x; \lambda) &= \rho_a (x) u_j^a (x; \lambda), \\
       \tilde{u}_j^b (x; \lambda) &= \rho_b (x) u_j^b (x; \lambda).
   \end{aligned} 
\end{equation*}

For some fixed $\lambda_0 \in \mathbb{C}\backslash \mathbb{R}$, We specify the domain 
\begin{equation} \label{D1-def}
    \mathcal{D}_{\lambda_0} := \mathcal{D}_c 
    + \mathrm{Span}\,\Big{\{} \{\tilde{u}_j^a (\cdot; \lambda_0)\}_{j=1}^n, \{\tilde{u}_j^b (\cdot; \lambda_0)\}_{j=1}^n \Big{\}},
\end{equation}
and we denote by $\mathcal{L}_{\lambda_0}$ the restriction of $\mathcal{L}_M$ to $\mathcal{D}_{\lambda_0}$. 

\begin{theorem}
Let Assumptions {\bf (A)}, {\bf (B)} and {\bf (C)} hold. 
Then the operator $\mathcal{L}_{\lambda_0}$ is essentially self-adjoint, and 
so in particular, $\mathcal{L} := \overline{\mathcal{L}_{\lambda_0}} = \mathcal{L}_{\lambda_0}^*$ 
is self-adjoint. The domain $\mathcal{D}$ of $\mathcal{L}$
is 
\begin{equation}
    \mathcal{D} = \{y \in \mathcal{D}_M: \lim_{x \to a^+} U^a (x; \lambda_0)^* J y(x) = 0,
    \quad \lim_{x \to b^-} U^b (x; \lambda_0)^* J y(x) = 0\}.
\end{equation}
\end{theorem}

\begin{proof}
First, let's check that $\mathcal{L}_{\lambda_0}$ is symmetric. Using (\ref{green1}),
we immediately see that for any $y, z \in \mathcal{D}_c$ we have 
\begin{equation*}
    \langle \mathcal{L}_{\lambda_0} y, z \rangle_{B_1} 
    - \langle  y, \mathcal{L}_{\lambda_0} z \rangle_{B_1}
    = (Jy,z)_a^b = 0.
\end{equation*}
and we can similarly use (\ref{green1}) along with the identities 
\begin{equation*}
    (J y, \tilde{u}_j^a)_a^b = 0,
    \quad (Jy, \tilde{u}_j^b)_a^b = 0,
    \quad (J \tilde{u}_j^a, \tilde{u}_k^b)_a^b = 0,
\end{equation*}
for all $j, k \in \{1, 2, \dots, n\}$ (following from support 
of the elements in all cases). It remains to show that 
\begin{equation}
    (J \tilde{u}_j^a, \tilde{u}_k^a)_a^b = 0,
    \quad (J \tilde{u}_j^b, \tilde{u}_k^b)_a^b = 0,
\end{equation}
but these identities are immediate from Lemma \ref{krall-niessen-lemma} (along
with the analogous statement associated with $x = a$),  
so symmetry is established. 

Next, we'll show that $\mathcal{L}_{\lambda_0}$ is essentially self-adjoint. 
According to Theorem 5.21 in \cite{Weidmann1997}, it suffices to 
show that for some (and hence for all) $\lambda \in \mathbb{C}\backslash \mathbb{R}$,
\begin{equation} \label{essential-sa-condition}
    \overline{\ran(\mathcal{L}_{\lambda_0} - \lambda)} = L^2_{B_1} ((a, b), \mathbb{C}^{2n}),
    \quad {\rm and} \quad \overline{\ran(\mathcal{L}_{\lambda_0} - \bar{\lambda})} = L^2_{B_1} ((a, b), \mathbb{C}^{2n}). 
\end{equation}
Since we can proceed with any $\lambda \in \mathbb{C} \backslash \mathbb{R}$, we can take
$\lambda_0$ from (\ref{D1-def}) as our choice. This is what we'll do, though for 
notational convenience we will denote this value by $\lambda$ for 
the rest of this proof.

We will show that 
\begin{equation} \label{ess-sa-cond}
    \ran(\mathcal{L}_{\lambda_0} - \lambda)^{\perp} = \{0\},
    \quad {\rm and} \quad \ran(\mathcal{L}_{\lambda_0} - \bar{\lambda})^{\perp} = \{0\}, 
\end{equation}
from which (\ref{essential-sa-condition}) is clear, since 
\begin{equation}
    L^2_{B_1} ((a, b), \mathbb{C}^{2n}) = 
    \ran(\mathcal{L}_{\lambda_0} - \lambda)^{\perp} \oplus  \overline{\ran(\mathcal{L}_{\lambda_0} - \lambda)},
\end{equation}
and likewise with $\lambda$ replaced by $\bar{\lambda}$.

Starting with the second relation in (\ref{ess-sa-cond}),
we suppose that for some $u \in L^2_{B_1} ((a, b), \mathbb{C}^{2n})$,
$\langle (\mathcal{L}_{\lambda_0} - \bar{\lambda}I) \psi, u \rangle_{B_1} = 0$
for all $\psi \in \mathcal{D}_{\lambda_0}$, and our goal is to show that 
this implies that $u = 0$. First, if we restrict to $\psi \in \mathcal{D}_c$, then we have
\begin{equation}
  \langle (\mathcal{L}_c - \bar{\lambda}I) \psi, u \rangle_{B_1} = 0, 
  \quad \forall \, \psi \in \mathcal{D}_c.   
\end{equation}
This relation implies that $u \in \dom ((\mathcal{L}_c - \bar{\lambda}I)^*) \,\, (= \mathcal{D}_M)$,
so we're justified in writing 
\begin{equation}
  \langle \psi, (\mathcal{L}_M - \lambda I) u \rangle_{B_1} = 0, 
  \quad \forall \, \psi \in \mathcal{D}_c.   
\end{equation}
Since $\mathcal{D}_c$ is dense in $L^2_{B_1} ((a, b), \mathbb{C}^{2n})$,
we can conclude that $u$ must satisfy $(\mathcal{L}_M - \lambda I) u = 0$. 

Next, we also have the relation 
\begin{equation}
  \langle (\mathcal{L}_{\lambda_0} - \bar{\lambda}I) \psi, u \rangle_{B_1} = 0, 
  \quad \forall \, \psi \in \mathrm{Span}\,\Big{\{} \{\tilde{u}_j^a\}_{j=1}^n, \{\tilde{u}_j^b\}_{j=1}^n \Big{\}}.   
\end{equation}
For each $j \in \{1, 2, \dots, n\}$, $\tilde{u}^b_j \in \mathcal{D}_M$, and 
we've already established that $u \in \mathcal{D}_M$, so we can apply 
Green's identity (\ref{green1}) to see that 
\begin{equation}
    \langle (\mathcal{L}_{\lambda_0} - \bar{\lambda}I) \tilde{u}_j^b, u \rangle_{B_1}
    = \langle \tilde{u}_j^b, (\mathcal{L}_M - \lambda I) u \rangle_{B_1}
    + (J \tilde{u}_j^b, u)_a^b.
\end{equation}
Since $(\mathcal{L}_M - \lambda I) u = 0$, we see that 
$(J \tilde{u}_j^b, u)_a^b = 0$. In addition, since 
$\tilde{u}_j^b$ is zero near $x=a$, we have
$(J \tilde{u}_j^b, u)_a = 0$, and consequently we can 
conclude $(J \tilde{u}_j^b, u)_b = 0$.
That is, 
\begin{equation*}
    \lim_{x \to b^-} u(x)^* J \tilde{u}_j^b (x; \lambda) = 0.
\end{equation*}
If we take the adjoint of this relation, and recall that 
$\tilde{u}_j^b$ is identical to $u_j^b$ for $x$ near
$b$, then we can express this limit in our preferred form 
\begin{equation*}
    \lim_{x \to b^-} u_j^b (x; \lambda)^* J u(x) = 0.
\end{equation*}
This last relation is true for all $j \in \{1, 2, \dots, n\}$,
and a similar relation holds near $x = a$. We can summarize these
observations with the following limits
\begin{equation}
    \begin{aligned}
        \lim_{x \to a^+} U^a (x; \lambda)^* J u(x) &= 0, \\ 
         \lim_{x \to b^-} U^b (x; \lambda)^* J u(x) &= 0.
    \end{aligned}
\end{equation}
We would like to show the following: the first of these relations 
ensures that $u$ can be expressed as a linear combination of the 
columns of $U^a (\cdot; \lambda)$, while the second ensures that $u$ can be 
expressed as a linear combination of the columns of $U^b (\cdot; \lambda)$.

Here, $u \in \mathcal{D}_M$ and $\mathcal{L}_M u = \lambda u$,
so $u$ must be a linear combination of the Niessen elements that lie
left in $(a, b)$, and at the same time, $u$ must be a linear combination
of the Niessen elements that lie right in $(a, b)$. If we focus on 
the case $x = b$, our labeling scheme sets 
$\{N_j^b (\lambda)\}_{j=1}^{r_b}$ to 
be the Niessen spaces satisfying 
$\dim N_j^b (\lambda) \cap L^2_{B_1} ((c, b), \mathbb{C}^{2n}) = 2$ and 
sets $\{N_j^b (\lambda)\}_{j=r_b+1}^n$ to be the Niessen spaces satisfying 
$\dim N_j^b (\lambda) \cap L^2_{B_1} ((c, b), \mathbb{C}^{2n}) = 1$. Here, we 
recall that $r_b = m_b - n$, where $m_b$ denotes the dimension
of the space of solutions to (\ref{linear-hammy}) that lie right
in $(a, b)$. 

The elements $\{u_j^b (x; \lambda)\}_{j=1}^{r_b}$
and $\{v_j^b (x; \lambda)\}_{j=1}^{r_b}$ are as described in 
Claim \ref{niessen-claim1}, and by construction, the collection 
$\{ \{u_j^b (x; \lambda)\}_{j=1}^n,  \{v_j^b (x; \lambda)\}_{j=1}^{r_b}\}$
is a basis for the space of solutions to (\ref{linear-hammy}) that 
lie right in $(a, b)$, so we can write 
\begin{equation*}
    u(x) = \sum_{j = 1}^n c_j (\lambda) u_j^b (x; \lambda)
    + \sum_{j=1}^{r_b} d_j (\lambda) v_j^b (x; \lambda),
\end{equation*}
for some appropriate scalar functions (of $\lambda$)
$\{c_j (\lambda)\}_{j=1}^n$, $\{d_j (\lambda\}_{j=1}^{r_b}$.
The boundary operator 
\begin{equation*}
    B_b (\lambda) u := \lim_{x \to b^-} U^b (x; \lambda)^* J u(x)
\end{equation*}
annihilates the elements $\{u_j^b (x; \lambda)\}_{j=1}^n$, so we 
immediately see that 
\begin{equation*}
    B_b (\lambda) u = \sum_{j=1}^{r_b} d_j (\lambda) B_b (\lambda) v_j^b (\cdot; \lambda).
\end{equation*}
According to Lemma \ref{krall-niessen-lemma},
we have 
\begin{equation*}
    (B_b (\lambda) v_j^b (\cdot; \lambda))_i = 
    \begin{cases}
    0 & i \ne j \\
    \kappa_j^b \ne 0 & i = j.
    \end{cases}
\end{equation*}
In this way, we see that 
\begin{equation*}
    B_b (\lambda) u = (d_1 (\lambda) \kappa_1 \,\, \dots \,\, d_{r_b} \kappa_{r_b} \,\,
    0 \,\, 0 \,\, \dots \,\, 0)^T,
\end{equation*}
and this can only be identically 0 if $d_j (\lambda) = 0$
for all $j \in \{1, 2, \dots, r_b\}$. We conclude
that there exists a $\zeta^b (\lambda) \in \mathbb{C}^n$
so that $u (x) = U^b (x; \lambda) \zeta^b (\lambda)$ for all 
$x \in (a, b)$, and similarly we can check that there exists a 
$\zeta^a (\lambda) \in \mathbb{C}^n$ so that $u(x) = U^a (x; \lambda) \zeta^a (\lambda)$ for all 
$x \in (a, b)$. This allows us to compute, using (\ref{green2}),
\begin{equation*}
    \begin{aligned}
    2i \textrm{Im } \lambda \|u\|_{B_1}^2
    &= (Ju, u)_a^b = (Ju, u)_b - (Ju, u)_a \\
    &= (J U^b \zeta^b, U^b \zeta^b)_b - (J U^a \zeta^a, U^a \zeta^a)_a = 0.
    \end{aligned}
\end{equation*}
We conclude from Atkinson positivity (i.e., Assumption {\bf (B)}) 
that $u=0$ in $L^2_{B_1} ((a, b), \mathbb{C}^{2n})$, and this establishes the first identity in (\ref{ess-sa-cond}). 

We now turn to the first condition in (\ref{ess-sa-cond}). For 
this, we suppose that for some $u \in L^2_{B_1} ((a, b), \mathbb{C}^{2n})$,
$\langle (\mathcal{L}_{\lambda_0} - \lambda I) \psi, u \rangle_{B_1} = 0$
for all $\psi \in \mathcal{D}_{\lambda_0}$, and our goal is to show that 
this implies that $u = 0$. Precisely as in the previous case, 
we can conclude that we must have $u \in \mathcal{D}_M$,
and $\mathcal{L}_M u = \bar{\lambda} u$, and
continuing as with the previous case, we next find that 
\begin{equation}
    \begin{aligned}
        \lim_{x \to a^+} U^a (x; \lambda)^* J u(x) &= 0, \\ 
         \lim_{x \to b^-} U^b (x; \lambda)^* J u(x) &= 0.
    \end{aligned}
\end{equation}
In this case, $u$ solves the ODE system
\begin{equation} \label{linear-hammy-bar}
    Ju' = (B_0 (x) + \bar{\lambda} B_1 (x)) u,
\end{equation}
so in particular there exists some vector $\zeta(\bar{\lambda}) \in \mathbb{C}^{2n}$
so that 
\begin{equation*}
    u(x) = \Phi (x; \bar{\lambda}) \zeta (\bar{\lambda}), 
\end{equation*}
where $\Phi (x; \bar{\lambda})$ denotes a fundamental 
solution to (\ref{linear-hammy-bar}) with 
$\Phi (c; \bar{\lambda}) = I_{2n}$.
Recalling that $U^b (x; \lambda) = \Phi (x; \lambda) \mathbf{R}^b (\lambda)$, 
this allows us to compute
\begin{equation*}
    U^b (x; \lambda)^* J u(x)
     = \mathbf{R}^b (\lambda)^* \Phi (x; \lambda)^* J \Phi (x; \bar{\lambda}) \zeta(\bar{\lambda})
     = \mathbf{R}^b (\lambda)^* J \zeta(\bar{\lambda}),
\end{equation*}
where we've used the relation 
\begin{equation*}
 \Phi (x; \lambda)^* J \Phi (x; \bar{\lambda}) = J.   
\end{equation*}
In this way, we see that we can only have 
\begin{equation*}
    \lim_{x \to b^-} U^b (x; \lambda)^* J u(x) = 0
\end{equation*}
if 
\begin{equation} \label{the-relation}
     \mathbf{R}^b (\lambda)^* J \zeta(\bar{\lambda}) = 0.
\end{equation}
The $n \times 2n$ matrix $\mathbf{R}^b (\lambda)^*$ has rank $n$, with 
corresponding nullity $n$, and we know from 
Claim \ref{niessen-claim2} that the kernel 
of $\mathbf{R}^b (\lambda)^*$ is spanned by the 
columns of $J \mathbf{R}^b (\bar{\lambda})$. We see
that (\ref{the-relation}) can only hold if 
$\zeta (\bar{\lambda}) \in \colspan \mathbf{R}^b (\bar{\lambda})$,
and in this case there exists a vector 
$\zeta^b (\bar{\lambda}) \in \mathbb{C}^n$ so that 
$\zeta (\bar{\lambda}) = \mathbf{R}^b (\bar{\lambda}) \zeta^b (\bar{\lambda})$,
and consequently 
$u (x) = \Phi (x; \bar{\lambda}) \zeta (\bar{\lambda}) = U^b (x; \bar{\lambda}) \zeta^b (\bar{\lambda})$. Likewise, 
we must have $u (x) = U^a (x; \bar{\lambda}) \zeta^a (\bar{\lambda})$
for some $\zeta^a (\bar{\lambda}) \in \mathbb{C}^n$.
Since $u \in \mathcal{D}_M$ satisfies 
$\mathcal{L}_M u = \bar{\lambda} u$, 
(\ref{green2}) becomes 
\begin{equation} \label{green2applied}
    \begin{aligned}
    - 2i \textrm{Im } \lambda \|u\|_{B_1}^2
    &= (Ju, u)_a^b \\
    &= (J U^b (\cdot; \bar{\lambda}) \zeta^b (\bar{\lambda}), U^b (\cdot; \bar{\lambda}) \zeta^b (\bar{\lambda}))_b 
    - (J U^a (\cdot; \bar{\lambda}) \zeta^a (\bar{\lambda}), U^a (\cdot; \bar{\lambda}) \zeta^a (\bar{\lambda}))_a.
    \end{aligned}
\end{equation}
By construction, the columns of $U^a (x; \bar{\lambda})$ are  
Niessen elements for (\ref{linear-hammy}) with $\lambda$ 
replaced by $\bar{\lambda}$, and similarly for $U^b (x; \bar{\lambda})$,
so we can conclude from Lemma \ref{krall-niessen-lemma} (applied with 
$\lambda$ replaced by $\bar{\lambda}$) that the two quantities
on the right-hand side of (\ref{green2applied}) are both 0. In this
way, we see that $\|u\|_{B_1} = 0$ and so $u = 0$ in 
$L^2_{B_1} ((a,b), \mathbb{C}^{2n})$. This establishes 
the second identity in (\ref{ess-sa-cond}).

Next, we characterize the operator $\mathcal{L}$, along with 
its domain $\mathcal{D} = \dom (\mathcal{L})$. First, we
have 
\begin{equation*}
    \mathcal{L}_c \subset \mathcal{L}_{\lambda_0}
    \implies \mathcal{L}_{\lambda_0}^* \subset \mathcal{L}_c^*,
\end{equation*}
and since $\mathcal{L}_{\lambda_0}^* = \mathcal{L}$ and 
$\mathcal{L}_c^* = \mathcal{L}_M$, we see that 
$\mathcal{L} \subset \mathcal{L}_M$. This leaves
only the question of what additional restrictions 
we have on $\mathcal{D}$ (in addition to the 
requirements of $\mathcal{D}_M$). Here, 
\begin{equation*}
    \begin{aligned}
    \mathcal{D} = \{u &\in \mathcal{D}_M: 
    \textrm{ there exists } v \in L^2_{B_1} ((a, b), \mathbb{C}^{2n}) \\
    &\textrm{ so that } \langle \mathcal{L}_{\lambda_0} \psi, u \rangle_{B_1}
    = \langle \psi, v \rangle_{B_1} \textrm{ for all } \psi \in \mathcal{D}_{\lambda_0}\}.
    \end{aligned}
\end{equation*}

Let $u \in \mathcal{D}_M$. For all $\psi \in \mathcal{D}_c$, we can immediately 
write 
\begin{equation*}
    \langle \mathcal{L}_{\lambda_0} \psi, u \rangle_{B_1} 
    = \langle \mathcal{L}_c \psi, u \rangle_{B_1} 
    = \langle \psi, \mathcal{L}_M u \rangle_{B_1}
    = \langle \psi, v \rangle_{B_1},
    \quad (v = \mathcal{L}_M u),
\end{equation*}
so in particular there are no additional restrictions on $\mathcal{D}$.
On the other hand, for any $j \in \{1, 2, \dots, n\}$, we have
Green's Identity
\begin{equation}
    \langle \mathcal{L}_{\lambda_0} \tilde{u}_j^a, u \rangle_{B_1}
     = \langle \tilde{u}_j^a, \mathcal{L}_M u \rangle_{B_1}
     - (J \tilde{u}_j^a, u)_a,
\end{equation}
where we've recalled that $\tilde{u}_j^a$ is 0 near 
$x = b$. We require $(J \tilde{u}_j^a, u)_a = 0$,
and since this must be true for all $j \in \{1, 2, \dots, n\}$,
we obtain the additional condition
\begin{equation*}
    \lim_{x \to a^+} U^a (x; \lambda)^* J u(x) = 0.
\end{equation*}
(Here, we're using the fact that $\mathcal{D} \subset \mathcal{D}_M$ 
to ensure that $\mathcal{L}_M u$ is the only candidate for $v$.) 
Proceeding similarly for $x = b$, we obtain additionally
\begin{equation*}
    \lim_{x \to b^-} U^b (x; \lambda)^* J u(x) = 0.
\end{equation*}
We've now exhausted the elements from $\mathcal{D}_{\lambda_0}$,
so these are the only possible additional constraints 
imposed on $\mathcal{D}$. This completes the proof.
\end{proof}

By essentially identical considerations, we can establish 
a similar theorem for $\mathcal{L}^{\alpha}$. In this 
case, we introduce solutions $\{u_j^{\alpha}\}_{j=1}^n$
to (\ref{linear-hammy}) initialized so that if 
$U^{\alpha} (x; \lambda)$ denotes the $2n \times n$
matrix comprising the elements $\{u_j^{\alpha}\}_{j=1}^n$
as its columns, then $U^{\alpha} (a; \lambda) = J \alpha^*$.
We now fix some $\lambda_0 \in \mathbb{C}\backslash \mathbb{R}$, 
and specify the domain 
\begin{equation} \label{D1-def-alpha}
    \mathcal{D}_{\lambda_0}^{\alpha} := \mathcal{D}_c 
    + \mathrm{Span}\,\Big{\{} \{\tilde{u}_j^{\alpha} (\cdot; \lambda_0)\}_{j=1}^n, \{\tilde{u}_j^b (\cdot; \lambda_0)\}_{j=1}^n \Big{\}}.
\end{equation}
We denote by $\mathcal{L}_{\lambda_0}^{\alpha}$ the restriction of $\mathcal{L}_M$ to $\mathcal{D}_{\lambda_0}^{\alpha}$. 

\begin{theorem} \label{L-alpha-theorem}
Let Assumptions {\bf (A)}, {\bf (A)$^\prime$}, {\bf (B)},
and {\bf (C)} hold. 
Then the operator $\mathcal{L}_{\lambda_0}^{\alpha}$ is essentially self-adjoint, and 
so in particular, $\mathcal{L}^{\alpha} := \overline{\mathcal{L}_{\lambda_0}^{\alpha}} = \mathcal{L}_{\lambda_0}^{\alpha \, *}$ 
is self-adjoint. The domain $\mathcal{D}^{\alpha}$ of $\mathcal{L}^{\alpha}$
is 
\begin{equation}
    \mathcal{D}^{\alpha} = \{y \in \mathcal{D}_M: \alpha y(a) = 0,
    \quad \lim_{x \to b^-} U^b (x; \lambda_0)^* J y(x) = 0\}.
\end{equation}
\end{theorem}

\begin{remark} \label{selection-remark} In conjunction with Lemma 
\ref{self-adjoint-operator-lemma}, we summarize the developments of 
Sections \ref{niessen-section} and \ref{properties-section}. In order
to specify the operator $\mathcal{L}$, we make a selection of 
Niessen elements $\{u_j^a (x; \lambda)\}_{j=1}^n$ and 
$\{u_j^b (x; \lambda)\}_{j=1}^n$ as described in Claim 
\ref{niessen-claim1}, and we denote by $U^a (x; \lambda)$ the 
matrix comprising the vector functions $\{u_j^a (x; \lambda)\}_{j=1}^n$
as its columns, and by $U^b (x; \lambda)$ the 
matrix comprising the vector functions $\{u_j^b (x; \lambda)\}_{j=1}^n$
as its columns. Then $\mathcal{L}$ is obtained from the maximal operator 
$\mathcal{L}_M$ by imposing the boundary conditions 
\begin{equation*}
    \lim_{x \to a^+} U^a (x; \lambda)^* J y(x) = 0; 
    \quad {\rm and} \quad \lim_{x \to b^-} U^b (x; \lambda)^* J y(x) = 0,
\end{equation*}
and $\mathcal{L}^{\alpha}$ is obtained from the maximal operator 
$\mathcal{L}_M^{\alpha}$ by imposing the boundary conditions 
\begin{equation*}
    \alpha y(a) = 0; 
    \quad {\rm and} \quad \lim_{x \to b^-} U^b (x; \lambda)^* J y(x) = 0.
\end{equation*}

\end{remark}

\subsection{Continuation to $\mathbb{R}$}
\label{continuation-section}

In the preceding considerations, we fixed some 
$\lambda_0 \in \mathbb{C} \backslash \mathbb{R}$ 
and used this value to specify the self-adjoint
operators $\mathcal{L}$ and $\mathcal{L}^{\alpha}$.
With these operators in hand, we would next like 
to fix values $\lambda \in \mathbb{R}$ and 
construct solutions $u^a (x; \lambda)$ 
to $\mathcal{L} y = \lambda y$ that lie 
left in $(a, b)$, along with solutions 
$u^b (x; \lambda)$ to $\mathcal{L} y = \lambda y$ that lie 
right in $(a, b)$ (and similarly for 
$\mathcal{L}^{\alpha}$). One difficulty we 
face is that the matrix $\mathcal{A} (x; \lambda)$
is not defined for $\lambda \in \mathbb{R}$,
and so we cannot directly extend Niessen's 
development to this setting. (Though 
see Section \ref{applications-section} for 
a calculation along these lines.) Instead of 
extending Niessen's development directly, 
we'll take advantage of our assumption that 
$[\lambda_1, \lambda_2]$ does not intersect the
essential spectrum of our operator of interest, 
along with a standard theorem about self-adjoint
operators. 

As a starting point, we fix some $c \in (a, b)$ and consider 
(\ref{linear-hammy}) on $(c, b)$ with 
boundary conditions 
\begin{equation} \label{boundary-c}
    \gamma y(c) = 0,
\end{equation}
and 
\begin{equation} \label{boundary-b}
    \lim_{x \to b^-} U^b (x; \lambda_0)^* J y(x) = 0,
\end{equation}
where the boundary matrix $\gamma \in \mathbb{C}^{n \times 2n}$
satisfies 
\begin{equation} \label{boundary-matrix}
  \rank \gamma = n, 
    \, \textrm{ and } \, \gamma J \gamma^* = 0.  
\end{equation}

Similarly as in Section \ref{properties-section}, we can associate 
(\ref{linear-hammy})-(\ref{boundary-c})-(\ref{boundary-b})
with a self-adjoint operator $\mathcal{L}_{c, b}^{\gamma}$,
with domain 
\begin{equation*}
    \mathcal{D}_{c, b}^{\gamma} := \{y \in \mathcal{D}_{c, b, M}: 
    \gamma y (c) = 0, 
    \quad \lim_{x \to b^-} U^b (x; \lambda_0) J y(x) = 0 \}.
\end{equation*}
Here, $\mathcal{D}_{c, b, M}$ denotes the domain of the 
maximal operator associated with (\ref{linear-hammy}) on $(c, b)$.

We start with a lemma. 

\begin{lemma} \label{lagrangian-limit-lemma}
Let Assumptions {\bf (A)}, {\bf (B)}, and {\bf (C)} hold. 
For any fixed $\lambda \in \mathbb{C}$, suppose 
$u^b (x; \lambda), v^b (x; \lambda)$ denote any two solutions of (\ref{linear-hammy}) 
(if such solutions exist)
that lie right in $(c, b)$ and satisfy (\ref{boundary-b}).
Then 
\begin{equation*}
    (Ju^b (\cdot; \lambda),v^b (\cdot; \lambda))_b = 0.
\end{equation*}
\end{lemma}

\begin{proof}
Since $u^b (x; \lambda), v^b (x; \lambda)$ lie right in $(c, b)$ 
and satisfy (\ref{boundary-b}), it's clear that the truncated
functions $\tilde{u}^b (x; \lambda), \tilde{v}^b (x; \lambda)$,
truncated with 
\begin{equation*}
     \rho_b (x) 
    = \begin{cases}
    0 & \mathrm{near }\,\, x = c \\
    1 & \mathrm{near }\,\, x = b
    \end{cases},
\end{equation*}
are contained in $\mathcal{D}_{c, b}^{\gamma}$. Using self-adjointness 
of $\mathcal{L}_{c, b}^{\gamma}$, we can write 
\begin{equation*}
    \begin{aligned}
    0 &= \langle \mathcal{L}_{c, b}^{\gamma} \tilde{u}^b (\cdot; \lambda), \tilde{v}^b (\cdot; \lambda) \rangle_{B_1}
    - \langle \tilde{u}^b (\cdot; \lambda), \mathcal{L}_{c, b}^{\gamma} \tilde{v}^b (\cdot; \lambda) \rangle_{B_1} \\
    &= (J \tilde{u}^b (\cdot; \lambda), \tilde{v}^b (\cdot; \lambda))_c^b 
    = (J \tilde{u}^b (\cdot; \lambda), \tilde{v}^b (\cdot; \lambda))_b. 
    \end{aligned}
\end{equation*}
Since  $\tilde{u}^b (x; \lambda), \tilde{v}^b (x; \lambda)$ are identical 
to $u^b (x; \lambda), v^b(x; \lambda)$ for $x$ near $b$, this gives the claim.
\end{proof}

\begin{lemma} \label{at-most-n}
Let Assumptions {\bf (A)}, {\bf (B)}, 
and {\bf (C)} hold. 
Then for any fixed $\lambda \in \mathbb{R}$, the space of solutions 
of (\ref{linear-hammy}) (if such solutions exist)
that lie right in $(c, b)$ and satisfy (\ref{boundary-b})
has dimension at most $n$. In the event that the dimension 
of this space is $n$, we let $\{u^b_j (x; \lambda)\}_{j=1}^n$
denote a choice of basis. Then for each $x \in (c, b)$ the 
vectors $\{u^b_j (x; \lambda)\}_{j=1}^n$ comprise the 
basis for a Lagrangian subspace of $\mathbb{C}^{2n}$.
\end{lemma}

\begin{proof}
Let $d$ denote the dimension of the space of solutions 
of (\ref{linear-hammy}) that lie right in $(c, b)$ and 
satisfy (\ref{boundary-b}), and suppose $d \ge n$. Let
$\{u_j^b (x; \lambda)\}_{j=1}^d$ denote a basis for 
this space, and notice that for any 
$j, k \in \{1, 2, \dots, d\}$ (and with $^{\prime}$ denoting
differentiation with respect to $x$), 
\begin{equation*}
    \begin{aligned}
    (u_j^b &(x; \lambda)^* J u^b_k (x; \lambda))'
    = u_j^{b \, \prime} (x; \lambda)^* J u^b_k (x; \lambda)
    + u^b_j (x; \lambda)^* J u_k^{b \, \prime} (x; \lambda) \\
    &= - (J u_j^{b \, \prime} (x; \lambda))^* u^b_k (x; \lambda)
    + u_j (x; \lambda)^* J u_k^{b \, \prime} (x; \lambda) \\
    &= - ((B_0 (x) + \lambda B_1 (x)) u_j^b (x; \lambda))^* u^b_k (x; \lambda)
    +  u_j (x; \lambda)^* ((B_0 (x) + \lambda B_1 (x)) u_k^b (x; \lambda) \\
    & - u_j (x; \lambda)^* ((B_0 (x) + \lambda B_1 (x)) u_k^b (x; \lambda)
    + u_j (x; \lambda)^* ((B_0 (x) + \lambda B_1 (x)) u_k^b (x; \lambda)
    = 0.
    \end{aligned}
\end{equation*}
We see that $u_j^b (x; \lambda)^* J u^b_k (x; \lambda)$ is constant for 
all $x \in (c, b)$. In addition, according to Lemma 
\ref{lagrangian-limit-lemma}, we have 
\begin{equation*}
    \lim_{x \to b^-} u_j^b (x; \lambda)^* J u^b_k (x; \lambda) = 0.
\end{equation*}
We conclude that $u_j^b (x; \lambda)^* J u^b_k (x; \lambda) = 0$ for 
all $x \in (c, b)$.

We see immediately that the first $n$ elements $\{u^b_j (x; \lambda)\}_{j=1}^n$
(or any other $n$ elements taken from $\{u_j^b (x; \lambda)\}_{j=1}^d$) 
form the basis for a Lagrangian subspace of $\mathbb{C}^{2n}$ for all
$x \in (c, b)$. If $d > n$, we get a contradiction to the maximality 
of Lagrangian subspaces, and so we can conclude that $d = n$ (recalling
that this is under the assumption that $d \ge n$). This, of course, 
leaves open the possibility that the dimension 
of the space of solutions 
of (\ref{linear-hammy}) that lie right in $(c, b)$ and 
satisfy (\ref{boundary-b}) is less than $n$.
\end{proof}

\begin{lemma} \label{boundary-matrix-lemma}
Let Assumptions {\bf (A)}, {\bf (B)}, and {\bf (C)} hold. 
Then for any fixed $\lambda \in \mathbb{R}$, 
there exists a matrix $\gamma \in \mathbb{C}^{n \times 2n}$ 
satisfying (\ref{boundary-matrix}) so that $\lambda$
is not an eigenvalue of $\mathcal{L}_{c, b}^{\gamma}$.
\end{lemma}

\begin{proof}
First, we recall that $\lambda$ is an eigenvalue of 
$\mathcal{L}_{c, b}^{\gamma}$ if and only if 
there exists a solution 
\begin{equation*}
    y (\cdot; \lambda) 
    \in \AC_{\loc} ([c, b), \mathbb{C}^{2n}) 
    \cap L^2_{B_1} ((c, b), \mathbb{C}^{2n}) 
\end{equation*}
to (\ref{linear-hammy}) so that (\ref{boundary-c})
and (\ref{boundary-b}) are both satisfied. 
Also, according to Lemma \ref{at-most-n}, 
the space of solutions of (\ref{linear-hammy})
that lie right in $(c, b)$ and satisfy (\ref{boundary-b})
has dimension at most $n$. We begin by assuming
that this space of solutions has dimension $n$, and we denote
a basis for the space by $\{u^b_j (x; \lambda)\}_{j=1}^n$.

As usual, we let $\Phi (x; \lambda)$ denote a fundamental 
matrix for (\ref{linear-hammy}), initialized 
by $\Phi (c; \lambda) = I_{2n}$. If $U^b (x; \lambda)$
denotes the matrix comprising $\{u^b_j (x; \lambda)\}_{j=1}^n$
as its columns, then there exists a $2n \times n$ 
matrix $\mathbf{R}^b (\lambda) = {R^b (\lambda) \choose S^b (\lambda)}$
so that 
\begin{equation*}
    U^b (x; \lambda) = \Phi (x; \lambda) \mathbf{R}^b (\lambda),
\end{equation*}
for all $x \in [c, b)$. Recalling the identity 
\begin{equation*}
    \Phi (x; \lambda)^* J \Phi (x; \lambda) = J
\end{equation*}
(i.e., (\ref{conj-no-conj}) with $\lambda \in \mathbb{R}$),
we can compute 
\begin{equation*}
    U^b (x; \lambda)^* J U^b (x; \lambda) = 
    \mathbf{R}^b (\lambda)^* \Phi (x; \lambda)^* J \Phi (x; \lambda) \mathbf{R}^b (\lambda)
     = \mathbf{R}^b (\lambda)^* J \mathbf{R}^b (\lambda).
\end{equation*}
We know from Lemma \ref{at-most-n} that $U^b (x; \lambda)$ is a frame 
for a Lagrangian subspace of $\mathbb{C}^{2n}$, and it follows 
immediately that the same is true for $\mathbf{R}^b (\lambda)$.

A value $\lambda \in \mathbb{R}$ will be an eigenvalue of 
$\mathcal{L}_{c, b}^{\gamma}$
if and only if there exists a vector 
$v \in \mathbb{C}^n$ so that 
$y(x; \lambda) = \Phi (x; \lambda) \mathbf{R}^b (\lambda) v$
satisfies 
\begin{equation*}
    \gamma y(c; \lambda) = 0,
\end{equation*}
which we can express (since $\Phi (c; \lambda) = I_{2n}$)
as $\gamma \mathbf{R}^b (\lambda) v = 0$. This relation will 
hold for a vector $v \ne 0$ if and only if the Lagrangian 
spaces with frames $J \gamma^*$ and $\mathbf{R}^b (\lambda)$
intersect. We choose $\gamma = \mathbf{R}^b (\lambda)^*$, noting
that in this case 
\begin{equation*}
    \gamma J \gamma^*
    = \mathbf{R}^b (\lambda)^* J \mathbf{R}^b (\lambda) = 0
\end{equation*}
(i.e., this is a valid choice for $\gamma$, satisfying (\ref{boundary-matrix})) but $\gamma \mathbf{R}^b (\lambda) =  \mathbf{R}^b (\lambda)^* \mathbf{R}^b (\lambda)$
is certainly non-singular, so $\lambda$ is not an 
eigenvalue of $\mathcal{L}_{c, b}^{\gamma}$.

In the event that the space of solutions of (\ref{linear-hammy})
that lie right in $(c, b)$ and satisfy (\ref{boundary-b})
has dimension less than $n$, the matrix $\mathbf{R}^b (\lambda)$
(as constructed just above) will have fewer than $n$ columns, but we can add columns 
(which don't correspond with solutions of (\ref{linear-hammy})
that lie right in $(c, b)$ and satisfy (\ref{boundary-b}))
to create the basis for a Lagrangian subspace of $\mathbb{C}^{2n}$.
We can then proceed precisely as before, and we conclude
that the Lagrangian subspace with frame $J \gamma^*$ does
not intersect the Lagrangian subspace with frame 
$\mathbf{R}^b (\lambda)$, certainly including the elements that
correspond with solutions of (\ref{linear-hammy})
that lie right in $(c, b)$ and satisfy (\ref{boundary-b}).
\end{proof}

\begin{lemma} \label{frames-lemma}
Let Assumptions {\bf (A)}, {\bf (B)}, and {\bf (C)} hold. 
Let $\lambda_1, \lambda_2 \in \mathbb{R}$, $\lambda_1 < \lambda_2$, 
and suppose $\sigma_{\ess} (\mathcal{L}) \cap [\lambda_1, \lambda_2] 
= \emptyset$. Then for each $\lambda \in [\lambda_1, \lambda_2]$, 
the space of solutions of (\ref{linear-hammy})
that lie right in $(c, b)$ and satisfy (\ref{boundary-b}) has 
dimension $n$. If we let $\{u^b_j (x; \lambda)\}_{j=1}^n$ denote
a basis for this space, then for each $x \in (c, b)$,
the vectors $\{u^b_j (x; \lambda)\}_{j=1}^n$ comprise a basis
for a Lagrangian subspace of $\mathbb{C}^{2n}$.
\end{lemma}

\begin{proof}
We fix any $\lambda \in [\lambda_1, \lambda_2]$, and observe 
from Lemma \ref{boundary-matrix-lemma} that we can select
$\gamma \in \mathbb{C}^{n \times 2n}$ satisfying  
(\ref{boundary-matrix}) so that $\lambda$ is not 
an eigenvalue of $\mathcal{L}_{c, b}^{\gamma}$.
In addition, we know from Theorem 11.5 in \cite{Weidmann1987},
appropriately adapted to our setting, 
that $\sigma_{\ess} (\mathcal{L}_{c, b}^{\gamma}) \subset \sigma_{\ess} (\mathcal{L})$,
so we can conclude (using our assumption  
$\sigma_{\ess} (\mathcal{L}) \cap [\lambda_1, \lambda_2] 
= \emptyset$) that, in fact, $\lambda \in \rho (\mathcal{L}_{c, b}^{\gamma})$.
This last inclusion allows us to apply Theorem 7.1 in 
\cite{Weidmann1987}, which asserts (among other things) that 
the space of solutions of (\ref{linear-hammy}) that lie 
right in $(c, b)$ and satisfy (\ref{boundary-b}) has
the same dimension for each  
$\lambda \in \rho (\mathcal{L}_{c, b}^{\gamma})$. We know by 
construction that for $\lambda_0$ this dimension is 
precisely $n$, and so we can conclude that it must be
$n$ for our fixed value $\lambda \in [\lambda_1, \lambda_2]$ 
as well. We can now conclude from 
Lemma \ref{at-most-n} that this space must be 
a Lagrangian subspace of $\mathbb{C}^{2n}$ for 
each $x \in (c, b)$.
\end{proof}

\begin{lemma} \label{continuation-lemma}
Let Assumptions {\bf (A)}, {\bf (B)}, and {\bf (C)} hold, and  
suppose $\lambda_1, \lambda_2 \in \mathbb{R}$, $\lambda_1 < \lambda_2$
are such that $\sigma_{\ess} (\mathcal{L}) \cap [\lambda_1, \lambda_2] 
= \emptyset$. For some fixed $\lambda_* \in [\lambda_1, \lambda_2]$,
let $\{u^b_j (x; \lambda_*)\}_{j=1}^n$ denote a basis for the 
$n$-dimensional space of solutions of (\ref{linear-hammy})
that lie right in $(c, b)$ and satisfy (\ref{boundary-b})
(guaranteed to exist by Lemma \ref{frames-lemma}). Then 
there exists a constant $r > 0$, depending on $\lambda_*$ 
and $\mathcal{L}_{c, b}^{\gamma}$ (including the choice of $\gamma$) so that the elements 
$\{u^b_j (x; \lambda_*)\}_{j=1}^n$ can be analytically 
extended in $\lambda$ to the ball $B (\lambda_*; r)$.
Moreover, the analytic extensions $\{u^b_j (x; \lambda)\}_{j=1}^n$
comprise a basis for the space of solutions 
of (\ref{linear-hammy}) contained in $\mathcal{D}_{c, b}^{\gamma}$.
In particular, these elements lie right in $(c, b)$ and 
satisfy (\ref{boundary-b}). 
\end{lemma}

\begin{proof}
Let $\lambda_* \in [\lambda_1, \lambda_2]$ be fixed, and use 
Lemma \ref{boundary-matrix-lemma} to find a boundary matrix
$\gamma$ so that $\lambda_* \in \rho (\mathcal{L}_{c, b}^{\gamma})$.
Our extensions $\{u^b_j (x; \lambda)\}_{j=1}^n$ will 
satisfy the equation 
\begin{equation} \label{lambda-equation}
    J u_j^{b\, \prime} = (B_0 (x) + \lambda B_1 (x)) u_j^b,
\end{equation}
which we can re-write as 
\begin{equation} \label{extension-inhomogeneous}
    J u_j^{b\, \prime} - (B_0 (x) + \lambda_* B_1 (x)) u_j^b 
    = (\lambda - \lambda_*) B_1 (x) u_j^b.
\end{equation}
If a solution to (\ref{extension-inhomogeneous}) exists and 
is contained in $\mathcal{D}_{c, b}^{\gamma}$, then 
we can express it as 
\begin{equation*}
    F_j^b (x; \lambda_*, \lambda)
    = (\lambda - \lambda_*) (\mathcal{L}_{c, b}^{\gamma} - \lambda_* I)^{-1} u_j^b (\cdot; \lambda).
\end{equation*}
Here, the resolvent 
\begin{equation*}
    \mathcal{R} (\mathcal{L}_{c, b}^{\gamma}; \lambda_*) 
    := (\mathcal{L}_{c, b}^{\gamma} - \lambda_* I)^{-1}
\end{equation*}
maps elements of $L^2_{B_1} ((c, b), \mathbb{C}^{2n})$ 
into $\mathcal{D}_{c, b}^{\gamma}$, so in particular 
$F_j^b (x; \lambda_*, \lambda)$ lies right in 
$(c, b)$ and satisfies (\ref{boundary-b}). 

Clearly, $F_j^b (x; \lambda_*, \lambda_*) = 0$, 
so in order to identify an analytic extenson 
of $u^b_j (x; \lambda_*)$, we look for solutions
of (\ref{lambda-equation}) of the form 
\begin{equation} \label{look-for}
    u^b_j (x; \lambda) = u^b_j (x; \lambda_*)
    + (\lambda - \lambda_*) \mathcal{R} (\mathcal{L}_{c, b}^{\gamma}; \lambda_*) u^b_j (\cdot; \lambda).
\end{equation}
Rearranging terms, we can express this relation as 
\begin{equation} \label{neumann-type}
    (I - (\lambda - \lambda_*) \mathcal{R} (\mathcal{L}_{c, b}^{\gamma}; \lambda_*)) u^b_j (\cdot; \lambda)
    = u^b_j (\cdot; \lambda_*).
\end{equation}
By the standard theory of Neumann series (for example, the 
discussion of Example 4.9 on p. 32 of \cite{Kato}), 
if 
\begin{equation*}
 \| (\lambda - \lambda_*) \mathcal{R} (\mathcal{L}_{c, b}^{\gamma}; \lambda_*) \| < 1,   
\end{equation*}
then we can solve (\ref{neumann-type}) with 
\begin{equation} \label{neumann-series}
    u^b_j (\cdot; \lambda)
    =  (I - (\lambda - \lambda_*) \mathcal{R} (\mathcal{L}_{c, b}^{\gamma}; \lambda_*))^{-1} u^b_j (\cdot; \lambda_*).
\end{equation}
Here, $u^b_j (\cdot; \lambda) \in L^2_{B_1} ((a, b), \mathbb{C}^{2n})$ is 
analytic in $\lambda$.

Since $\lambda_* \in \rho (\mathcal{L}_{c, b}^{\gamma})$, there exists
a constant $C > 0$, depending on $\lambda_*$ and 
$\mathcal{L}_{c, b}^{\gamma}$ so that 
\begin{equation*}
    \|\mathcal{R} (\mathcal{L}_{c, b}^{\gamma}; \lambda_*)\| \le C.
\end{equation*}
In this way, we see that we can use 
(\ref{neumann-series}) so long as $|\lambda - \lambda_*| < r := 1/C$.
We conclude that (\ref{look-for}) has a unique
solution $u^b_j (\cdot; \lambda) \in L^2_{B_1} ((a, b), \mathbb{C}^{2n})$.
We've already noted that $F_j^b (x; \lambda_*, \lambda)$ is 
contained in $\mathcal{D}_{c, b}^{\gamma}$, and the same
holds for $u^b_j (\cdot; \lambda_*)$. 
We can conclude that $u^b_j (x; \lambda)$ is a solution of 
(\ref{lambda-equation}) contained in $\mathcal{D}_{c, b}^{\gamma}$. 
Proceeding similarly for
each $j \in \{1, 2, \dots, n\}$, we obtain a collection 
of extensions $\{u^b_j (x; \lambda)\}_{j=1}^n$.

In addition, by virtue of (\ref{neumann-type})-(\ref{neumann-series}),
we see that $\{u^b_j (x; \lambda)\}_{j=1}^n$ inherits 
linear independence from the set $\{u^b_j (x; \lambda_*)\}_{j=1}^n$.
We conclude from Lemma \ref{at-most-n} that 
the set $\{u^b_j (x; \lambda)\}_{j=1}^n$ comprises a 
basis for the space of solutions 
of (\ref{linear-hammy}) that lie right in $(c, b)$ and 
satisfy (\ref{boundary-b}), and additionally
that for each $x \in (c, b)$ the vectors
$\{u^b_j (x; \lambda)\}_{j=1}^n$ comprise the basis
of a Lagrangian subspace of $\mathbb{C}^{2n}$. 
\end{proof}

\begin{lemma} \label{continuity-lemma}
Let Assumptions {\bf (A)}, {\bf (B)}, and {\bf (C)} hold, and  
suppose $\lambda_1, \lambda_2 \in \mathbb{R}$, $\lambda_1 < \lambda_2$
are such that $\sigma_{\ess} (\mathcal{L}) \cap [\lambda_1, \lambda_2] 
= \emptyset$. In addition, for each $\lambda \in [\lambda_1, \lambda_2]$, 
let $\ell_b (x; \lambda)$ denote the path of Lagrangian subspaces
$\ell_b (\cdot; \lambda): (c, b) \to \Lambda (n)$ associated with 
the basis $\{u^b_j (x; \lambda)\}_{j=1}^n$ constructed in 
Lemma \ref{frames-lemma}. 
Then $\ell_b: (c, b) \times [\lambda_1, \lambda_2] \to \Lambda (n)$
is continuous. 
\end{lemma}

\begin{proof}
First, for each fixed $\lambda_* \in [\lambda_1, \lambda_2]$,
we can use Lemma \ref{continuation-lemma} to obtain a locally
analytic family of bases $\{u^b_j (x; \lambda)\}_{j=1}^n$,
for all $|\lambda - \lambda_*| < r_*$, where $r_* > 0$
is a constant depending on $\lambda_*$ (and $\mathcal{L}_{c, b}^{\gamma}$, including the boundary matrix $\gamma$). 
This process creates an open cover of $[\lambda_1, \lambda_2]$, 
created by the union of all of these disks. Next, we use 
compactness of the interval $[\lambda_1, \lambda_2]$ to extract 
a finite subcover, which we denote $\{B (\lambda_*^j; r_*^j)\}_{j=1}^N$,
where for notational convenience, we can select the values 
$\{\lambda_*^j\}_{j=1}^N$ so that 
\begin{equation*}
    \lambda_1 =: \lambda_*^1 < \lambda_*^2 < \dots < \lambda_*^N := \lambda_2,
\end{equation*}
and where the values $r_*^j > 0$ are constants respectively 
associated with the values $\lambda_*^j$ in our construction of the family of 
disks. 

Starting at $\lambda_1$, we can take $\{u^b_j (x; \lambda_1)\}_{j=1}^n$
to be a basis for the Lagrangian subspace $\ell_b (x; \lambda_1)$. As
$\lambda$ increases from $\lambda_1$, the analytic extensions
$\{u^{b, \lambda_1}_j (x; \lambda)\}_{j=1}^n$ in $B (\lambda_1, r_*^1)$ 
comprise bases for the Lagrangian paths $\ell_b (x; \lambda)$. 
By construction, the set 
$B (\lambda_1; r_*^1) \cap B (\lambda_*^2; r_*^2)$ must 
be non-empty. We take any $\lambda_*^{1, 2}$ in this intersection,
and we note that at this value of $\lambda$ the 
analytic extensions
$\{u^{b, \lambda_1}_j (x; \lambda_*^{1, 2})\}_{j=1}^n$ in $B (\lambda_1, r_*^1)$ 
serve as a basis for the same Lagrangian subspace as 
the analytic extensions 
$\{u^{b, \lambda_*^2}_j (x; \lambda_*^{1, 2})\}_{j=1}^n$ in $B (\lambda_*^2, r_*^2)$. 
This allows us to continuously switch from the
frame $\{u^{b, \lambda_1}_j (x; \lambda_*^{1, 2})\}_{j=1}^n$ to the 
frame $\{u^{b, \lambda_*^2}_j (x; \lambda_*^{1, 2})\}_{j=1}^n$. 

We now allow $\lambda$ to increase from $\lambda_*^{1, 2}$, 
and the elements $\{u^{b, \lambda_*^2}_j (x; \lambda)\}_{j=1}^n$
serve as bases for the Lagrangian subspaces $\ell_b (x; \lambda)$. 
Continuing in this way, we see that 
$\ell_b: (c, b) \times [\lambda_1, \lambda_2] \to \Lambda (n)$ 
is continuous. 
\end{proof}

\begin{remark} \label{frames-remark}
We observe that during the course of this construction, we 
have set notation for the frames associated with $\ell_b (x; \lambda)$
as $\lambda$ varies from $\lambda_1$ to $\lambda_2$. In 
particular, the interval $[\lambda_1, \lambda_2]$ has been 
partitioned into values 
\begin{equation*}
    \lambda_1 =: \lambda_*^{0, 1} < \lambda_*^{1, 2} < \lambda_*^{2, 3} 
    < \dots < \lambda_*^{N-1, N} < \lambda_*^{N, N+1} := \lambda_2,
\end{equation*}
and we use the frame $\{u_j^{b, \lambda_*^k} (x; \lambda)\}_{j=1}^n$ on 
the interval $[\lambda_*^{k-1, k}, \lambda_*^{k, k+1}]$ for all 
$k = 1, 2, \dots, N$. It's clear from the construction that 
for each $j \in \{1, 2, \dots, n\}$, $u_j^{b, \lambda_*^k} (x; \lambda)$
is analytic on $(\lambda_*^{k-1, k}, \lambda_*^{k, k+1})$. 
\end{remark}

Lemmas \ref{lagrangian-limit-lemma}--\ref{continuity-lemma}
can be stated with $(c, b)$ replaced by $(a, c)$, 
$\{u_j^b\}_{j=1}^n$ replaced by $\{u_j^a\}_{j=1}^n$,
and $\mathcal{L}_{c, b}^{\gamma}$ replaced with 
$\mathcal{L}_{a, c}^{\gamma}$, specified with 
domain 
\begin{equation*}
    \mathcal{D}_{a, c}^{\gamma} := \{y \in \mathcal{D}_{a, c, M}: 
    \lim_{x \to a^+} U^a (x; \lambda_0) J y(x) = 0, 
    \quad \gamma y (c) = 0\}.
\end{equation*}
Here, $\mathcal{D}_{a, c, M}$ denotes the domain of the 
maximal operator for (\ref{linear-hammy}) on $(a, c)$.

Under the additional assumption {\bf (A)$^\prime$}, Lemmas 
\ref{frames-lemma}, \ref{continuation-lemma}, and 
\ref{continuity-lemma} hold with $\mathcal{L}$ replaced by 
$\mathcal{L}^{\alpha}$.

\subsection{The Green's Function}
\label{green-section}

During the proof of Theorem \ref{regular-singular-theorem},
we will make brief use of a relevant Green's function, and 
for completeness we include in the current section
a full construction of this Green's 
function. 
Precisely, assuming as usual that 
$[\lambda_1, \lambda_2] \cap \sigma_{\ess} (\mathcal{L}^{\alpha}) = \emptyset$,
we fix $\lambda \in [\lambda_1, \lambda_2] \backslash \sigma_p (\mathcal{L}^{\alpha})$ 
(so, in particular, $\lambda \in \rho (\mathcal{L}^{\alpha})$),
and we construct the Green's function $G^{\alpha} (x, \xi; \lambda)$ for the equation
\begin{equation} \label{resolvent-equation}
(\mathcal{L}^{\alpha} - \lambda I) y = f.
\end{equation}
(In fact, we will only use the case $\lambda = \lambda_2$.) 
This will allow us to express the action of the 
resolvent operator
\begin{equation*}
    \mathcal{R} (\mathcal{L}^{\alpha}; \lambda)
    = (\mathcal{L}^{\alpha} - \lambda I)^{-1}
\end{equation*}
as 
\begin{equation*}
   \mathcal{R} (\mathcal{L}^{\alpha}; \lambda)f
   = \int_a^b G^{\alpha} (x, \xi; \lambda) B_1 (\xi) f(\xi) d\xi.
\end{equation*}

Equation (\ref{resolvent-equation}) is equivalent to the ODE
\begin{equation} \label{resolvent-equation-equivalent}
J y' - (B_0 (x) + \lambda B_1 (x)) y = B_1 (x) f,
\quad y \in \mathcal{D}^{\alpha},
\end{equation}
which we can solve with variation of parameters. 
For this, we let $\Phi (x; \lambda)$ denote a fundamental 
matrix for (\ref{linear-hammy}), initialized 
by $\Phi (a; \lambda) = I_{2n}$, and we look for 
solutions to (\ref{resolvent-equation-equivalent})
of the form $y(x; \lambda) = \Phi (x; \lambda) v(x; \lambda)$, 
where $v (x; \lambda)$ is a vector function to be determined.
Computing directly, we find that 
this leads to the relation $J \Phi v' = B_1 f$. 
Recalling (\ref{conj-no-conj}) (with $\lambda \in \mathbb{R}$), we see that 
\begin{equation*}
(J \Phi (x; \lambda))^{-1} = - J \Phi (x; \lambda)^*,
\end{equation*}
allowing us to write 
\begin{equation*}
v' (x; \lambda) = - J \Phi (x; \lambda)^* B_1 (x) f(x).
\end{equation*}
Upon integration, we find that 
\begin{equation*}
v(x; \lambda) = - \int_a^x J \Phi (\xi; \lambda)^* B_1 (\xi) f(\xi) d\xi
+ k (\lambda),
\end{equation*}
for some vector $k(\lambda)$ independent of $x$, and we conclude 
\begin{equation} \label{vp-solution}
y(x; \lambda) = - \Phi (x; \lambda) \int_a^x J \Phi (\xi; \lambda)^* B_1 (\xi) f(\xi) d\xi
+ \Phi (x; \lambda) k (\lambda).
\end{equation} 

In order to identify $k(\lambda)$, we impose the boundary conditions 
associated with $\mathcal{D}^{\alpha}$. First, 
for the boundary condition at $x=a$, we set $x=a$ in (\ref{vp-solution})
to see that $\alpha y(a) = 0$ becomes $\alpha k (\lambda) = 0$, which 
we can express as 
\begin{equation} \label{green-boundary1}
    (J \alpha^*)^* J k (\lambda) = 0.
\end{equation}
For the boundary condition at $b$, we have
\begin{equation} \label{boundary-b-equation}
    \lim_{x \to b^-} U^b (x; \lambda_0)^* J y (x) = 0.
\end{equation}
We see from Lemma \ref{lagrangian-limit-lemma} that 
if $y$ lies right in $(a, b)$ and satisfies 
(\ref{boundary-b-equation}) then for any 
$\lambda \in \mathbb{C}$ for which 
\begin{equation*}
    \lim_{x \to b^-} U^b (x; \lambda_0)^* J U^b (x; \lambda) = 0,
\end{equation*}
we have
\begin{equation} \label{boundary-b-equation-lambda}
    \lim_{x \to b^-} U^b (x; \lambda)^* J y (x) = 0.
\end{equation}
Here, $U^b (x; \lambda)$ is the $2n \times n$ 
matrix comprising as its columns the basis 
elements $\{u_j^b (x; \lambda)\}_{j=1}^n$ described in 
Lemma \ref{continuity-lemma}. Since these columns 
are necessarily linearly independent, there must 
exist a rank-$n$ $2n \times n$ matrix $\mathbf{R}^b (\lambda)$
so that $U^b (x; \lambda) = \Phi (x; \lambda) \mathbf{R}^b (\lambda)$.
We know from Lemma \ref{frames-lemma} that 
for $\lambda \in [\lambda_1, \lambda_2]$, 
solutions $y$ of (\ref{linear-hammy}) that lie 
right in $(a, b)$ satisfy (\ref{boundary-b-equation-lambda})
if and only if they satisfy (\ref{boundary-b-equation}). 
This allows us to work with (\ref{boundary-b-equation-lambda})
as our boundary condition at $x=b$ rather than 
(\ref{boundary-b-equation}).

We proceed now by multiplying (\ref{vp-solution})
on the left by $U^b (x; \lambda)^* J$, 
giving 
\begin{equation*}
\begin{aligned}
U^b (x; \lambda)^* J y(x; \lambda) 
&= - U^b (x; \lambda)^* J \Phi (x; \lambda) \int_a^x J \Phi (\xi; \lambda)^* B_1 (\xi) f(\xi) d\xi \\
&\quad \quad + U^b (x; \lambda)^* J \Phi (x; \lambda) k (\lambda) \\
&= \int_a^x \mathbf{R}^b (\lambda)^* \Phi (\xi; \lambda)^* B_1 (\xi) f(\xi) d\xi
+  \mathbf{R}^b (\lambda)^* J k(\lambda),
\end{aligned}
\end{equation*}
where we've used the identity (\ref{conj-no-conj}). By construction, 
$\Phi (\xi; \lambda) \mathbf{R}^b (\lambda) \in L^2_{B_1} ((a, b), \mathbb{C}^{2n})$,
so in the limit as $x \to b^-$, we obtain the relation
\begin{equation} \label{green-boundary2}
\int_a^b \mathbf{R}^b (\lambda)^* \Phi (\xi; \lambda)^* B_1 (\xi) f(\xi) d\xi
+  \mathbf{R}^b (\lambda)^* J k(\lambda) = 0.
\end{equation} 
Combining (\ref{green-boundary1}) and (\ref{green-boundary2}),
we obtain the system 
\begin{equation} \label{system-for-k}
    \begin{pmatrix}
    (J \alpha^*)^* \\ \mathbf{R}^b (\lambda)^*
    \end{pmatrix} 
    J k (\lambda)
    = \begin{pmatrix}
    0 \\ - \int_a^b \mathbf{R}^b (\lambda)^* \Phi (\xi; \lambda)^* B_1 (\xi) f(\xi) d\xi
    \end{pmatrix}. 
\end{equation}

We set 
\begin{equation*}
    \mathcal{E} (\lambda) 
    := \begin{pmatrix}
    J \alpha^* & \mathbf{R}^b (\lambda)
    \end{pmatrix},  
\end{equation*}
and we observe that if $\lambda \notin \sigma_p (\mathcal{L}^{\alpha})$
then $\mathcal{E} (\lambda)$ is invertible. This is because
$U^a (x; \lambda) = \Phi (x; \lambda) J \alpha^*$ and 
$U^b (x; \lambda) = \Phi (x; \lambda) \mathbf{R}^b (\lambda)$, so 
that 
\begin{equation*}
    U^a (x; \lambda)^* J U^b (x; \lambda)
    = (J \alpha^*)^* J \mathbf{R}^b (\lambda).
\end{equation*}
The left-hand side of this last relation is non-singular 
if and only if $\lambda \notin \sigma_p (\mathcal{L}^{\alpha})$
(because in that case the Lagrangian subspaces with 
frames $U^a (x; \lambda)$ and $U^b (x; \lambda)$ do 
not intersect), and the right-hand side of this last relation is 
non-singular if and only if $\mathcal{E} (\lambda)$
is non-singular. Accordingly, we can solve (\ref{system-for-k})
with 
\begin{equation*}
    k(\lambda) = J (\mathcal{E} (\lambda)^*)^{-1}
    \int_a^b \begin{pmatrix}
    0 & \mathbf{R}^b (\lambda) 
    \end{pmatrix}^* \Phi (\xi; \lambda)^* B_1 (\xi) f(\xi) d\xi. 
\end{equation*}
Upon substitution back into (\ref{vp-solution}),
we obtain 
\begin{equation*}
\begin{aligned}
y(x; \lambda) &= - \Phi (x; \lambda) \int_a^x J \Phi (\xi; \lambda)^* B_1 (\xi) f(\xi) d\xi \\
&+ \Phi (x; \lambda) J (\mathcal{E} (\lambda)^*)^{-1}
    \int_a^b \begin{pmatrix}
    0 & \mathbf{R}^b (\lambda)
    \end{pmatrix}^* \Phi (\xi; \lambda)^* B_1 (\xi) f(\xi) d\xi \\
&= - \Phi (x; \lambda) J (\mathcal{E} (\lambda)^*)^{-1} \mathcal{E} (\lambda)^* \int_a^x \Phi (\xi; \lambda)^* B_1 (\xi) f(\xi) d\xi \\
&+ \Phi (x; \lambda) J (\mathcal{E} (\lambda)^*)^{-1}
    \int_a^b \begin{pmatrix}
    0 & \mathbf{R}^b (\lambda)
    \end{pmatrix}^* \Phi (\xi; \lambda)^* B_1 (\xi) f(\xi) d\xi.
\end{aligned}
\end{equation*} 
Continuing with this calculation, we next see that 
\begin{equation*}
    \begin{aligned}
    y(x; \lambda) &= - \Phi (x; \lambda) J (\mathcal{E} (\lambda)^*)^{-1} 
    \begin{pmatrix}
    J \alpha^* & 0
    \end{pmatrix}^*
    \int_a^x \Phi (\xi; \lambda)^* B_1 (\xi) f(\xi) d\xi \\
    &- \Phi (x; \lambda) J (\mathcal{E} (\lambda)^*)^{-1} 
    \begin{pmatrix}
    0 & \mathbf{R}^b (\lambda)
    \end{pmatrix}^*
    \int_a^x \Phi (\xi; \lambda)^* B_1 (\xi) f(\xi) d\xi \\
    &+ \Phi (x; \lambda) J (\mathcal{E} (\lambda)^*)^{-1}
    \begin{pmatrix}
    0 & \mathbf{R}^b (\lambda)
    \end{pmatrix}^*
    \int_a^b \Phi (\xi; \lambda)^* B_1 (\xi) f(\xi) d\xi \\
    &= - \Phi (x; \lambda) J (\mathcal{E} (\lambda)^*)^{-1} 
    \begin{pmatrix}
    J \alpha^* & 0
    \end{pmatrix}^*
    \int_a^x \Phi (\xi; \lambda)^* B_1 (\xi) f(\xi) d\xi \\
    &+ \Phi (x; \lambda) J (\mathcal{E} (\lambda)^*)^{-1}
    \begin{pmatrix}
    0 & \mathbf{R}^b (\lambda)
    \end{pmatrix}^*
    \int_x^b \Phi (\xi; \lambda)^* B_1 (\xi) f(\xi) d\xi. 
    \end{aligned}
\end{equation*}

We see by inspection that 
\begin{equation*}
    G^{\alpha} (x, \xi; \lambda)
    = \begin{cases}
    - \Phi (x; \lambda) J (\mathcal{E} (\lambda)^*)^{-1} 
    \begin{pmatrix}
    J \alpha^* & 0
    \end{pmatrix}^*
    \Phi (\xi; \lambda)^* & a < \xi < x < b \\
    \Phi (x; \lambda) J (\mathcal{E} (\lambda)^*)^{-1}
    \begin{pmatrix}
    0 & \mathbf{R}^b (\lambda)
    \end{pmatrix}^*
    \Phi (\xi; \lambda)^* & a < x < \xi < b.
    \end{cases}
\end{equation*}

We can express $G^{\alpha} (x, \xi; \lambda)$ in a more symmetric form. 
To see this, we first observe that 
\begin{equation*}
    \begin{aligned}
    \mathcal{E} (\lambda)^* J \mathcal{E} (\lambda)
    &= \begin{pmatrix}
    - \alpha J \\ \mathbf{R}^b (\lambda)^*
    \end{pmatrix} J
    \begin{pmatrix}
    J \alpha^* & \mathbf{R}^b (\lambda)
    \end{pmatrix} \\
    &= \begin{pmatrix}
    \alpha J \alpha^* & \alpha \mathbf{R}^b (\lambda) \\
    - \mathbf{R}^b (\lambda)^* \alpha^* & \mathbf{R}^b (\lambda)^* J \mathbf{R}^b (\lambda)
    \end{pmatrix}
    = \begin{pmatrix}
    0 & \alpha \mathbf{R}^b (\lambda) \\
    - (\alpha \mathbf{R}^b (\lambda))^*  & 0
    \end{pmatrix},
    \end{aligned}
\end{equation*}
where we've used the observations that $J \alpha^*$ and 
$\mathbf{R}^b (\lambda)$ are frames for Lagrangian subspaces of 
$\mathbb{C}^{2n}$. Here, $\alpha \mathbf{R}^b (\lambda) = (J \alpha^*)^* J \mathbf{R}^b (\lambda)$,
and we've already seen that this matrix is non-singular so long
as $\lambda \notin \sigma_p (\mathcal{L}^{\alpha})$. This allows 
us to write 
\begin{equation} \label{inverse-matrix}
   (\mathcal{E} (\lambda)^* J \mathcal{E} (\lambda))^{-1}
   = \begin{pmatrix}
    0 & - ((\alpha \mathbf{R}^b (\lambda))^*)^{-1} \\
    (\alpha \mathbf{R}^b (\lambda))^{-1}  & 0
    \end{pmatrix}.
\end{equation}
It follows that 
\begin{equation*}
    \begin{aligned}
    - &\begin{pmatrix}
    J \alpha^* & 0
    \end{pmatrix}
    \mathcal{E} (\lambda)^{-1} J (\mathcal{E} (\lambda)^*)^{-1} 
    \begin{pmatrix}
    0 & \mathbf{R}^b (\lambda)
    \end{pmatrix}^* \\
    & = \begin{pmatrix}
    J \alpha^* & 0
    \end{pmatrix}
    \begin{pmatrix}
    0 & - ((\alpha \mathbf{R}^b (\lambda))^*)^{-1} \\
    (\alpha \mathbf{R}^b (\lambda))^{-1}  & 0
    \end{pmatrix}
    \begin{pmatrix}
    0 \\ \mathbf{R}^b (\lambda)^*
    \end{pmatrix} \\
    &= - \begin{pmatrix}
    J \alpha^* & 0
    \end{pmatrix}
    \begin{pmatrix}
    ((\alpha \mathbf{R}^b (\lambda))^*)^{-1} \mathbf{R}^b (\lambda)^* \\ 0
    \end{pmatrix}
    = - (J \alpha^*) (\alpha \mathbf{R}^b (\lambda)^*)^{-1} \mathbf{R}^b (\lambda)^*.
    \end{aligned}
\end{equation*}

On the other hand, (\ref{inverse-matrix}) also 
allows us to write 
\begin{equation*}
    (\mathcal{E} (\lambda)^*)^{-1}
    = J \mathcal{E} (\lambda)
    \begin{pmatrix}
    0 & - ((\alpha \mathbf{R}^b (\lambda))^*)^{-1} \\
    (\alpha \mathbf{R}^b (\lambda))^{-1}  & 0
    \end{pmatrix},
\end{equation*}
from which we see that 
\begin{equation*}
\begin{aligned}
(\mathcal{E} (\lambda)^*)^{-1}  
    &\begin{pmatrix}
    0 & \mathbf{R}^b (\lambda)
    \end{pmatrix}^* 
= J \mathcal{E} (\lambda)
    \begin{pmatrix}
    0 & - ((\alpha \mathbf{R}^b (\lambda))^*)^{-1} \\
    (\alpha \mathbf{R}^b (\lambda))^{-1}  & 0
    \end{pmatrix}
    \begin{pmatrix}
    0 \\ \mathbf{R}^b (\lambda)^*
    \end{pmatrix} \\
    &= J 
    \begin{pmatrix}
    J \alpha^* & \mathbf{R}^b (\lambda)
    \end{pmatrix}
    \begin{pmatrix}
    - ((\alpha \mathbf{R}^b (\lambda))^*)^{-1} \mathbf{R}^b (\lambda)^* \\ 0
    \end{pmatrix}
    = \alpha^* ((\alpha \mathbf{R}^b (\lambda))^*)^{-1} \mathbf{R}^b (\lambda)^*.
\end{aligned}
\end{equation*}
In this way, we see that 
\begin{equation*}
    J (\mathcal{E} (\lambda)^*)^{-1}  
    \begin{pmatrix}
    0 & \mathbf{R}^b (\lambda)
    \end{pmatrix}^* 
    =  \begin{pmatrix}
    J \alpha^* & 0
    \end{pmatrix} 
    \mathcal{E} (\lambda)^{-1} 
    J (\mathcal{E} (\lambda)^*)^{-1}  
    \begin{pmatrix}
    0 & \mathbf{R}^b (\lambda)
    \end{pmatrix}^*.
\end{equation*}
We will set 
\begin{equation*}
    \mathcal{M} (\lambda) := \mathcal{E} (\lambda)^{-1} J (\mathcal{E} (\lambda)^*)^{-1}, 
\end{equation*}
from which we observe that 
\begin{equation*}
    \mathcal{M} (\lambda)^* = - \mathcal{M} (\lambda).
\end{equation*}
For $a < x < \xi < b$, we will re-write 
$G^{\alpha} (x, \xi; \lambda)$ by using the relation 
\begin{equation*}
    J (\mathcal{E} (\lambda)^*)^{-1}  
    \begin{pmatrix}
    0 & \mathbf{R}^b (\lambda)
    \end{pmatrix}^* 
    = \begin{pmatrix}
    J \alpha^* & 0
    \end{pmatrix} 
    \mathcal{M} (\lambda)
    \begin{pmatrix}
    0 & \mathbf{R}^b (\lambda)
    \end{pmatrix}^*.
\end{equation*}

Proceeding similarly for $a < \xi < x < b$, we find  
\begin{equation*}
    J (\mathcal{E} (\lambda)^*)^{-1}  
    \begin{pmatrix}
    J \alpha^* & 0
    \end{pmatrix}^* 
    = \begin{pmatrix}
    0 & \mathbf{R}^b (\lambda)
    \end{pmatrix} 
    \mathcal{M} (\lambda)
    \begin{pmatrix}
    J \alpha^* & 0
    \end{pmatrix}^*.
\end{equation*}
These relations allow us to express $G^{\alpha} (x, \xi; \lambda)$
as 
\begin{equation*}
    G^{\alpha} (x, \xi; \lambda)
    = \begin{cases}
    - \Phi (x; \lambda) 
    \begin{pmatrix}
    0 & \mathbf{R}^b (\lambda)
    \end{pmatrix} 
    \mathcal{M} (\lambda)
    \begin{pmatrix}
    J \alpha^* & 0
    \end{pmatrix}^*
    \Phi (\xi; \lambda)^* & a < \xi < x < b \\
    \Phi (x; \lambda) 
    \begin{pmatrix}
    J \alpha^* & 0
    \end{pmatrix} 
    \mathcal{M} (\lambda)
    \begin{pmatrix}
    0 & \mathbf{R}^b (\lambda)
    \end{pmatrix}^*
    \Phi (\xi; \lambda)^* & a < x < \xi < b.
    \end{cases}
\end{equation*}

\section{The Maslov Index} \label{maslov-section}

Our framework for computing the Maslov index is adapted from 
Section 2 of \cite{HS2}, and we briefly sketch the main ideas
here. Given any pair of Lagrangian subspaces $\ell_1$ and 
$\ell_2$ with respective frames $\mathbf{X}_1 = {X_1 \choose Y_1}$
and $\mathbf{X}_2 = {X_2 \choose Y_2}$, we consider the matrix
\begin{equation} \label{tildeW}
\tilde{W} := - (X_1 + iY_1)(X_1-iY_1)^{-1} (X_2 - iY_2)(X_2+iY_2)^{-1}. 
\end{equation}
In \cite{HS2}, the authors establish: (1) the inverses 
appearing in (\ref{tildeW}) exist; (2) $\tilde{W}$ is independent
of the specific frames $\mathbf{X}_1$ and $\mathbf{X}_2$ (as long
as these are indeed frames for $\ell_1$ and $\ell_2$); (3) $\tilde{W}$ is 
unitary; and (4) the identity 
\begin{equation} \label{key}
\dim (\ell_1 \cap \ell_2) = \dim (\ker (\tilde{W} + I)).
\end{equation}
Given two continuous paths of Lagrangian subspaces 
$\ell_i: [0, 1] \to \Lambda (n)$, $i = 1, 2$, with 
respective frames $\mathbf{X}_i: [0,1] \to \mathbb{C}^{2n \times n}$,
relation (\ref{key}) allows us to compute the Maslov 
index $\mas (\ell_1, \ell_2; [0,1])$ as a spectral flow
through $-1$ for the path of matrices 
\begin{equation} 
\tilde{W} (t) := - (X_1 (t) + iY_1 (t))(X_1 (t)-iY_1 (t))^{-1} 
(X_2 (t) - iY_2 (t))(X_2 (t)+iY_2 (t))^{-1}. 
\end{equation}

In \cite{HS2}, the authors provide a rigorous definition 
of the Maslov index based on the spectral flow developed 
in \cite{P96}. Here, rather, we give only an intuitive 
discussion. As a starting point, 
if $-1 \in \sigma (\tilde{W} (t_*))$ for some $t_* \in [0, 1]$, 
then we refer to $t_*$ as a conjugate point, and its 
multiplicity is taken to be $\dim (\ell_1 (t_*) \cap \ell_2 (t_*))$, 
which by virtue of (\ref{key}) is equivalent to its 
multiplicity as an eigenvalue of $\tilde{W} (t_*)$. 
We compute the Maslov index $\mas (\ell_1, \ell_2; [0, 1])$ 
by allowing $t$ to increase from $0$ to $1$ and incrementing 
the index whenever an eigenvalue crosses $-1$ in the 
counterclockwise direction, while decrementing the index
whenever an eigenvalue crosses $-1$ in the clockwise
direction. These increments/decrements are counted with 
multiplicity, so for example, if a pair of eigenvalues 
crosses $-1$ together in the counterclockwise direction, 
then a net amount of $+2$ is added to the index. Regarding
behavior at the endpoints, if an eigenvalue of $\tilde{W}$
rotates away from $-1$ in the clockwise direction as $t$ increases
from $0$, then the Maslov index decrements (according to 
multiplicity), while if an eigenvalue of $\tilde{W}$
rotates away from $-1$ in the counterclockwise direction as $t$ increases
from $0$, then the Maslov index does not change. Likewise, 
if an eigenvalue of $\tilde{W}$ rotates into $-1$ in the 
counterclockwise direction as $t$ increases
to $1$, then the Maslov index increments (according to 
multiplicity), while if an eigenvalue of $\tilde{W}$
rotates into $-1$ in the clockwise direction as $t$ increases
to $1$, then the Maslov index does not change. Finally, 
it's possible that an eigenvalue of $\tilde{W}$ will arrive 
at $-1$ for $t = t_*$ and remain at $-1$ as $t$ traverses
an interval. In these cases, the 
Maslov index only increments/decrements upon arrival or 
departure, and the increments/decrements are determined 
as for the endpoints (departures determined as with $t=0$,
arrivals determined as with $t = 1$).

One of the most important features of the Maslov index is homotopy invariance, 
for which we need to consider continuously varying families of Lagrangian 
paths. To set some notation, we denote by $\mathcal{P} (\mathcal{I})$ the collection 
of all paths $\mathcal{L} (t) = (\ell_1 (t), \ell_2 (t))$, where 
$\ell_1, \ell_2: \mathcal{I} \to \Lambda (n)$ are continuous paths in the 
Lagrangian--Grassmannian. We say that two paths 
$\mathcal{L}, \mathcal{M} \in \mathcal{P} (\mathcal{I})$ are homotopic provided 
there exists a family $\mathcal{H}_s$ so that 
$\mathcal{H}_0 = \mathcal{L}$, $\mathcal{H}_1 = \mathcal{M}$, 
and $\mathcal{H}_s (t)$ is continuous as a map from $(t,s) \in \mathcal{I} \times [0,1]$
into $\Lambda (n) \times \Lambda (n)$. 

The Maslov index has the following properties. 

\medskip
\noindent
{\bf (P1)} (Path Additivity) If $\mathcal{L} \in \mathcal{P} (\mathcal{I})$
and $a, b, c \in \mathcal{I}$, with $a < b < c$, then 
\begin{equation*}
\mas (\mathcal{L};[a, c]) = \mas (\mathcal{L};[a, b]) + \mas (\mathcal{L}; [b, c]).
\end{equation*}

\medskip
\noindent
{\bf (P2)} (Homotopy Invariance) If $\mathcal{L}, \mathcal{M} \in \mathcal{P} (\mathcal{I})$ 
are homotopic, with $\mathcal{L} (a) = \mathcal{M} (a)$ and  
$\mathcal{L} (b) = \mathcal{M} (b)$ (i.e., if $\mathcal{L}, \mathcal{M}$
are homotopic with fixed endpoints) then 
\begin{equation*}
\mas (\mathcal{L};[a, b]) = \mas (\mathcal{M};[a, b]).
\end{equation*} 

Straightforward proofs of these properties appear in \cite{HLS1}
for Lagrangian subspaces of $\mathbb{R}^{2n}$, and proofs in the current setting of 
Lagrangian subspaces of $\mathbb{C}^{2n}$ are essentially identical.

\subsection{Direction of Rotation} \label{rotation_section}

As noted previously, the direction we associate with a 
conjugate point is determined by the direction in which eigenvalues
of $\tilde{W}$ rotate through $-1$ (counterclockwise is positive, 
while clockwise is negative). In this subsection, we review
the framework developed in \cite{HS2} for analyzing this direction.
Our starting point is the following lemma from \cite{HS2}. 

\begin{lemma} \label{monotonicity1}
Suppose $\ell_1, \ell_2: \mathcal{I} \to \Lambda (n)$ denote paths of 
Lagrangian subspaces of $\mathbb{C}^{2n}$ with absolutely 
continuous frames $\mathbf{X}_1 = {X_1 \choose Y_1}$
and $\mathbf{X}_2 = {X_2 \choose Y_2}$ (respectively). If there exists 
$\delta > 0$ so that the matrices 
\begin{equation*}
- \mathbf{X}_1 (t)^* J \mathbf{X}_1' (t) = X_1 (t)^* Y_1' (t) - Y_1 (t)^* X_1'(t)
\end{equation*}
and (noting the sign change)
\begin{equation*}
\mathbf{X}_2 (t)^* J \mathbf{X}_2' (t) = - (X_2 (t)^* Y_2' (t) - Y_2 (t)^* X_2'(t))
\end{equation*}
are both a.e.-non-negative in $(t_0-\delta,t_0+\delta)$, and at least one is 
a.e.-positive definite in $(t_0-\delta,t_0+\delta)$ then the eigenvalues of 
$\tilde{W} (t)$ rotate in the counterclockwise direction as $t$ increases through $t_0$. 
Likewise, if both of these matrices are a.e.-non-positive, and at least one is 
a.e.-negative definite, 
then the eigenvalues of $\tilde{W} (t)$ rotate in the clockwise direction as 
$t$ increases through $t_0$.
\end{lemma}

\begin{remark} The corresponding statement Lemma 4.2 in \cite{HLS1} is stated
in the slightly more restrictive case in which the frames are continuously
differentiable. 
\end{remark}

For our applications to linear Hamiltonian systems, Lemma \ref{monotonicity1} is 
generally all we need to establish monotonicity in the spectral parameter. However, 
for monotonicity as the independent variable varies, we typically require 
additional information. 

Our primary interest is with solutions of (\ref{linear-hammy}), so (suppressing the spectral
parameter for the moment) let 
$\ell_1 (t)$ and $\ell_2 (t)$ denote Lagrangian subspaces 
with respective frames 
\begin{equation*}
\mathbf{X}_1 (t) = {X_1 (t) \choose Y_1 (t)}; \quad \mathbf{X}_2 (t) = {X_2 (t) \choose Y_2 (t)}, 
\end{equation*}
satisfying 
\begin{equation*}
\begin{aligned}
J \mathbf{X}_1' &= \mathbb{B}_1 (t) \mathbf{X}_1 \\
J \mathbf{X}_2' &= \mathbb{B}_2 (t) \mathbf{X}_2, 
\end{aligned}
\end{equation*}
where $\mathbb{B}_1, \mathbb{B}_2 \in L^1 ((a,b);\mathbb{C}^{2n \times 2n})$, $a < b$,
are paths of self-adjoint matrices. 

In this setting, we have the following lemma from \cite{HS2}. 

\begin{lemma} \label{rotation_lemma}
Suppose $\mathbb{B}_1, \mathbb{B}_2 \in L^1 ((a,b);\mathbb{C}^{2n \times 2n})$,
with $\mathbb{B}_1 (t), \mathbb{B}_2 (t)$ self-adjoint for a.e. $t \in (a, b)$,
and let $\ell_1 (t)$ and $\ell_2 (t)$ be Lagrangian subspaces with 
respective frames $\mathbf{X}_1 (t)$ and $\mathbf{X}_2 (t)$ satisfying 
\begin{equation*}
J \mathbf{X}_i' = \mathbb{B}_i (t) \mathbf{X}_i (t); \quad t \in [a, b], 
\quad i = 1, 2.
\end{equation*}
Let $t_* \in [a,b]$ be a conjugate point for $\ell_1 (t)$ and $\ell_2 (t)$
so that $\dim (\ell_1 (t_*) \cap \ell_2 (t_*)) = m \in \mathbb{N}$, and let $\mathbb{P}_*$
denote projection onto $\ell_1 (t_*) \cap \ell_2 (t_*)$. Fix $\delta_0 > 0$
sufficiently small so that $t_*$ is the only conjugate point for 
$\ell_1 (t)$ and $\ell_2 (t)$ on $(t_* - \delta_0, t_* + \delta_0)$. If 
there exists $0 < \delta < \delta_0$ so that 
$\mathbb{P}_* (\mathbb{B}_2 (t) - \mathbb{B}_1 (t)) \mathbb{P}_*$
has $m_-$ a.e.-negative eigenvalues on $(t_* - \delta, t_*+\delta) \cap [a,b]$,
and $m_+$ a.e.-positive eigenvalues on $(t_* - \delta, t_*+\delta) \cap [a,b]$,
and if in addition $m_- + m_+ = m$, then the following hold: 

\smallskip
\noindent
(i) if $t_* \in (a,b)$,  
\begin{equation*}
\mas (\ell_1, \ell_2; [t_* - \delta, t_* + \delta]) = m_+ - m_-;
\end{equation*}  

\noindent
(ii) If $t_*=a$, then 
\begin{equation*}
\mas (\ell_1, \ell_2; [a, a+ \delta]) = - m_-;
\end{equation*}

\noindent
(iii) If $t_* = b$, then 
\begin{equation*}
\mas (\ell_1, \ell_2; [b - \delta, b]) = m_+.  
\end{equation*}
\end{lemma}

\begin{remark}
We emphasize the assumption in Lemma \ref{rotation_lemma} 
that $\mathbb{B}_1$ and $\mathbb{B}_2$ are in the space 
$L^1 ((a, b), \mathbb{C}^{2n \times 2n})$,
rather than $L^1_{\loc} ((a, b), \mathbb{C}^{2n \times 2n})$. In the current 
setting, this means that the lemma can be applied on subintervals
$[c, d] \subset (a, b)$.
\end{remark}

\section{Proofs of the Main Theorems} \label{theorems-section}

In this section, we use our Maslov index framework to prove Theorems 
\ref{regular-singular-theorem} and \ref{singular-theorem}.

\subsection{Proof of Theorem \ref{regular-singular-theorem}} 
\label{regular-singular-section}

Fix any pair $\lambda_1, \lambda_2 \in \mathbb{R}$, $\lambda_1 < \lambda_2$,
so that $\sigma_{\ess} (\mathcal{L}^{\alpha}) \cap [\lambda_1, \lambda_2] = \emptyset$,
and let $\ell_{\alpha} (x; \lambda)$ denote the map of Lagrangian subspaces
associated with the frames $\mathbf{X}_{\alpha} (x; \lambda)$ specified 
in (\ref{frame-alpha}). Keeping in mind that $\lambda_2$ is fixed, let 
$\ell_b (x; \lambda_2)$ denote the map of Lagrangian subspaces
associated with the frames $\mathbf{X}_b (x; \lambda_2)$ specified 
in (\ref{frame-b}). We emphasize that since $\lambda_2$ is fixed we
don't yet require Lemma \ref{continuity-lemma} to extend the 
frame $\mathbf{X}_b (x; \lambda_2)$ to additional values
$\lambda \in [\lambda_1, \lambda_2]$. 
We will establish Theorem \ref{regular-singular-theorem} by 
considering the Maslov index for $\ell_{\alpha} (x; \lambda)$ and 
$\ell_b (x; \lambda_2)$ along a path designated as the 
{\it Maslov box} in the next paragraph. As described in 
Section \ref{maslov-section}, this Maslov index is computed as
a spectral flow for the matrix 
\begin{equation} \label{tildeW-bc1}
\begin{aligned}
\tilde{W} (x; \lambda) &= - (X_{\alpha} (x; \lambda) + i Y_{\alpha} (x; \lambda))
(X_{\alpha} (x; \lambda) - i Y_{\alpha} (x; \lambda))^{-1} \\
& \times (X_b (x; \lambda_2) - i Y_b (x; \lambda_2))
(X_b (x; \lambda_2) + i Y_b (x; \lambda_2))^{-1}.
\end{aligned}
\end{equation}

By Maslov Box, in this case we mean the following sequence of contours, 
specified for some value $c \in (a, b)$ to be chosen sufficiently  
close to $b$ during the analysis (sufficiently large if $b = + \infty$):
(1) fix $x = a$ and let $\lambda$ increase from $\lambda_1$ to $\lambda_2$ 
(the {\it bottom shelf}); 
(2) fix $\lambda = \lambda_2$ and let $x$ increase from $a$ to $c$ 
(the {\it right shelf}); (3) fix $x = c$ and let $\lambda$
decrease from $\lambda_2$ to $\lambda_1$ (the {\it top shelf}); and (4) fix
$\lambda = \lambda_1$ and let $x$ decrease from $c$ to $a$ (the 
{\it left shelf}). 

{\it Right shelf.}
We begin our analysis with the right shelf, for which $\mathbf{X}_{\alpha}$ 
and $\mathbf{X}_b$ are both evaluated at $\lambda_2$. By construction, 
$\ell_{\alpha} (\cdot; \lambda_2)$ will intersect $\ell_b (\cdot; \lambda_2)$
at some $x$ (and so for all $x \in [a, c]$) with dimension $m$ if and only if $\lambda_2$ is an 
eigenvalue of $\mathcal{L}^{\alpha}$ with multiplicity $m$. In the event that $\lambda_2$
is not an eigenvalue of $\mathcal{L}^{\alpha}$, there will be no 
conjugate points along the right shelf. On the other hand, if
$\lambda_2$ is an eigenvalue of $\mathcal{L}^{\alpha}$ with multiplicity
$m$, then $\tilde{W} (x; \lambda_2)$ will have $-1$ as an eigenvalue
with multiplicity $m$ for all $x \in [a, c]$. In either case,
\begin{equation} 
\mas (\ell_{\alpha} (\cdot; \lambda_2), \ell_b (\cdot; \lambda_2); [a, c]) = 0.
\end{equation} 

{\it Bottom shelf.} For the bottom shelf, $\ell_{\alpha} (a; \lambda)$ is fixed, 
independent of $\lambda$, so in particular $\ell_{\alpha} (a; \lambda) = \ell_{\alpha} (a; \lambda_2)$ 
for all $\lambda \in [\lambda_1, \lambda_2]$. In this way, $\tilde{W} (a; \lambda)$
is actually independent of $\lambda$, and so we certainly have 
\begin{equation} 
\mas (\ell_{\alpha} (a; \cdot), \ell_b (a; \lambda_2); [\lambda_1, \lambda_2]) = 0.
\end{equation} 
Moreover, $\ell_{\alpha} (a; \lambda)$ 
will intersect $\ell_b (a; \lambda_2)$ with intersection dimension $m$ if and only 
if $\lambda_2$ is an eigenvalue of $\mathcal{L}^{\alpha}$ with multiplicity $m$. In the event 
that $\lambda_2$ is not an eigenvalue of $\mathcal{L}^{\alpha}$, there will be no 
conjugate points along the bottom shelf. On the other hand, if
$\lambda_2$ is an eigenvalue of $\mathcal{L}^{\alpha}$ with multiplicity
$m$, then $\tilde{W} (a; \lambda)$ will have $-1$ as an eigenvalue
with multiplicity $m$ for all $\lambda \in [\lambda_1, \lambda_2]$. 

{\it Top shelf.} For the top shelf, $\tilde{W} (c; \lambda)$ detects intersections
between $\ell_{\alpha} (c; \lambda)$ and $\ell_b (c; \lambda_2)$ as $\lambda$
decreases from $\lambda_2$ to $\lambda_1$. In this way, intersections
correspond precisely with eigenvalues of the finite-interval 
(or {\it truncated}) operator $\mathcal{L}_{a, c}^{\alpha}$, 
with domain 
\begin{equation*}
    \mathcal{D}_{a, c}^{\alpha} := 
    \{y \in \mathcal{D}_{a, c, M}: \alpha y(a) = 0, 
    \quad \mathbf{X}_b (c; \lambda_2)^* J y (c) = 0 \},
\end{equation*}
where $\mathcal{D}_{a, c, M}$ denotes the domain of the maximal
operator specified as in Definition \ref{maximal-operator}, 
except on $(a, c)$.
Similarly as in Section \ref{operator-section}, we can check that 
$\mathcal{L}_{a, c}^{\alpha}$ is a self-adjoint operator. (In fact,
since $\mathcal{L}_{a, c}^{\alpha}$ is posed on a bounded interval
$(a, c)$ with $B_0, B_1 \in L^1 ((a, c), \mathbb{C}^{2n \times 2n})$, 
self-adjointness can be established by more routine considerations.)

We know from Lemma \ref{monotonicity1} that monotonicity
in $\lambda$ is determined by 
$- \mathbf{X}_{\alpha} (c; \lambda)^* J \partial_{\lambda} \mathbf{X}_{\alpha} (c; \lambda)$, 
and we readily compute 
\begin{equation*}
\begin{aligned}
\frac{\partial}{\partial x} \mathbf{X}_{\alpha}^* &(x; \lambda) J \partial_{\lambda} \mathbf{X}_{\alpha} (x; \lambda)
= \mathbf{X}_{\alpha}^{\prime} (x; \lambda)^* J \partial_{\lambda} \mathbf{X}_{\alpha} (x; \lambda) 
+ \mathbf{X}_{\alpha} (x; \lambda)^* J \partial_{\lambda} \mathbf{X}_{\alpha}^{\prime} (x; \lambda) \\
&= - \mathbf{X}_{\alpha}^{\prime}(x; \lambda)^* J^* \partial_{\lambda} \mathbf{X}_{\alpha} (x; \lambda) 
+ \mathbf{X}_{\alpha} (x; \lambda)^* \partial_{\lambda} J \mathbf{X}_{\alpha}^{\prime} (x; \lambda) \\
&= - \mathbf{X}_{\alpha} (x; \lambda)^* (B_0 (x) + \lambda B_1 (x)) \partial_{\lambda} \mathbf{X}_{\alpha} (x; \lambda) 
+ \mathbf{X}_{\alpha} (x; \lambda)^* (B_0 (x) + \lambda B_1 (x)) \partial_{\lambda} \mathbf{X}_{\alpha} (x; \lambda) \\
&+ \mathbf{X}_{\alpha}^* \partial_{\lambda} (B_0 (x) + \lambda B_1 (x)) \mathbf{X}_{\alpha} (x; \lambda) 
= \mathbf{X}_{\alpha} (x; \lambda)^* B_1 (x) \mathbf{X}_{\alpha} (x; \lambda).
\end{aligned}
\end{equation*}
Integrating on $[a,x]$, and noting that $\partial_{\lambda} \mathbf{X}_{\alpha} (a;\lambda) = 0$,
we see that   
\begin{equation*}
\mathbf{X}_{\alpha} (x; \lambda)^* J \partial_{\lambda} \mathbf{X}_{\alpha} (x; \lambda)
= \int_a^x \mathbf{X}_{\alpha} (y;\lambda)^* B_1 (y) \mathbf{X}_{\alpha} (y; \lambda) dy. 
\end{equation*}
Monotonicity along the top shelf follows by setting $x = c$ and appealing 
to Assumption {\bf (B)}. 
In this way, we see that Assumption {\bf (B)} ensures that 
as $\lambda$ increases the eigenvalues of $\tilde{W} (c; \lambda)$ will 
rotate monotonically in the clockwise direction. Since each crossing along the top shelf
corresponds with an eigenvalue of $\mathcal{L}_{a, c}^{\alpha}$, we can conclude that 
\begin{equation} \label{truncated-count-alpha}
\mathcal{N}_{a, c}^{\alpha} ([\lambda_1, \lambda_2)) = 
- \mas (\ell_{\alpha} (c; \cdot), \ell_b (c; \lambda_2); [\lambda_1, \lambda_2]),
\end{equation}
where $\mathcal{N}_{a, c}^{\alpha} ([\lambda_1, \lambda_2))$ denotes a count, including
multiplicities, of the eigenvalues of $\mathcal{L}_{a, c}^{\alpha}$ on 
$[\lambda_1, \lambda_2)$.
We note that $\lambda_1$ is included in the count, because in the event
that $(c, \lambda_1)$ is conjugate, eigenvalues of 
$\tilde{W} (c; \lambda)$ will rotate away from $-1$ in the clockwise 
direction as $\lambda$ increases from $\lambda_1$ (thus decrementing the 
Maslov index). Likewise, $\lambda_2$ is not included in the count, because in the event
that $(c, \lambda_2)$ is conjugate, eigenvalues of 
$\tilde{W} (c; \lambda)$ will rotate into $-1$ in the clockwise 
direction as $\lambda$ increases to $\lambda_2$ (thus leaving the 
Maslov index unchanged). 

\begin{remark} We note that monotonicity in $\lambda$ at any
shelf $x \in (a,c]$ also follows from Assumption {\bf (B)}, 
and indeed this fact is important in the proof of 
Theorem \ref{nullity-theorem} (see \cite{HS2}).
\end{remark}

{\it Left shelf.} 
Our analysis so far leaves only the left shelf to consider, and 
we observe that it can be expressed as 
\begin{equation*}
- \mas (\ell_{\alpha} (\cdot; \lambda_1), \ell_b (\cdot; \lambda_2); [a, c]),
\end{equation*} 
which is part of the Maslov index that appears in the statement of 
Theorem \ref{regular-singular-theorem}.
Using path additivity and homotopy invariance, we can sum the Maslov
indices on each shelf of the Maslov Box to arrive at the relation 
\begin{equation} \label{box-sum-alpha}
\mathcal{N}_{a, c}^{\alpha} ([\lambda_1, \lambda_2)) 
= \mas (\ell_{\alpha} (\cdot; \lambda_1), \ell_b (\cdot; \lambda_2); [a, c]).
\end{equation}

In order to obtain a statement about $\mathcal{N}^{\alpha} ([\lambda_1, \lambda_2))$, 
we observe that eigenvalues of $\mathcal{L}^{\alpha}$ correspond precisely 
with intersections of $\ell_{\alpha} (c; \lambda)$
and $\ell_b (c; \lambda)$. (We emphasize that in this last statement, 
$\ell_b$ is evaluated at $\lambda$, not $\lambda_2$, and so we 
are using Lemma \ref{continuity-lemma} ). Employing a monotonicity
argument similar to the one above, we can conclude that 
\begin{equation} \label{full-count-alpha}
\mathcal{N}^{\alpha} ([\lambda_1, \lambda_2)) = 
- \mas (\ell_{\alpha} (c; \cdot), \ell_b (c; \cdot); [\lambda_1, \lambda_2]).
\end{equation}

\begin{claim} \label{triangle-claim-alpha}
Under the assumptions of Theorem \ref{regular-singular-theorem}, and for any $c \in (a, b)$,
\begin{equation*}
\begin{aligned}
\mas (\ell_{\alpha} (c; \cdot), \ell_b (c; \cdot); [\lambda_1, \lambda_2])
& =
\mas (\ell_{\alpha} (c; \lambda_1), \ell_b (c; \cdot); [\lambda_1, \lambda_2]) \\
&+ \mas (\ell_{\alpha} (c; \cdot), \ell_b (c; \lambda_2); [\lambda_1, \lambda_2]). 
\end{aligned}
\end{equation*}
\end{claim}

\begin{proof} With $c \in (a, b)$ fixed, we consider 
$\ell_{\alpha} (c; \cdot), \ell_b (c; \cdot): [\lambda_1, \lambda_2] \to \Lambda (n)$
and set 
\begin{equation*}
\begin{aligned}
\tilde{W}_c (\lambda, \mu) &:= 
- (X_{\alpha} (c; \lambda) + i Y_{\alpha} (c; \lambda)) (X_{\alpha} (c; \lambda) - i Y_{\alpha} (c; \lambda))^{-1} \\
& \times (X_b (c; \mu) - i Y_b (c; \mu)) (X_b (c; \mu) + i Y_b (c; \mu))^{-1}. 
\end{aligned}
\end{equation*}
We now compute the Maslov index associated with $\tilde{W}_c (\lambda, \mu)$ along
the triangular path in $[\lambda_1, \lambda_2] \times [\lambda_1, \lambda_2]$ comprising 
the following three paths: (1) fix $\lambda = \lambda_1$ and let $\mu$ increase from 
$\lambda_1$ to $\lambda_2$; (2) fix $\mu = \lambda_2$ and let $\lambda$ increase 
from $\lambda_1$ to $\lambda_2$; and (3) let $\lambda$ and $\mu$ decrease together 
(i.e., with $\lambda = \mu$) from $\lambda_2$ to $\lambda_1$. 
(See Figure \ref{Maslov-triangle}.)
The claim follows from path additivity and homotopy invariance.
\end{proof}

\begin{figure}[ht] \label{Maslov-triangle}  
\begin{center}
\begin{tikzpicture}
\draw[<-, thick] (-5,0) -- (1,0);	%horizontal line
\draw[->, thick] (.5,.5) -- (.5,-5);	%vertical line
%
%Draw triangle
\draw[->, thick] (-4,-4) -- (-4,-2.5);
\draw[thick] (-4,-2.5) -- (-4,-1);
\draw[->, thick] (-4,-1) -- (-2.5,-1);
\draw[thick] (-2.5, -1) -- (-1, -1);
\draw[->, thick] (-1,-1) -- (-2.5,-2.5);
\draw[thick] (-2.5,-2.5) -- (-4,-4);
%
%NOW ADD LABELING
\node at (-5, .5) {$\lambda$};
\node at (-4,.5) {$\lambda_1$};
\draw[thick] (-4,.1) -- (-4,-.1);
\node at (-1,.5) {$\lambda_2$};
\draw[thick] (-1,.1) -- (-1,-.1);
\node at (1,-1) {$\lambda_2$};
\draw[thick] (.4,-1) -- (.6,-1);
\node at (1,-4) {$\lambda_1$};
\draw[thick] (.4,-4) -- (.6,-4);
\node at (1,-5) {$\mu$};

%\draw (0,0) circle (4);	%circle
%
%\draw (3.8,4.45) circle (.1);
%\node at (5.5,4.5) {$\in \sigma (\tilde{W}_c (-c; \lambda_1))$};
%\draw (3.70,3.85) rectangle (3.90,4.05);
%\node at (5.5,4.0) {$\in \sigma (\tilde{\mathcal{W}} (-c; \lambda_1))$};
%
%\filldraw [black] (3.8,-3.90) circle (.1);
%\node at (5.4,-3.85) {$\in \sigma (\tilde{W}_c (c; \lambda_1))$};
%\filldraw [black] (3.70,-4.6) rectangle (3.90,-4.4);
%\node at (5.4,-4.5) {$\in \sigma (\tilde{\mathcal{W}} (c; \lambda_1))$};
%
%PLACE SHAPES ON THE CIRCLE
%PAIR 1
%\draw (-3.94,.75) circle (.1);
%\draw (-3.9,1.1) rectangle (-3.7,1.3);
%
%PAIR 2
%\draw (2.95,2.7) circle (.1);
%\draw (3.1,2.3) rectangle (3.3,2.5);
%
%PAIR 3
%\filldraw [black] (-3.92,-.75) circle (.1);
%\filldraw [black] (-3.95,-1.2) rectangle (-3.75,-1.0);
%
%PAIR 4
%\filldraw [black] (-.92,-3.89) circle (.1);
%\filldraw [black] (-1.4,-3.89) rectangle (-1.2,-3.69);
%
%ADD EPSILON AND DELTA
%\draw[<->] (-4.25,.8) -- (-4.1,1.25);
%\node at (-4.4,1.1) {$\epsilon$};
%\draw[<->] (-3.65,0) -- (-3.65,.75);
%\node at (-3.4,.4) {$\delta$};
\end{tikzpicture}
\end{center}
\caption{Triangular path in the $(\lambda, \mu)$-plane for Claim \ref{triangle-claim-alpha}}

\end{figure}
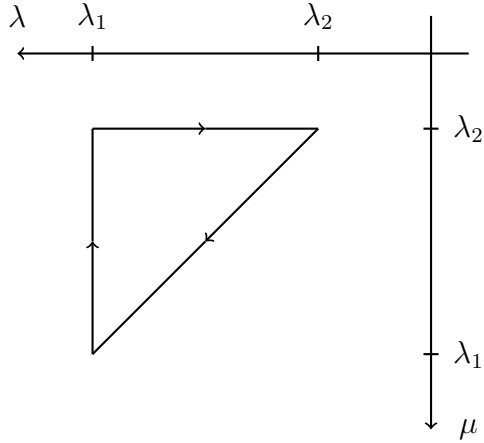

We can conclude from (\ref{truncated-count-alpha}), (\ref{full-count-alpha}), and Claim \ref{triangle-claim-alpha} 
that 
\begin{equation} \label{N-Nc-alpha}
\mathcal{N}^{\alpha} ([\lambda_1, \lambda_2)) = \mathcal{N}_{a, c}^{\alpha} ([\lambda_1, \lambda_2))
- \mas (\ell_{\alpha} (c; \lambda_1), \ell_b (c; \cdot); [\lambda_1, \lambda_2]). 
\end{equation}
By monotonicity, 
\begin{equation*}
\mas (\ell_{\alpha} (c; \lambda_1), \ell_b (c; \cdot); [\lambda_1, \lambda_2]) \le 0,
\end{equation*}
and we can conclude that 
\begin{equation*}
\mathcal{N}^{\alpha} ([\lambda_1, \lambda_2)) 
\ge \mathcal{N}_{a, c}^{\alpha} ([\lambda_1, \lambda_2)). 
\end{equation*}
In light of (\ref{box-sum-alpha}), this gives 
\begin{equation} \label{count-inequality-alpha}
  \mathcal{N}^{\alpha} ([\lambda_1, \lambda_2)) 
\ge \mas (\ell_{\alpha} (\cdot; \lambda_1), \ell_b (\cdot; \lambda_2); [a, c]).  
\end{equation}
Here, we emphasize that under our assumption that 
$\sigma_{\ess} (\mathcal{L}^{\alpha}) \cap [\lambda_1, \lambda_2] = \emptyset$,
the count $\mathcal{N}^{\alpha} ([\lambda_1, \lambda_2))$ must be
finite. 

The Maslov index on the right-hand side of this last expression
increases monotonically with $c$, as described in the following claim 
from \cite{HS2}. 

\begin{claim} \label{crossing-claim}
Let the assumptions of Theorem \ref{regular-singular-theorem} hold,  
and let $x_* \in [a,c]$ 
be a conjugate point along the left shelf. If $x_* \in (a,c]$, then no eigenvalue 
of $\tilde{W} (\cdot; \lambda_1)$ can arrive at $-1$ moving in the clockwise 
direction as $x$ increases to $x_*$. If $x_* = a$, then no eigenvalue 
of $\tilde{W} (\cdot; \lambda_1)$ can rotate away from $-1$ moving in the 
clockwise direction as $x$ increases from $a$.  
\end{claim}

From this claim, we see that there can be at most a finite number
of conjugate points for $\ell_{\alpha} (\cdot; \lambda_1)$ 
and $\ell_b (\cdot; \lambda_2)$ on $[a, b)$. It follows that 
the limit as $c \to b^-$ of the right-hand side of 
(\ref{count-inequality-alpha}) is well-defined. Since the left-hand side 
of (\ref{count-inequality-alpha}) is independent of 
$c$, we can take the limit as $c \to b^-$
on both sides to obtain the inequality claimed 
in Theorem \ref{regular-singular-theorem}.

For the second assertion of Theorem \ref{regular-singular-theorem}
we additionally assume that 
$\lambda_1, \lambda_2 \notin \sigma_p (\mathcal{L}^{\alpha})$, and 
we will closely follow the approach taken in \cite{GZ2017}. We emphasize 
that while we are using almost precisely the same argument 
as in \cite{GZ2017}, our result is not limited to the 
limit-point case (as assumed in \cite{GZ2017}).
Since $\lambda_2 \notin \sigma_p (\mathcal{L}^{\alpha})$, we are 
justified in working with the resolvent operator 
\begin{equation*}
    \mathcal{R} (\mathcal{L}^{\alpha}; \lambda_2)
    := (\mathcal{L}^{\alpha} - \lambda_2 I)^{-1},
\end{equation*}
which we can specify in terms of the Green's function
$G^{\alpha} (x, \xi; \lambda_2)$ constructed in 
Section \ref{green-section}. In particular, for any 
$f \in L^2_{B_1} ((a, b), \mathbb{C}^{2n})$ 
we can write 
\begin{equation*}
 \mathcal{R} (\mathcal{L}^{\alpha}; \lambda_2) f
 = \int_a^b G^{\alpha} (x, \xi; \lambda_2) B_1 (\xi) f(\xi) d\xi.
\end{equation*}

Turning to the operator $\mathcal{L}^{\alpha}_{a, c}$ 
specified above with domain $\mathcal{D}^{\alpha}_{a, c}$,
we first note that by virtue of the appearance of 
$\lambda_2$ in the boundary condition at $x = c$,
$\lambda_2$ is an eigenvalue of $\mathcal{L}^{\alpha}_{a, c}$
if and only if it is an eigenvalue of $\mathcal{L}^{\alpha}$.
We are assuming $\lambda_2 \notin \sigma_p (\mathcal{L}^{\alpha})$,
so we can conclude that $\lambda_2 \notin \sigma_p (\mathcal{L}^{\alpha}_{a, c})$,
and this allows us to work with the resolvent operator
\begin{equation*}
    \mathcal{R} (\mathcal{L}^{\alpha}_{a, c}; \lambda_2)
    := (\mathcal{L}^{\alpha}_{a, c} - \lambda_2 I)^{-1},
\end{equation*}
which we can specify in terms of a Green's function
$G^{\alpha}_{a, c} (x, \xi; \lambda_2)$. In particular, for any 
$f \in L^2_{B_1} ((a, c), \mathbb{C}^{2n})$ 
we can write 
\begin{equation*}
 \mathcal{R} (\mathcal{L}^{\alpha}_{a, c}; \lambda_2) f
 = \int_a^c G^{\alpha}_{a, c} (x, \xi; \lambda_2) B_1 (\xi) f(\xi) d\xi.
\end{equation*}
Proceeding with a construction similar to that for 
$G^{\alpha} (x, \xi; \lambda_2)$ in Section \ref{green-section}, we find that 
$G^{\alpha}_{a, c} (x, \xi; \lambda_2)$ can be 
expressed as 
\begin{equation*}
    G^{\alpha}_{a, c} (x, \xi; \lambda_2) 
    = G^{\alpha} (x, \xi; \lambda_2), 
    \quad \forall \, x, \xi \in (a, c).
\end{equation*}

According to Lemma 2 in Section 4 of Chapter XIII in \cite{RSIV}
(also, Theorem 2.3 in \cite{EE1987}), we can express the 
spectrum of $\mathcal{R} (\mathcal{L}^{\alpha}; \lambda_2)$
as 
\begin{equation*}
    \sigma (\mathcal{R} (\mathcal{L}^{\alpha}; \lambda_2))
    \backslash \{0\}
    = \Big{\{} \frac{1}{\lambda - \lambda_2}: \lambda \in \sigma (\mathcal{L}^{\alpha}) \Big{\}}.
\end{equation*}
In particular, we see that $\mathcal{L}^{\alpha}$ has an eigenvalue on the interval
$(\lambda_1, \lambda_2)$ if and only if $\mathcal{R} (\mathcal{L}^{\alpha}; \lambda_2)$
has an eigenvalue on the interval $(- \infty, (\lambda_1 - \lambda_2)^{-1})$, with corresponding algebraic and geometric multiplicities as well. We can express this as 
\begin{equation} \label{spectral-mapping-equation-alpha}
    \mathcal{N}^{\alpha} ((\lambda_1, \lambda_2))
    = \mathcal{N}^{\alpha, \mathcal{R}} ((- \infty, \frac{1}{\lambda_1 - \lambda_2})),
\end{equation}
where the right-hand side of (\ref{spectral-mapping-equation-alpha})
denotes a count, including multiplicities, of the eigenvalues of 
$\mathcal{R} (\mathcal{L}^{\alpha}; \lambda_2)$ on the 
interval $(-\infty, (\lambda_1 - \lambda_2)^{-1})$. 
Likewise, 
\begin{equation} \label{spectral-mapping-equation-ac}
    \mathcal{N}^{\alpha}_{a, c} ((\lambda_1, \lambda_2))
    = \mathcal{N}^{\alpha, \mathcal{R}}_{a, c} ((- \infty, \frac{1}{\lambda_1 - \lambda_2})),
\end{equation}
where the right-hand side of (\ref{spectral-mapping-equation-ac})
denotes a count, including multiplicities, of the eigenvalues of 
$\mathcal{R} (\mathcal{L}^{\alpha}_{a, c}; \lambda_2)$ on the 
interval $(-\infty, (\lambda_1 - \lambda_2)^{-1})$. 

For ease of notation, we will denote by 
$\Pi_{a, c}: L^2_{B_1} ((a, b), \mathbb{C}^{2n})
 \to L^2_{B_1} ((a, c), \mathbb{C}^{2n})$ the 
restriction operator
\begin{equation*}
    \Pi_{a, c} f = f \Big{|}_{(a, c)},
\end{equation*}
and we will denote by 
$\mathcal{P}_{a, c}: L^2_{B_1} ((a, b), \mathbb{C}^{2n})
\to L^2_{B_1} ((a, b), \mathbb{C}^{2n})$ the 
truncation operator 
\begin{equation*}
 \mathcal{P}_{a, c} f
 = \begin{cases}
f & \textrm{ in } (a, c) \\
0 & \textrm{ in } (c, b).
 \end{cases}
\end{equation*}
With this notation, we can write (exploiting our Green's function 
associated with $\mathcal{L}^{\alpha}$)
\begin{equation*}
    \mathcal{R} (\mathcal{L}^{\alpha}_{a, c}; \lambda_2) \Pi_{a, c} f
    = \Pi_{a, c} \mathcal{R} (\mathcal{L}^{\alpha}; \lambda_2) \mathcal{P}_{a, c} f,
\end{equation*}
for all $f \in  L^2_{B_1} ((a, b), \mathbb{C}^{2n})$. If 
we express $L^2_{B_1} ((a, b), \mathbb{C}^{2n})$ as a 
direct sum
\begin{equation} \label{direct-sum}
    L^2_{B_1} ((a, b), \mathbb{C}^{2n})
    = \Pi_{a, c} L^2_{B_1} ((a, b), \mathbb{C}^{2n})
    \oplus (I - \Pi_{a, c}) L^2_{B_1} ((a, b), \mathbb{C}^{2n}),
\end{equation}
then we can write 
\begin{equation} \label{oplus-relation}
    \begin{aligned}
    (\mathcal{R} (\mathcal{L}^{\alpha}_{a, c}; \lambda_2) \oplus 0 ) f
    &= \Big(\mathcal{R} (\mathcal{L}^{\alpha}_{a, c}; \lambda_2) \Pi_{a, c} f\Big) \oplus 0 \\
    &= \Big( \Pi_{a, c} \mathcal{R} (\mathcal{L}^{\alpha}; \lambda_2) \mathcal{P}_{a, c} f \Big) \oplus 0
    = \mathcal{P}_{a, c} \mathcal{R} (\mathcal{L}^{\alpha}; \lambda_2) \mathcal{P}_{a, c} f.
    \end{aligned}
\end{equation}
(Cf. Corollary 3.3 in \cite{GZ2017}.)

\begin{claim} \label{strong-convergence-claim}
For each $f \in L^2_{B_1} ((a, b), \mathbb{C}^{2n})$,
\begin{equation*}
    \mathcal{P}_{a, c} \mathcal{R} (\mathcal{L}^{\alpha}; \lambda_2) \mathcal{P}_{a, c} f
    \overset{c \to b^-}{\longrightarrow} \mathcal{R} (\mathcal{L}^{\alpha}; \lambda_2) f,
\end{equation*}
in $L^2_{B_1} ((a, b), \mathbb{C}^{2n})$. I.e., 
$\mathcal{P}_{a, c} \mathcal{R} (\mathcal{L}^{\alpha}; \lambda_2) \mathcal{P}_{a, c}$
converges to $\mathcal{R} (\mathcal{L}^{\alpha}; \lambda_2)$ in the 
strong sense as $c \to b^-$. 
\end{claim}

\begin{proof}
Writing $I = \mathcal{P}_{a, c} + (I - \mathcal{P}_{a, c})$, we can compute 
\begin{equation*}
    \begin{aligned}
    \|\mathcal{P}_{a, c} &\mathcal{R} (\mathcal{L}^{\alpha}; \lambda_2) \mathcal{P}_{a, c} f
    - \mathcal{R} (\mathcal{L}^{\alpha}; \lambda_2) f\|_{B_1} \\
    &= \|\mathcal{P}_{a, c} \mathcal{R} (\mathcal{L}^{\alpha}; \lambda_2) \mathcal{P}_{a, c} f
    - \mathcal{P}_{a, c} \mathcal{R} (\mathcal{L}^{\alpha}; \lambda_2) f 
    - (I - \mathcal{P}_{a, c}) \mathcal{R} (\mathcal{L}^{\alpha}; \lambda_2) f\|_{B_1} \\
    &\le \| \mathcal{P}_{a, c} \mathcal{R} (\mathcal{L}^{\alpha}; \lambda_2) \mathcal{P}_{a, c} f
    - \mathcal{P}_{a, c} \mathcal{R} (\mathcal{L}^{\alpha}; \lambda_2) f  \|_{B_1} 
    + \| (I - \mathcal{P}_{a, c}) \mathcal{R} (\mathcal{L}^{\alpha}; \lambda_2) f \|_{B_1} \\
    &= \| \mathcal{P}_{a, c} \mathcal{R} (\mathcal{L}^{\alpha}; \lambda_2) (\mathcal{P}_{a, c} - I) f \|_{B_1}
    + \| (I - \mathcal{P}_{a, c}) \mathcal{R} (\mathcal{L}^{\alpha}; \lambda_2) f \|_{B_1}.
    \end{aligned}
\end{equation*}
For the first of these last two summands, we can write 
\begin{equation*}
    \| \mathcal{P}_{a, c} \mathcal{R} (\mathcal{L}^{\alpha}; \lambda_2) (\mathcal{P}_{a, c} - I) f \|_{B_1}
    \le \| \mathcal{P}_{a, c} \mathcal{R} (\mathcal{L}^{\alpha}; \lambda_2) \|
    \| (\mathcal{P}_{a, c} - I) f \|_{B_1}.
\end{equation*}
Since $\lambda_2 \in \rho (\mathcal{L}^{\alpha})$, 
$\| \mathcal{P}_{a, c} \mathcal{R} (\mathcal{L}^{\alpha}; \lambda_2) \|$
is bounded. Also, 
\begin{equation*}
  \| (\mathcal{P}_{a, c} - I) f \|_{B_1}^2
  = \int_c^b (B_1 (x) f(x), f(x)) dx.
\end{equation*}
Here, $(B_1 (\cdot) f(\cdot), f(\cdot)) \in L^1 ((a, b), \mathbb{C}^{2n})$
and we can conclude that 
\begin{equation*}
    \lim_{c \to b^-} \| (\mathcal{P}_{a, c} - I) f \|_{B_1} = 0.
\end{equation*}
The summand $\| (I - \mathcal{P}_{a, c}) \mathcal{R} (\mathcal{L}^{\alpha}; \lambda_2) f \|_{B_1}$
can be handled similarly with $\mathcal{R} (\mathcal{L}^{\alpha}; \lambda_2) f$
(which is in $L^2 ((a, b), \mathbb{C}^{2n})$) replacing 
$f$. 
\end{proof}

As noted in \cite{GZ2017} (during the proof of Theorem 3.6), we can 
use a slight restatement of Lemma 5.2 from \cite{GST1996}, along with 
the strong convergence established in Claim \ref{strong-convergence-claim}
just above, to conclude that 
\begin{equation} \label{swiped-from-gz2017}
    \mathcal{N}^{\alpha, \mathcal{R}} ((-\infty, \frac{1}{\lambda_1 - \lambda_2}))
    \le \liminf_{c \to b^-} \mathcal{N}^{\alpha, \mathcal{R}}_c ((-\infty, \frac{1}{\lambda_1 - \lambda_2})) 
\end{equation}
where the count on the right-hand side of (\ref{swiped-from-gz2017})  
corresponds with the number of eigenvalues, counted with multiplicity, that 
$\mathcal{P}_{a, c} \mathcal{R} (\mathcal{L}^{\alpha}; \lambda_2) \mathcal{P}_{a, c}$ 
has on the interval $(-\infty, (\lambda_1 - \lambda_2)^{-1})$.

\begin{claim} \label{resolvent-spectrum-claim}
For each $c \in (a, b)$, 
\begin{equation*}
    \sigma (\mathcal{R}(\mathcal{L}^{\alpha}_{a, c}; \lambda_2) \oplus 0)
    = \sigma (\mathcal{R}(\mathcal{L}^{\alpha}_{a, c}; \lambda_2)),
\end{equation*}
and so by virtue of (\ref{oplus-relation}) 
\begin{equation*}
    \sigma (\mathcal{P}_{a, c} \mathcal{R} (\mathcal{L}^{\alpha}; \lambda_2) \mathcal{P}_{a, c})
    = \sigma (\mathcal{R} (\mathcal{L}^{\alpha}_{a, c}; \lambda_2)).
\end{equation*}
In particular, 
\begin{equation*}
    \mathcal{N}_c^{\alpha, \mathcal{R}} ((-\infty, \frac{1}{\lambda_1 - \lambda_2}))
    = \mathcal{N}_{a, c}^{\alpha, \mathcal{R}} ((-\infty, \frac{1}{\lambda_1 - \lambda_2})). 
\end{equation*}
\end{claim}

\begin{proof}
First, we check that 
\begin{equation*}
    \sigma_p (\mathcal{R}(\mathcal{L}^{\alpha}_{a, c}; \lambda_2) \oplus 0)
    = \sigma_p (\mathcal{R}(\mathcal{L}^{\alpha}_{a, c}; \lambda_2)).
\end{equation*}
For this, we observe that 
\begin{equation}
\mathcal{R}(\mathcal{L}^{\alpha}_{a, c}; \lambda_2) \Pi_{a, c} \phi = \mu \Pi_{a, c} \phi    
\end{equation}
for some $\phi \in L^2_{B_1} ((a, b), \mathbb{C}^{2n})$
if and only if 
\begin{equation}
(\mathcal{R}(\mathcal{L}^{\alpha}_{a, c}; \lambda_2) \oplus 0) \mathcal{P}_{a, c} \phi 
= \mu \mathcal{P}_{a, c} \phi,    
\end{equation}
from which its clear that $\Pi_{a, c} \phi$ is an eigenfunction for 
$\mathcal{R}(\mathcal{L}^{\alpha}_{a, c}; \lambda_2)$ with 
eigenvalue $\mu$ if and only if $\mathcal{P}_{a, c} \phi$
is an eigenfunction for $\mathcal{R}(\mathcal{L}^{\alpha}_{a, c}; \lambda_2) \oplus 0$
with eigenvalue $\mu$. 

Next, since $\mathcal{L}^{\alpha}_{a, c}$ is regular at both endpoints, 
its spectrum is entirely discrete. In particular, this means that if 
$\mu \notin \sigma_p (\mathcal{R} (\mathcal{L}^{\alpha}_{a, c}; \lambda_2)) \cup \{0\}$
then $\mu \in \rho (\mathcal{R} (\mathcal{L}^{\alpha}_{a, c}; \lambda_2))$. (Since 
$\mathcal{L}^{\alpha}_{a, c}$ is unbounded, 
$0 \in \sigma (\mathcal{R} (\mathcal{L}^{\alpha}_{a, c}; \lambda_2) 
\backslash \sigma_p (\mathcal{R} (\mathcal{L}^{\alpha}_{a, c}; \lambda_2))$.)

For  $\mu \in \rho (\mathcal{R} (\mathcal{L}^{\alpha}_{a, c}; \lambda_2))$, 
the operator 
\begin{equation*}
     \mathcal{R} (\mathcal{L}^{\alpha}_{a, c}; \lambda_2) - \mu I_{L^2_{B_1} ((a, c), \mathbb{C}^{2n})}
\end{equation*}
maps $L^2_{B_1} ((a, c), \mathbb{C}^{2n})$ onto $L^2_{B_1} ((a, c), \mathbb{C}^{2n})$. We claim
that it follows that 
\begin{equation*}
     (\mathcal{R} (\mathcal{L}^{\alpha}_{a, c}; \lambda_2) \oplus 0) - \mu I_{L^2_{B_1} ((a, b), \mathbb{C}^{2n})}
\end{equation*}
maps $L^2_{B_1} ((a, b), \mathbb{C}^{2n})$ onto $L^2_{B_1} ((a, b), \mathbb{C}^{2n})$. To see 
this, we take any $f \in L^2_{B_1} ((a, b), \mathbb{C}^{2n})$, and we will 
identify $\psi \in L^2_{B_1} ((a, b), \mathbb{C}^{2n})$ so that 
\begin{equation} \label{psi-equation}
  \Big(  (\mathcal{R} (\mathcal{L}^{\alpha}_{a, c}; \lambda_2) \oplus 0) - \mu I_{L^2_{B_1} ((a, b), \mathbb{C}^{2n})} \Big)  
  \psi = f.
\end{equation}
Since $\mathcal{R} (\mathcal{L}^{\alpha}_{a, c}; \lambda_2) - \mu I_{L^2_{B_1} ((a, c), \mathbb{C}^{2n})}$
maps $L^2_{B_1} ((a, c), \mathbb{C}^{2n})$ onto $L^2_{B_1} ((a, c), \mathbb{C}^{2n})$, we can 
find $\phi \in L^2_{B_1} ((a, c), \mathbb{C}^{2n})$ so that 
\begin{equation*}
     \Big(\mathcal{R} (\mathcal{L}^{\alpha}_{a, c}; \lambda_2) - \mu I_{L^2_{B_1} ((a, c), \mathbb{C}^{2n})}\Big) \phi
     = \Pi_{a, c} f. 
\end{equation*}
It follows that 
\begin{equation*}
    \psi := 
    \begin{cases}
        \phi & \textrm{ in } (a, c) \\
        -\frac{1}{\mu} f & \textrm{ in } (c, b)
    \end{cases}
\end{equation*}
satisfies (\ref{psi-equation}). This gives the claim.
\end{proof}

Using (respectively) (\ref{spectral-mapping-equation-alpha}), 
(\ref{swiped-from-gz2017}), Claim \ref{resolvent-spectrum-claim}, 
(\ref{spectral-mapping-equation-ac}), and (\ref{box-sum-alpha}) for
the first five relations below, we can now compute as follows: 
\begin{equation*}
    \begin{aligned}
    \mathcal{N}^{\alpha} ((\lambda_1, \lambda_2)) 
    &= \mathcal{N}^{\alpha, \mathcal{R}} ((- \infty, \frac{1}{\lambda_1 - \lambda_2})) \\
    &\le \liminf_{c \to b^-} \mathcal{N}^{\alpha, \mathcal{R}}_c ((-\infty, \frac{1}{\lambda_1 - \lambda_2})) \\
    &= \liminf_{c \to b^-} \mathcal{N}^{\alpha, \mathcal{R}}_{a, c} ((-\infty, \frac{1}{\lambda_1 - \lambda_2})) \\
    &=  \liminf_{c \to b^-} \mathcal{N}^{\alpha}_{a, c} ((\lambda_1, \lambda_2)) \\
    &= \liminf_{c \to b^-} \mas (\ell_{\alpha} (\cdot; \lambda_1), \ell_b (\cdot; \lambda_2); [a, c]) \\
    &= \mas (\ell_{\alpha} (\cdot; \lambda_1), \ell_b (\cdot; \lambda_2); [a, b)).
    \end{aligned}
\end{equation*}
We conclude that 
\begin{equation*}
 \mathcal{N}^{\alpha} ((\lambda_1, \lambda_2)) 
 \le \mas (\ell_{\alpha} (\cdot; \lambda_1), \ell_b (\cdot; \lambda_2); [a, b)),
\end{equation*}
and this gives the claim of equality in Theorem \ref{regular-singular-theorem}.
For this final observation, we note that since $\lambda_2 \notin \sigma_p (\mathcal{L}^{\alpha})$,
we cannot have a conjugate point at $x = a$ (cf. remarks about the bottom shelf 
above), and so the interval $[a, b)$ can be replaced by $(a, b)$.
\hfill $\square$

\begin{remark} \label{regular-singular-remark}
We see from the preceding discussion (especially (\ref{N-Nc-alpha})) 
that we have equality in 
Theorem \ref{regular-singular-theorem} if and only if 
\begin{equation} \label{large-c}
\mas (\ell_{\alpha} (c; \lambda_1), \ell_b (c; \cdot); [\lambda_1, \lambda_2]) = 0,
\end{equation}
for all $c \in (a, b)$ sufficiently close to $b$ (sufficiently 
large if $b = + \infty$). In making this observation, we've
used the fact that for each $c \in (a, b)$, 
$\mas (\ell_{\alpha} (c; \lambda_1), \ell_b (c; \cdot); [\lambda_1, \lambda_2])$
is a non-negative integer, so we can only have 
\begin{equation*}
    \lim_{c \to b^-} \mas (\ell_{\alpha} (c; \lambda_1), \ell_b (c; \cdot); [\lambda_1, \lambda_2])
    = 0
\end{equation*}
if (\ref{large-c}) holds as described.
By monotonicity as $\lambda$ varies, this last relation is true if 
and only if 
\begin{equation} \label{intersection-condition1}
\ell_{\alpha} (c; \lambda_1) \cap \ell_b (c; \lambda) = \{0\}, \quad
\forall\, \lambda \in [\lambda_1, \lambda_2),
\end{equation}
for all $c \in (a, b)$ sufficiently close to $b$ (sufficiently 
large if $b = + \infty$). Here, the rotation is clockwise, so 
$\lambda_2$ is excluded, since a conjugate arrival as 
$\lambda$ increases to $\lambda_2$ would not affect 
the Maslov index. 
\end{remark}

\subsection{Proof of Theorem \ref{singular-theorem}} 
\label{singular-theorem-section}

Similarly as in the proof of Theorem \ref{regular-singular-theorem}, we fix any pair 
$\lambda_1, \lambda_2 \in \mathbb{R}$, $\lambda_1 < \lambda_2$ 
for which $\sigma_{\ess} (\mathcal{L}) \cap [\lambda_1, \lambda_2] = \emptyset$. 
For the proof of Theorem \ref{singular-theorem}, we let $\ell_b (x; \lambda_2)$
be as in the proof of Theorem \ref{regular-singular-theorem}, and we 
let $\ell_a (x; \lambda)$ denote the map 
of Lagrangian subspaces associated with the frames $\mathbf{X}_a (x; \lambda)$
constructed as in Lemma \ref{continuity-lemma}, except for the operator 
$\mathcal{L}_{a, c}$. 
We will establish Theorem \ref{singular-theorem} by considering the Maslov
index for $\ell_a (x; \lambda)$ and $\ell_b (x; \lambda_2)$
along the Maslov box designated just below. As described in 
Section \ref{maslov-section}, this Maslov index is computed as
a spectral flow for the matrix 
\begin{equation} \label{singular-tildeW}
\begin{aligned}
\tilde{W} (x; \lambda) &= - (X_a (x; \lambda) + i Y_a (x; \lambda))
(X_a (x; \lambda) - i Y_a (x; \lambda))^{-1} \\
& \times (X_b (x; \lambda_2) - i Y_b (x; \lambda_2))
(X_b (x; \lambda_2) + i Y_b (x; \lambda_2))^{-1}
\end{aligned}
\end{equation}
(re-defined from Section \ref{regular-singular-section}). 

In this case, the Maslov Box will consist of the following sequence of contours, 
specified for some values $c_1, c_2 \in (a, b)$, $c_1 < c_2$ to be 
chosen sufficiently close to $a$ and $b$ (respectively) during the analysis:
(1) fix $x = c_1$ and let $\lambda$ increase from $\lambda_1$ to $\lambda_2$ 
(the {\it bottom shelf}); 
(2) fix $\lambda = \lambda_2$ and let $x$ increase from $c_1$ to $c_2$ 
(the {\it right shelf}); (3) fix $x = c_2$ and let $\lambda$
decrease from $\lambda_2$ to $\lambda_1$ (the {\it top shelf}); and (4) fix
$\lambda = \lambda_1$ and let $x$ decrease from $c_2$ to $c_1$ 
(the {\it left shelf}). 

{\it Right shelf.} 
In this case, our calculation along the right shelf detects intersections 
between $\ell_a (x; \lambda_2)$ and $\ell_b (x; \lambda_2)$ as $x$ increases
from $c_1$ to $c_2$. By construction, 
$\ell_a (\cdot; \lambda_2)$ will intersect $\ell_b (\cdot; \lambda_2)$
at some value $x$ with dimension $m$ if and only if $\lambda_2$ is an 
eigenvalue of $\mathcal{L}$ with multiplicity $m$. In the event that $\lambda_2$
is not an eigenvalue of $\mathcal{L}$, there will be no 
conjugate points along the right shelf. On the other hand, if
$\lambda_2$ is an eigenvalue of $\mathcal{L}$ with multiplicity
$m$, then $\tilde{W} (x; \lambda_2)$ will have $-1$ as an eigenvalue
with multiplicity $m$ for all $x \in [c_1, c_2]$. In either case,
\begin{equation} \label{right-shelf-singular}
\mas (\ell_a (\cdot; \lambda_2), \ell_b (\cdot; \lambda_2); [c_1, c_2]) = 0.
\end{equation} 

{\it Bottom shelf.} For the bottom shelf, we're looking for 
intersections between $\ell_a (c_1;\lambda)$ and 
$\ell_b (c_1; \lambda_2)$ as $\lambda$ increases from 
$\lambda_1$ to $\lambda_2$. Since $\ell_a (x; \lambda)$
corresponds with solutions that lie left in $(a, b)$, 
this leads to a calculation similar to the calculation of 
\begin{equation*}
\mas (\ell_{\alpha} (c; \cdot), \ell_b (c; \lambda_2); [\lambda_1, \lambda_2]),
\end{equation*}  
which arose in our analysis of the top shelf for the proof
of Theorem \ref{regular-singular-theorem}. For the moment, 
the only thing we will note about this quantity is that due 
to monotonicity in $\lambda$, we have
the inequality 
\begin{equation*}
\mas (\ell_a (c_1; \lambda_1), \ell_b (c_1; \cdot); [\lambda_1, \lambda_2]) \le 0.
\end{equation*}  

{\it Top shelf.} For the top shelf, $\tilde{W} (c_2; \lambda)$ detects intersections
between $\ell_a (c_2; \lambda)$ and $\ell_b (c_2; \lambda_2)$ as $\lambda$
decreases from $\lambda_2$ to $\lambda_1$. In this way, intersections
correspond precisely with eigenvalues of the restriction
$\mathcal{L}_{a, c_2}$ of the maximal operator associated with
(\ref{linear-hammy}) on $(a, c_2)$ to the domain 
\begin{equation*}
    \mathcal{D}_{a, c_2} := 
    \{y \in \mathcal{D}_{a, c_2, M}: \lim_{x \to a^+} U^a (x; \lambda_0)^* J y (x) = 0, 
    \quad \mathbf{X}_b (c_2; \lambda_2)^* J y (c_2) = 0 \}.
\end{equation*}
Similarly as in Section \ref{operator-section}, we can check that 
$\mathcal{L}_{a, c_2}$ is a self-adjoint operator. 

We can verify monotonicity along the top shelf almost precisely as 
in the proof of Theorem \ref{regular-singular-theorem}, and we can conclude from 
this that 
\begin{equation} \label{truncated-count}
\mathcal{N}_{a, c_2} ([\lambda_1, \lambda_2)) = 
- \mas (\ell_a (c_2; \cdot), \ell_b (c_2; \lambda_2); [\lambda_1, \lambda_2]),
\end{equation}
where $\mathcal{N}_{a, c_2} ([\lambda_1, \lambda_2))$ denotes a count 
of the number of eigenvalues that $\mathcal{L}_{a, c_2}$ has on 
the interval $[\lambda_1, \lambda_2)$. (The inclusion of $\lambda_1$ and exclusion 
of $\lambda_2$ are precisely as discussed in the proof of 
Theorem \ref{regular-singular-theorem}.)

Similarly as with Claim \ref{triangle-claim-alpha}, we obtain the relation 
\begin{equation} \label{triangle-claim}
\begin{aligned}
\mas (\ell_a (c_2; \cdot), \ell_b (c_2; \cdot); [\lambda_1, \lambda_2])
&= 
\mas (\ell_a (c_2; \lambda_1), \ell_b (c_2; \cdot); [\lambda_1, \lambda_2]) \\
&+ \mas (\ell_a (c_2; \cdot), \ell_b (c_2; \lambda_2); [\lambda_1, \lambda_2]).
\end{aligned}
\end{equation}
Recalling that $\mathcal{N} ([\lambda_1, \lambda_2))$ denotes 
the number of eigenvalues that $\mathcal{L}$
has on the interval $[\lambda_1, \lambda_2)$, we can write
\begin{equation*}
\begin{aligned}
\mathcal{N}([\lambda_1, \lambda_2)) &= - \mas (\ell_a (c_2; \cdot), \ell_b (c_2; \cdot); [\lambda_1, \lambda_2]) \\
&= - \mas (\ell_a (c_2; \lambda_1), \ell_b (c_2; \cdot); [\lambda_1, \lambda_2])
- \mas (\ell_a (c_2; \cdot), \ell_b (c_2; \lambda_2); [\lambda_1, \lambda_2]) \\
&= \mathcal{N}_{a, c_2} ([\lambda_1, \lambda_2)) - \mas (\ell_a (c_2; \lambda_1), \ell_b (c_2; \cdot); [\lambda_1, \lambda_2]).   
\end{aligned}
\end{equation*}

{\it Left shelf.} 
Our analysis so far leaves only the left shelf to consider, and 
we observe that it can be expressed as 
\begin{equation*}
- \mas (\ell_a (\cdot; \lambda_1), \ell_b (\cdot; \lambda_2); [c_1, c_2]),
\end{equation*} 
which is part of the Maslov index that appears in the
statement of Theorem \ref{singular-theorem}.
Using path additivity and homotopy invariance, we can sum the Maslov
indices on each shelf of the Maslov Box to arrive at the relation 
\begin{equation} \label{box-sum-equation}
\mathcal{N}_{a, c_2} ([\lambda_1, \lambda_2)) 
= \mas (\ell_a (\cdot; \lambda_1), \ell_b (\cdot; \lambda_2); [c_1, c_2])
- \mas (\ell_a (c_1;\cdot), \ell_b (c_1; \lambda_2); [\lambda_1, \lambda_2]).
\end{equation}
We can now write 
\begin{equation} \label{full-count-equation}
\begin{aligned}
\mathcal{N} ([\lambda_1, \lambda_2)) &=
\mathcal{N}_{a, c_2} ([\lambda_1, \lambda_2)) - \mas (\ell_a (c_2; \lambda_1), \ell_b (c_2; \cdot); [\lambda_1, \lambda_2]) \\
&= 
\mas (\ell_a (\cdot; \lambda_1), \ell_b (\cdot; \lambda_2); [c_1, c_2])
- \mas (\ell_a (c_1; \cdot), \ell_b (c_1; \lambda_2); [\lambda_1, \lambda_2]) \\
& - \mas (\ell_a (c_2; \lambda_1), \ell_b (c_2; \cdot); [\lambda_1, \lambda_2]).
\end{aligned}
\end{equation}

Recalling the monotonicity relations, 
\begin{equation*}
\begin{aligned}
\mas (\ell_a (c_1; \cdot), \ell_b (c_1; \lambda_2); [\lambda_1, \lambda_2]) &\le 0, \\
\mas (\ell_a (c_2; \lambda_1), \ell_b (c_2; \cdot); [\lambda_1, \lambda_2]) &\le 0,
\end{aligned}
\end{equation*}
we can conclude the inequality 
\begin{equation} \label{count-inequality} 
\mathcal{N} ([\lambda_1, \lambda_2)) 
\ge \mas (\ell_a (\cdot; \lambda_1), \ell_b (\cdot; \lambda_2); [c_1, c_2]).
\end{equation}
Using again Claim \ref{crossing-claim} from the proof of 
Theorem \ref{regular-singular-theorem}, we see that there can be at most a finite number
of conjugate points for $\ell_a (\cdot; \lambda_1)$ 
and $\ell_b (\cdot; \lambda_2)$ on $(a, b)$. It follows that 
the limit as $c_1 \to a^+$ of the right-hand side of 
(\ref{count-inequality}) is well-defined, as is the subsequent
limit as $c_2 \to b^-$ . Since the left-hand side 
of (\ref{count-inequality}) is independent of 
$c_1$ and $c_2$, we can take the pair of limits 
on both sides to obtain the inequality claimed in 
Theorem \ref{singular-theorem}.

For the second assertion of Theorem \ref{regular-singular-theorem}
we additionally assume that 
$\lambda_1, \lambda_2 \notin \sigma_p (\mathcal{L})$.
Our goal is to show that 
\begin{equation} \label{opposite-inequality-equation}
    \mathcal{N} ((\lambda_1, \lambda_2))
    \le \mas (\ell_a (\cdot; \lambda_1), \ell_b (\cdot; \lambda_2); (a, b)),
\end{equation}
and we note from (\ref{full-count-equation}) that this is implied 
if {\it both} of the following two conditions hold: 
\begin{equation} \label{intersection-condition-a}
\ell_a (c_1; \lambda) \cap \ell_b (c_1; \lambda_2)) = \{0\}, \quad
\forall \,\lambda \in [\lambda_1, \lambda_2),
\end{equation}
for all $c_1 \in (a, b)$ sufficiently close to $a$ (sufficiently 
negative if $a = - \infty$), and 
\begin{equation} \label{intersection-condition-b}
\ell_a (c_2; \lambda_1) \cap \ell_b (c_2; \lambda) = \{0\}, \quad
\forall \,\lambda \in [\lambda_1, \lambda_2),
\end{equation}
for all $c_2 \in (a, b)$ sufficiently close to $b$ (sufficiently 
large if $b = + \infty$).
(The inclusion of $\lambda_1$ in the intervals and exclusion 
of $\lambda_2$ is discussed in Remark \ref{regular-singular-remark}.)

We proceed by dividing the analysis into two half-interval 
problems. For this, we first fix any $c \in (a, b)$,  
and we introduce a new operator $\mathcal{L}_{c, b}$
as the restriction of $\mathcal{L}_{c, b, M}$ to the 
domain 
\begin{equation*}
    \mathcal{D}_{c, b} 
    := \{y \in \mathcal{D}_{c, b, M}: \mathbf{X}_a (c; \lambda_1)^* J y(c) = 0,
    \quad \lim_{x \to b^-} U^b (x; \lambda_0)^* J y(x) = 0 \}. 
\end{equation*}
We can view $\mathcal{L}_{c, b}$ as a special case of the operator 
$\mathcal{L}^{\alpha}_{a, b}$ analyzed in Section \ref{regular-singular-section}, 
with $a$ replaced by $c$ and $\alpha$ replaced by $\mathbf{X}_a (c; \lambda_1)^* J$. 
It follows that $\ell_{\alpha} (x; \lambda_1)$ from 
Section \ref{regular-singular-section} is replaced by 
$\ell_a (x; \lambda_1)$, so that by virtue of Remark \ref{regular-singular-remark}, 
we can conclude that
\begin{equation*} 
\ell_a (c_2; \lambda_1) \cap \ell_b (c_2; \lambda)) = \{0\}, \quad
\forall \,\lambda \in [\lambda_1, \lambda_2),
\end{equation*}
for all $c_2 \in (a, b)$ sufficiently close to $b$ (sufficiently 
large if $b = + \infty$). This is precisely (\ref{intersection-condition-b}).

Likewise, we introduce an operator $\mathcal{L}_{a, c}$
as the restriction of $\mathcal{L}_{a, c, M}$ to the 
domain 
\begin{equation*}
    \mathcal{D}_{a, c} 
    := \{y \in \mathcal{D}_{c, b, M}: \lim_{x \to a^+} U^a (x; \lambda_0)^* J y(x) = 0,
    \quad \mathbf{X}_b (c; \lambda_2)^* J y(c) = 0 \}. 
\end{equation*}
Proceeding similarly as in Section \ref{regular-singular-section},
we find that in this case 
\begin{equation*}
\ell_a (c_1; \lambda) \cap \ell_b (c_1; \lambda_2)) = \{0\}, \quad
\forall \,\lambda \in [\lambda_1, \lambda_2),
\end{equation*}
for all $c_1 \in (a, b)$ sufficiently close to $a$ (sufficiently 
negative if $a = - \infty$). This is precisely 
(\ref{intersection-condition-a}). 

As already noted, (\ref{intersection-condition-a}) and 
(\ref{intersection-condition-b}) together imply 
(\ref{opposite-inequality-equation}), and this completes 
the proof of Theorem \ref{singular-theorem}.
\hfill $\square$

\section{Applications} \label{applications-section}

In this section, we will discuss two specific applications of our main results, 
though we first need to make one further observation 
associated with Niessen's approach. We recall that 
the key element in Niessen's approach is an emphasis 
on the matrix 
\begin{equation*}
    \mathcal{A} (x; \lambda)
    = \frac{1}{2 {\rm Im}\,\lambda} \Phi (x; \lambda)^* (J/i) \Phi(x; \lambda),
\end{equation*}
where $\Phi (x; \lambda)$ denotes a fundamental matrix for 
(\ref{linear-hammy}), and we clearly require 
${\rm Im }\,\lambda \ne 0$. We saw in Section \ref{operator-section} that 
if $\{\mu_j (x; \lambda)\}_{j=1}^{2n}$ denote the eigenvalues
of $\mathcal{A} (x; \lambda)$, then the number of solutions of 
(\ref{linear-hammy}) that lie left in $(a, b)$ is precisely 
the number of these eigenvalues with a finite limit as 
$x$ approaches $a$, while the number of solutions of 
(\ref{linear-hammy}) that lie right in $(a, b)$ is precisely 
the number of these eigenvalues with a finite limit as 
$x$ approaches $b$. Since this number does not vary as $\lambda$
varies in the upper half-plane (or, alternatively, in the 
lower half-plane), we can categorize the limit-case 
(i.e., limit-point, limit-circle, or limit-$m$) of 
(\ref{linear-hammy}) by fixing some $\lambda \in \mathbb{C}$
with ${\rm Im }\,\lambda > 0$ and computing the values 
$\{\mu_j (x; \lambda)\}_{j=1}^{2n}$ as $x$ tends to $a$
and as $x$ tends to $b$. (This is precisely what we will do
in our examples below.) Furthermore, we have additionally seen in Section \ref{operator-section} 
that for each $\mu_j (x; \lambda)$ (with or without a finite 
limit), we can associate a sequence of eigenvectors 
$\{v_j (x_k; \lambda)\}_{k=1}^{\infty}$
that converges, as $x_k \to a^+$, to some $v_j^a (\lambda)$ that lies on the 
unit circle in $\mathbb{C}^{2n}$, and similarly for a 
sequence $x_k \to b^-$. If $\mu_j (x; \lambda)$
has a finite limit as $x \to a^+$, then $\Phi (x; \lambda) v_j^a (\lambda)$
will lie left in $(a, b)$, while if $\mu_j (x; \lambda)$
has a finite limit as $x \to b^-$, then $\Phi (x; \lambda) v_j^b (\lambda)$
will lie right in $(a, b)$.

In practice, we would like to extend these ideas to values
$\lambda \in \mathbb{R}$, and for this, we replace 
$\mathcal{A} (x; \lambda)$ with 
\begin{equation} \label{mathcal-B}
    \mathcal{B} (x; \lambda) := \Phi(x; \lambda)^* J \partial_{\lambda} \Phi (x; \lambda).
\end{equation}
If we differentiate (\ref{mathcal-B}) with respect to $x$, we find that 
\begin{equation} \label{mathcal-B-prime}
    \mathcal{B}' (x; \lambda) 
    = \Phi(x; \lambda)^* B_1 (x) \Phi (x; \lambda), 
\end{equation}
and upon integrating we see that we can alternatively express 
$\mathcal{B} (x; \lambda)$ as 
\begin{equation} \label{mathcal-B-integrated}
    \mathcal{B} (x; \lambda) 
    = \int_c^x \Phi (\xi; \lambda)^* B_1 (\xi) \Phi (\xi;\lambda)d\xi,
\end{equation}
where we've observed that since $\Phi (c;\lambda) = I_{2n}$,
we have $\mathcal{B} (c; \lambda) = 0$.
Recalling that $B_1 (x)$ is self-adjoint for a.e. $x \in (a, b)$,
we see from this relation that $\mathcal{B} (x; \lambda)$ is 
self-adjoint for all $x \in (a, b)$. 
Consequently, the eigenvalues of $\mathcal{B} (x; \lambda)$
must be real-valued, and we denote these values $\{\nu_j (x; \lambda)\}_{j=1}^{2n}$.
Since $\mathcal{B} (c; \lambda) = 0$, we can conclude that 
$\nu_j (c; \lambda) = 0$ for all $j \in \{1, 2, \dots, 2n\}$,
and all $\lambda \in \mathbb{R}$. In addition, according to 
(\ref{mathcal-B-prime}), along with Condition {\bf (B)}, 
for each fixed $\lambda \in \mathbb{R}$,
the eigenvalues $\{\nu_j (x; \lambda)\}_{j=1}^{2n}$ will 
be non-decreasing as $x$ increases. As $x \to b^-$, each 
eigenvalue $\nu_j (x; \lambda)$ will either approach $+ \infty$
or a finite limit. In the latter case, we set 
\begin{equation*}
    \nu_j^b (\lambda) := \lim_{x \to b^-} \nu_j (x; \lambda).
\end{equation*}
Likewise, as $x \to a^+$, each eigenvalue $\nu_j (x; \lambda)$
will either approach $- \infty$ or a finite limit. In the latter
case, we set 
\begin{equation*}
    \nu_j^a (\lambda) := \lim_{x \to a^+} \nu_j (x; \lambda).
\end{equation*}

Comparing the relations (\ref{A-alternative}) and (\ref{mathcal-B-integrated}), 
we see that the proof of Lemma \ref{subspace-dimensions-lemma} can be adapted
with almost no changes to establish the following lemma.

\begin{lemma} \label{subspace-dimensions-lemma-real}
Let Assumptions {\bf (A)} and {\bf (B)} hold,
and let $\lambda \in [\lambda_1, \lambda_2]$ be fixed.
Then the dimension $m_a (\lambda)$ of the subspace of solutions to 
(\ref{linear-hammy}) that lie left in $(a, b)$ is precisely 
the number of eigenvalues $\nu_j (x; \lambda) \in \sigma (\mathcal{B} (x; \lambda))$
that approach a finite limit as $x \to a^+$. Likewise, 
the dimension $m_b (\lambda)$ of the subspace of solutions to 
(\ref{linear-hammy}) that lie right in $(a, b)$ is precisely 
the number of eigenvalues $\nu_j (x; \lambda) \in \sigma (\mathcal{B} (x; \lambda))$
that approach a finite limit as $x \to b^-$.
\end{lemma}

\begin{remark}
We emphasize that as opposed to the case $\lambda \in \mathbb{C} \backslash \mathbb{R}$,
we cannot conclude from these considerations that $m_a (\lambda), m_b (\lambda) \ge n$.
Rather, in this case we conclude these inequalities for all $\lambda \in [\lambda_1, \lambda_2]$ 
from Lemma \ref{frames-lemma} (under assumptions {\bf (A)}, {\bf (B)}, 
and {\bf (C)}). Here, as usual, we are taking 
$[\lambda_1, \lambda_2] \cap \sigma_{\ess} (\mathcal{L}) = \emptyset$ (or, likewise,
$[\lambda_1, \lambda_2] \cap \sigma_{\ess} (\mathcal{L}^{\alpha}) = \emptyset$). 
\end{remark}

If, for each $x \in (a, b)$, we let $\{w_j (x; \lambda)\}_{j=1}^{2n}$ denote 
an orthonormal collection of eigenvectors associated with 
the eigenvalues $\{\nu_j (x; \lambda)\}_{j=1}^{2n}$, then as in the 
proof of Lemma \ref{subspace-dimensions-lemma}, we can find 
(for each $j \in \{1, 2, \dots, 2n\}$) a sequence $\{w_j (x_k; \lambda)\}_{k=1}^{\infty}$
that converges, as $x_k \to a^+$, to some $w_j^a (\lambda)$ on the unit circle in 
$\mathbb{C}^{2n}$, and likewise we can find 
a sequence $\{w_j (x_k; \lambda)\}_{k=1}^{\infty}$
that converges, as $x_k \to b^-$, to some $w_j^b (\lambda)$ on the unit circle in 
$\mathbb{C}^{2n}$. Moreover, if $\nu_j (x; \lambda)$
has a finite limit as $x \to a^+$, then $\Phi (x; \lambda) w_j^a (\lambda)$
will lie left in $(a, b)$, while if $\nu_j (x; \lambda)$
has a finite limit as $x \to b^-$, then $\Phi (x; \lambda) w_j^b (\lambda)$
will lie right in $(a, b)$.

These considerations provide a method for constructing the frames
$\mathbf{X}_a (x; \lambda)$ and $\mathbf{X}_b (x; \lambda)$ that 
we'll need in order to implement Theorems \ref{regular-singular-theorem}
and \ref{singular-theorem}. Most directly, if (\ref{linear-hammy})
is limit-point at $x = a$ (respectively, $x=b$), then the procedure 
described in the previous paragraph will provide precisely $n$ 
linearly independent solutions to (\ref{linear-hammy}) that 
lie left in $(a, b)$ (respectively, right in $(a, b)$), and 
these will necessarily comprise the columns of $\mathbf{X}_a (x; \lambda)$
(respectively, $\mathbf{X}_b (x; \lambda)$). 

More generally, Lemma \ref{subspace-dimensions-lemma} can be used 
to construct left and right lying solutions of (\ref{linear-hammy})
for some $\lambda_0 \in \mathbb{C} \backslash \mathbb{R}$,
and these can then be used to specify the Niessen elements described 
in the lead-in to Lemma \ref{krall-niessen-lemma}. I.e., the matrices
$U^a (x; \lambda_0)$ and $U^b (x; \lambda_0)$ discussed in 
Section \ref{operator-section} can be constructed in this way. 
Working, for example, with the solutions constructed above 
for $\lambda \in \mathbb{R}$ that life left in $(a, b)$, we 
can identify $n$ linearly independent solutions $\{u_j^a (x; \lambda)\}_{j=1}^n$
that satisfy 
\begin{equation*}
    \lim_{x \to a^+} U^a (x; \lambda_0)^* J u_j^a (x; \lambda) = 0. 
\end{equation*}
This collection $\{u_j^a (x; \lambda)\}_{j=1}^n$ will comprise 
the columns of $\mathbf{X}_a (x; \lambda)$, and we can proceed 
similarly for $x = b$. 

We now turn to our examples.

\subsection{Counting Eigenvalues in Spectral Gaps}
\label{gaps-section}

In this section, we discuss (single) Schr\"odinger equations
\begin{equation*}
    \begin{aligned}
    H \phi &:= - \phi'' + V(x) \phi = \lambda \phi, \quad \textrm{in } (0, \infty) \\
    \alpha_1 &\phi(0) + \alpha_2 \phi'(0) = 0,
    \end{aligned}
\end{equation*}
where $V(x)$ is a bounded, real-valued potential obtained by 
compactly perturbing a periodic potential $V_0 (x)$,
and $\alpha_1, \alpha_2 \in \mathbb{R}$ are not both 
0. 

It's well known (see, for example, \cite{Kuchment2016} and 
the references cited there) that if we set 
\begin{equation*}
       H_0 \phi := - \phi'' + V_0 (x) \phi = \lambda \phi,
       \quad \textrm{in } (0, \infty),
\end{equation*}
along with any self-adjoint boundary condition at $x = 0$,
then $\sigma_{\ess} (H_0)$ can be expressed as a union of 
closed intervals 
\begin{equation*}
    \sigma_{\ess} (H_0) = \bigcup_{j=1}^{\infty} [a_j, b_j],
\end{equation*}
or in some special cases as a similar finite union that 
includes an unbounded interval $[b_N, +\infty)$. The 
intervals $\{[a_j, b_j]\}_{j=1}^{\infty}$ are referred 
to as spectral bands for $H_0$, and the intervening intervals 
$[b_j, a_{j+1}]$ are referred to as spectral gaps. 
(It may be the case that $b_j = a_{j+1}$, leaving no 
gap.) In addition,
if $V_0 (x)$ is perturbed to a new potential 
$V(x) = V_0 (x) + V_1 (x)$, where $V_1 \in L^1 ((0, \infty), \mathbb{R})$,
then we will have $\sigma_{\ess} (H) = \sigma_{\ess} (H_0)$. 
(See, for example, Corollary XIII.4.2 in \cite{RSIV}.) 
However, it may be the case that $H$ has additional 
eigenvalues in the spectral gaps, including up to 
an infinite number accumulating at an endpoint 
of essential spectrum. 
Let $[b_j, a_{j+1}]$, $b_j < a_{j+1}$ denote some 
particular spectral gap. Then our approach allows us
to fix any interval $(\lambda_1, \lambda_2) \in [b_j, a_{j+1}]$,
$\lambda_1, \lambda_2 \notin \sigma (H)$
and determine the number of eigenvalues on this 
interval.

As a specific example, taken from \cite{AGM2006}, we consider
$H$ with 
\begin{equation*}
    V(x) = V_0 (x) + V_1 (x) = \sin(x) + \frac{60}{1+x^2},
    \quad \alpha_1 = \cos(\pi/8), \, \alpha_2 = \sin(\pi/8).
\end{equation*}
In \cite{AGM2006}, the authors identify the first two 
spectral gaps for $H_0$ as 
\begin{equation*}
    J_1 = (-\infty, -.3785), 
    \quad J_2 = (-.3477, .5948),
\end{equation*}
and they verify that $-.3477$ serves as an accumulation point
for eigenvalues of $H$ in the interval $J_2$. In addition,
the authors identify the 13 right-most eigenvalues of $H$ in 
this interval. (In these calculations, the authors proceed
with a higher degree of precision than given above; 
see \cite{AGM2006} for the full results.) 

In order to place this equation in our setting, we 
set $y = {y_1 \choose y_2} = {\phi \choose \phi'}$,
from which we arrive at (\ref{linear-hammy}) with 
\begin{equation*}
    B_0 (x) + \lambda B_1 (x)
    = \begin{pmatrix}
    - \sin(x) - \frac{60}{1+x^2} & 0 \\
    0 & 1
    \end{pmatrix}
    + \lambda \begin{pmatrix}
    1 & 0 \\
    0 & 0
    \end{pmatrix}.
\end{equation*}
With this choice of $\mathbb{B} (x; \lambda)$,
(\ref{linear-hammy}) is regular at $x = 0$ and of course 
singular at $x = +\infty$. (I.e., we are in the 
case in which {\bf (A)$^{\prime}$} holds.) In order to determine if
(\ref{linear-hammy}) is limit point or limit circle at $+\infty$, 
we fix $\lambda_0 = i$ (arbitrarily selected as an 
element $\lambda_0 \in \mathbb{C} \backslash \mathbb{R}$) 
and numerically generate 
the eigenvalues of $\mathcal{A} (x; \lambda_0)$ 
as $x$ increases. (In this case, we initialize 
the fundamental matrix $\Phi (x; \lambda_0)$ at 
$x = 0$.) We know from our general theory developed 
in Section \ref{operator-section} that 
the eigenvalues $\{\mu_j (x; \lambda_0)\}_{j=1}^2$ 
of  $\mathcal{A} (x; \lambda_0)$ will satisfy
(with our choice of indexing)
$\mu_1 (x; \lambda_0) < 0 < \mu_2 (x; \lambda_0)$
for all $x \in (0, \infty)$. 
As $x$ increases, these eigenvalues will both 
monotonically increase, and so $\mu_1 (x; \lambda_0)$
will certainly approach a finite limit (since it 
is bounded above by $0$). In this way, the limit case is determined
by whether $\mu_2 (x; \lambda_0)$ approaches a 
finite limit as $x$ tends to $+ \infty$. 
Computing numerically, we find 
$\mu_2 (5; \lambda_0) = 1.1543 \times 10^9$,
suggesting that $H$ is limit-point at $+ \infty$. 

\begin{remark} \label{computation-remark}
Throughout this section, our numerical calculations are 
intended only to illustrate the theory, and we make no 
effort to rigorously justify either the values we obtain
or the conclusions we draw from them. For example, 
in this last calculation, we have not attempted to 
find a rigorous error interval for the value of 
$\mu_2 (5; \lambda_0)$, and we offer no additional 
direct justification that $\mu_2 (x; \lambda_0)$ is indeed
tending to $+ \infty$ as $x$ tends to $+ \infty$. 
(It follows from Corollary 1 in Chapter 9 of \cite{CL1955}
that $H$ is indeed limit-point at $+\infty$, and from this
we can conclude that this limiting behavior must be
correct.) In all cases, the calculations are carried out with 
built-in MATLAB functions, primarily {\it ode45.m}.
\end{remark}

\begin{remark} \label{equivalence-remark1}
It's straightforward to check that $H$ and $\mathcal{L}^{\alpha}$ 
(the latter constructed as in Lemma \ref{self-adjoint-operator-lemma})
have precisely the same sets of essential spectrum, and 
also the same sets of discrete eigenvalues. Here, 
\begin{equation*}
    \begin{aligned}
    \dom(H) &= \{\phi \in L^2 ((0,\infty),\mathbb{C}): 
    \phi, \phi' \in \AC_{\loc} ([0,\infty),\mathbb{C}), \\
    & \quad \quad H\phi \in L^2 ((0,\infty),\mathbb{C}), \,
    \alpha_1 \phi(0) + \alpha_2 \phi'(0) = 0
    \}.
    \end{aligned}
\end{equation*}
\end{remark}

Since $H$ is regular at $x = 0$, we can find $\mathbf{X}_{\alpha} (x; \lambda_1)$
by solving the initial value problem 
\begin{equation*}
    J \mathbf{X}_{\alpha}' = (B_0 (x) + \lambda_1 B_1 (x)) \mathbf{X}_{\alpha}; 
    \quad \mathbf{X}_{\alpha} (0; \lambda_1) = {-\sin(\pi/8) \choose \cos(\pi/8)}.
\end{equation*}
For $\mathbf{X}_b (x; \lambda_2)$, our observation that $H$ is limit point 
at $+ \infty$ allows us to conclude 
that $\mathbf{X}_b (x; \lambda_2)$
must be the unique (up to constant multiple) solution of 
$ J \mathbf{X}_b' = (B_0 (x) + \lambda_1 B_1 (x)) \mathbf{X}_b$ that lies
right in $(a, b)$. In order to find $\mathbf{X}_b (x; \lambda_2)$, 
we compute the eigenvalues of $\mathcal{B} (x; \lambda_2)$ 
for (relatively) large values of $x$. Specifically, we will 
take $\lambda_2 = .2$, and for this value we find 
$\nu_1 (5; \lambda_2) = .0039$ and 
$\nu_2 (5; \lambda_2) = 1.0724 \times 10^{15}$. The unit eigenvector 
associated with $\nu_1 (5; \lambda_2)$ is 
\begin{equation*}
    w_1 (5; \lambda_2) 
    = {-.1287022477 \choose .9916832818}.
\end{equation*}
Regarding these values, our only justification for keeping so 
many decimal places is that the value of $w_1 (x; \lambda_2)$
remains consistent to this many places as we continue to increase
$x$ beyond $5$. We emphasize that while our general theory 
requires the selection of a convergent subsequence of eigenvectors,
the actual (numerically generated) sequence of eigenvectors 
converges quickly and with extraordinary consistency. 
According to our general theory, we can take
$\mathbf{X}_b (x; \lambda_2) = \Phi (x; \lambda_2) w_1^b (\lambda_2)$,
and we'll approximate the limit-obtained vector $w_1^b (\lambda_2)$
with $w_1 (5; \lambda_2)$. 

Equipped now with frames $\mathbf{X}_{\alpha} (x; \lambda_1)$ and 
$\mathbf{X}_b (x; \lambda_2)$, we can readily compute 
\begin{equation} \label{mas-app1}
    \mas (\ell_{\alpha} (\cdot; \lambda_1), \ell_b (\cdot; \lambda_2); (0, + \infty))
\end{equation}
as a spectral flow for the matrix $\tilde{W} (x; \lambda_1)$ as 
specified in (\ref{tildeW-bc1}). 

For this example, we have the advantage of knowing in advance 
accurate values for the 13 right-most eigenvalues of $H$ on the 
interval $J_2$. The right-most five of these are as 
follows:
\begin{equation*}
   -.3154, \quad - .2946, \quad -.2542, \quad -.1613, \quad .1332,
\end{equation*}
obtained from \cite{AGM2006}, in which the values are 
actually computed to substantially higher precision
than presented here. We will illustrate our approach 
by counting the right-most four eigenvalues, and also by providing 
the full Maslov box associated with this calculation.
For this, we will keep $\lambda_2 = .2$ as above, and 
set $\lambda_1 = -.3100$. Computing (\ref{mas-app1}) via a spectral flow 
for $\tilde{W} (x; \lambda_1)$, we identify conjugate points
at $14.5$, $20.2$, $26.8$, and $33.7$, after which 
$\tilde{W} (x; \lambda_1)$ begins to oscillate through 
values in the third quadrant of the complex plane. 
(These conjugate points can be obtained with much greater
precision, but there's no advantage in this.) We conclude 
that in this case
\begin{equation*}
    \mathcal{N}^{\alpha} ((\lambda_1, \lambda_2))
    = \mas (\ell_a (\cdot; \lambda_1), \ell_b (\cdot; \lambda_2); (0, +\infty))
    = 4,
\end{equation*}
as expected. This is the entirety of the necessary 
calculation associated with the number of eigenvalues that 
$H$ has on the interval $(-.31, .2)$, but in order to illustrate the idea, we provide
the full Maslov box associated with this calculation, along with 
the relevant spectral curves (see Figure \ref{maslov-box-figure}, 
created with MATLAB.) 
In this figure, we see clearly that each spectral curve
intersects the boundary of the Maslov box precisely twice, 
once along the left shelf and once along the top shelf. 
Intersections along the top shelf correspond with eigenvalues
of $H$, and so it is exactly this correspondence (via the 
spectral curves) that allows us to count conjugate points
along the left shelf rather than along the top shelf. We
emphasize that, strictly speaking, the top shelf should be 
associated with a limit as $x \to + \infty$, but the dynamics
are already thoroughly apparent for $x = 50$, as depicted. 
As discussed in \cite{HS2}, the monotonicity of the spectral 
curves in this figure is a general feature of 
renormalized oscillation theory, and  
follows from monotonicity in $\lambda$ along horizontal shelves and the 
monotonicity in $x$ of Claim \ref{crossing-claim}.

\begin{figure}[ht] \label{maslov-box-figure}  
\begin{center}\includegraphics[%
  width=14cm,
  height=10cm]{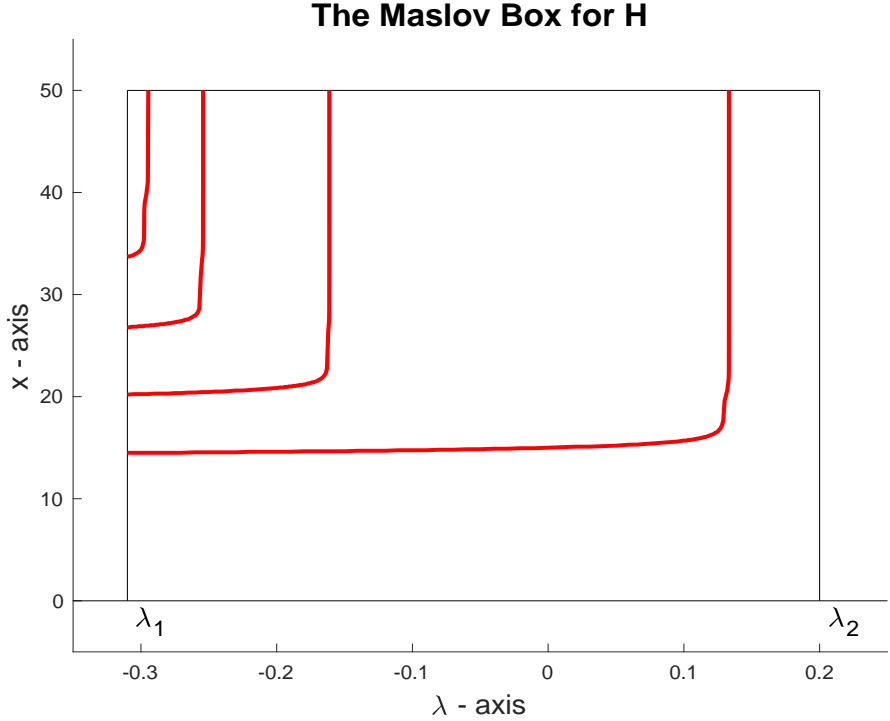}\end{center}
\caption{The Full Maslov Box for $H$ on $[-.31, .2]$.}
\end{figure}

\subsection{Energy Levels for the Hydrogen Atom}
\label{hydrogen-section}

When Schr\"odinger's equation for the hydrogen atom is 
expressed in spherical coordinates and analyzed by 
separation of variables, the resulting radial equation 
can be expressed in the form 
\begin{equation} \label{radial-equation}
    H \phi := 
    - \frac{1}{x^2} (x^2 \phi')' - \frac{\gamma}{x} \phi + \frac{\ell (\ell+1)}{x^2} \phi
    = \lambda \phi,
\end{equation}
where $\gamma > 0$ is a physical constant 
and $\ell$ is an integer associated with angular momentum
(see, e.g., Chapter 12 in \cite{Ga1974}). 
The natural domain for (\ref{radial-equation})
is $(0, \infty)$, and it's clear that $H$ is singular 
at both endpoints. In order to place this equation in our setting, we 
set $y = {y_1 \choose y_2} = {\phi \choose x^2 \phi'}$,
from which we arrive at (\ref{linear-hammy}) with 
\begin{equation*}
    B_0 (x) + \lambda B_1 (x)
    = \begin{pmatrix}
    \gamma x - \ell (\ell + 1) & 0 \\
    0 & \frac{1}{x^2}
    \end{pmatrix}
    + \lambda \begin{pmatrix}
    x^2 & 0 \\
    0 & 0
    \end{pmatrix}.
\end{equation*}
It's well-known that any self-adjoint extension of the 
minimal operator associated with $H$ has essential spectrum 
$[0, +\infty)$ (see, e.g., \cite{R1966}). 
The eigenvalues of $H$ are typically reported in physics
literature to be 
\begin{equation} \label{hydrogen-eigenvalues}
    \lambda_n = - (\frac{\gamma}{2n})^2, 
    \quad n = \ell + 1, \ell + 2, \dots
\end{equation}
(see, e.g., \cite{Ga1974}), and in this section we 
would like to understand how this relation should
be interpreted in our setting. (See Remark 
\ref{equivalence-remark2} below for a
formulation of $H$, including its precise domain.)
For computational purposes, we'll take 
$\gamma = 4$, and we'll focus on the case $\ell = 0$, 
which is particularly interesting from our point of view
because $H$ is limit-circle at $x = 0$ in this case, 
whereas it is limit-point at $x = 0$ for all 
$\ell \ge 1$. 

We begin by setting $\lambda_0 = i$ and verifying 
(numerically) that $H$ is limit-circle at $x=0$. 
In this case, we initialize the fundamental matrix
$\Phi (x; \lambda_0)$ at $x = 1$, and we compute 
the eigenvalues of $\mathcal{A} (x; \lambda_0)$, 
as $x$ tends toward $0$. At $x = 10^{-5}$, we find
$\mu_1 (10^{-5}; \lambda_0) = -.7478$ and 
$\mu_2 (10^{-5}; \lambda_0) = .3343$, with both 
values stable as $x$ continues to decrease, 
suggesting that $H$ is indeed limit-circle
at $x = 0$. Respectively,
we find the associated unit eigenvectors to be 
\begin{equation*}
    v_1 (10^{-5}; \lambda_0) = {.7834 \choose -.0001+.6216i},
    \quad v_2 (10^{-5}; \lambda_0) = {.0001 + .6216i \choose .7834},
\end{equation*}
and we take these vectors as approximations for the 
limit-obtained eigenvectors $v_1^a (\lambda_0)$
and $v_2^a (\lambda_0)$.
As discussed in Section \ref{operator-section}, 
there will be a single Niessen space for this 
problem, and it will be spanned by two elements that 
both lie left in $(0, + \infty)$, namely 
$y_1^a (x; \lambda_0) = \Phi (x; \lambda_0) v_1^a (\lambda_0)$
and $y_2^a (x; \lambda_0) = \Phi (x; \lambda_0) v_2^a (\lambda_0)$. 
In order to specify our boundary condition at 
$x = 0$, we also need to compute 
\begin{equation*}
    \rho = \sqrt{-\mu_1 (\lambda_0)/\mu_2(\lambda_0)}
    = 1.4956,
\end{equation*}
and select some $\beta \in \mathbb{C}$ with 
$|\beta| = \rho$. (See the discussion leading
into Lemma \ref{krall-niessen-lemma}.) 
Given this choice, we will specify 
our boundary condition via the element 
\begin{equation*}
    U^a (x; \lambda_0) 
    = \Phi (x; \lambda_0) (v_1^a (\lambda_0) + \beta v_2^a (\lambda_0)).
\end{equation*}
We emphasize that each choice of $\beta$ from the circle 
$|\beta| = \rho$ will correspond with 
a different boundary condition, and so for a different 
self-adjoint restriction of $H$. In order to fix a 
specific case, we will take $\beta$ to be the real 
value $\beta_1 = 1.4956$,
where the subscript anticipates that we will later consider
an alternative choice. 

Next, we fix $\lambda_1 = -5$, and construct a frame 
$\mathbf{X}_a (x; \lambda_1)$ satisfying 
\begin{equation} \label{Xsuba-app2}
    J \mathbf{X}_a' = (B_0 (x) + \lambda_1 B_1 (x)) \mathbf{X}_a; 
    \quad \lim_{x \to a^+} U^a (x; \lambda_0)^* J \mathbf{X}_a (x; \lambda_1) = 0. 
\end{equation}
In order to do this, we work with the matrix $\mathcal{B} (x; \lambda_1)$,
for which we compute the eigenvalues $\{\nu_j (x; \lambda_1)\}_{j=1}^2$ and
the associated eigenvectors $\{w_j (x; \lambda_1)\}_{j=1}^2$ as $x$
tends to $0$. Taking an approximation obtained by evaluating 
$\mathcal{B} (x; \lambda_1)$ at $x = 10^{-5}$, we obtain the 
approximate values $\nu_1^a (\lambda_1) = - .4205$, $\nu_2^a (\lambda_1) = -.1106$,
with associated approximate limit-obtained unit vectors 
\begin{equation*}
    w_1^a (\lambda_1) = {-.8615 \choose .5077},
    \quad w_2^a (\lambda_1) = {-.5077 \choose -.8615}.
\end{equation*}
We can now compute $\mathbf{X}_a (x; \lambda_1)$ as a 
linear combination 
\begin{equation*}
    \mathbf{X}_a (x; \lambda_1) = \Phi(x; \lambda_1) (c_1 w_1^a (\lambda_1) + c_2 w_2^a (\lambda_1)), 
\end{equation*}
for some appropriate constants $c_1$ and $c_2$. In particular, 
$c_1$ and $c_2$ are determined by the limit specified in 
(\ref{Xsuba-app2}). We can express this as 
\begin{equation*}
c_1 \lim_{x \to a^+} U^a (x; \lambda_0)^* J \Phi(x; \lambda_1) w_1^a (\lambda_1)    
+ c_2 \lim_{x \to a^+} U^a (x; \lambda_0)^* J \Phi(x; \lambda_1) w_2^a (\lambda_1)  
= 0.
\end{equation*}
We approximate the limits by evaluation at $x = 10^{-5}$ to obtain 
\begin{equation*}
    \begin{aligned}
    \lim_{x \to a^+} U^a (x; \lambda_0)^* J \Phi(x; \lambda_1) w_1^a (\lambda_1) 
    &\cong -1.2050 + 1.2050i \\
    \lim_{x \to a^+} U^a (x; \lambda_0)^* J \Phi(x; \lambda_1) w_2^a (\lambda_1) 
    &\cong -.6139 + .6139i.
    \end{aligned}
\end{equation*}
It follows immediately that we can choose $c_1$ and $c_2$ to be 
$c_1 = 1$, $c_2 = (-1.2050 + 1.2050i)/(-.6139 + .6139i) = - 1.9629$. We
conclude that 
\begin{equation*}
    \mathbf{X}_a (x; \lambda_1) 
    = \Phi (x; \lambda_1) w^a (\lambda_1);
    \quad w^a (\lambda_1) = {.0613 \choose .9981},
\end{equation*}
where $w^a (\lambda_1)$ has been normalized to have unit
length. 

We now turn to the right endpoint $b = + \infty$. If we 
evaluate $\mathcal{A} (x; i)$ at $x = 25$, we obtain 
eigenvalues $\mu_1 (25; i) = 1.9352 \times 10^{-22}$
and $\mu_2 (25; i) = 4.6925 \times 10^{11}$. This indicates
that $\mu_2 (x; i)$ is tending toward $+ \infty$ 
as $x$ increases to $+ \infty$, and we conclude that 
$H$ is limit-point at $b = + \infty$. This means that
no additional boundary condition is necessary at 
$b = + \infty$. We will denote by $H_{\beta_1}$ 
the operator obtained from $H$ by adding our choice of 
boundary condition taken above at the left endpoint. 

\begin{remark} \label{spectrum2-rigorously} Similarly 
as with our first application, these calculations have 
not been rigorously justified, but the limit-circle/point 
conclusions have been rigorously justified elsewhere. 
In particular, if we adopt the change of variables
$\phi = \psi/x$, then (\ref{radial-equation}) with 
$\ell = 0$ becomes
\begin{equation*}
    \mathcal{H} \psi := - \psi'' - \frac{\gamma}{x} \psi 
    = \lambda \psi,
\end{equation*}
which is known to be limit-circle at $x=0$ and limit-point
at $+ \infty$ (see, e.g., \cite{E2005}).
\end{remark}

In an effort to count the first three eigenvalues of 
$H$, we will set $\lambda_2 = -3/8$, and in order 
to compute $\mathbf{X}_b (x; \lambda_2)$, we will 
compute the eigenvalues and eigenvectors of 
$\mathcal{B} (x; \lambda_2)$ as $x$ tends toward 
$+ \infty$. Taking $x = 40$ in this case, we 
find $\nu_1 (40; -3/8) = 6.3054$ and 
$\nu_2 (40; -3/8) = 3.7724 \times 10^{11}$. 
The unit eigenvector associated with $\nu_1 (40; -3/8)$
is 
\begin{equation*}
    w_1 (40; -3/8) 
    = {-.3357895545 \choose .9419370335},
\end{equation*}
where similarly as with our previous application, the 
number of decimals given is simply an indication 
of the consistent values as $x$ continues to 
increase. We use $w_1 (40; -3/8)$ as an approximation
of $w_1^b (-3/8)$, and we set 
$\mathbf{X}_b (x; \lambda_2) = \Phi (x; \lambda_2) w_1^b (-3/8)$.

Equipped now with frames $\mathbf{X}_a (x; \lambda_1)$ and 
$\mathbf{X}_b (x; \lambda_2)$, we can readily compute 
\begin{equation} \label{mass-app2}
    \mas (\ell_a (\cdot; \lambda_1), \ell_b (\cdot; \lambda_2); (0, + \infty))
\end{equation}
as a spectral flow for the matrix $\tilde{W} (x; \lambda_1)$ as 
specified in (\ref{singular-tildeW}). We find conjugate points 
at approximately $x = 1.95$ and $x = 5.00$, after which the 
value of $\tilde{W} (x; \lambda_1)$ remains near $-1$, without
crossing, as $x$ continues to increase. We conclude that 
$H_{\beta_1}$ has two eigenvalues on the interval $[-5, -3/8]$. 

Naively, we might have expected to find three eigenvalues on the 
interval $[-5, -3/8]$ (namely, $-4$, $-1$, $-4/9$), but we recall
that the eigenvalues given in (\ref{hydrogen-eigenvalues}) correspond
with a particular choice of boundary condition (based on physical
considerations). In particular, the argument from physics goes
roughly as follows. For $\ell = 0$, equation (\ref{radial-equation})
has two linearly independent solutions, one of which is bounded
as $x$ approaches $0$, while the other is unbounded. (Both of 
which correspond via the above relation  
$y = {y_1 \choose y_2} = {\phi \choose x^2 \phi'}$ with 
functions that lie left in $(0,+\infty)$.) Based on 
physical arguments, the unbounded solution is generally 
eliminated, and this effectively selects a particular 
left-hand boundary condition. Precisely, this physical argument
asserts that we need to identify a fixed vector $w \in \mathbb{C}^2$ so that 
$\mathbf{X}_a (x; \lambda_1) = \Phi (x; \lambda_1) w$ 
remains bounded as $x$ approaches 0. By a straightforward 
minimization argument, we find $w = {.7121 \choose -.7020}$.
This solution corresponds with a particular choice of 
$\beta$. In particular, we can identify the value of 
$\beta \in \mathbb{C}$, $|\beta| = \rho$ so that 
\begin{equation*}
    \lim_{x \to 0^+} \Big(\Phi (x; \lambda_0) (v_1 (\lambda_0) + \beta v_2 (\lambda_0))\Big)^*
    J \Phi (x; \lambda_1) w = 0.
\end{equation*}
We can approximate $\beta$ by setting $x = 10^{-5}$ and computing 
\begin{equation*}
    \beta \cong - \frac{v_1 (\lambda_0)^* \Phi (x; \lambda_0)^* J \Phi (x; \lambda_1)w}
    {v_2 (\lambda_0)^* \Phi (x; \lambda_0)^* J \Phi (x; \lambda_1)w}
    = .2952 - 1.4663i.
\end{equation*}
Using this choice of $\beta$ leads to a new boundary condition, 
specified via $U^a (x; \lambda_0)$, and consequently to a new 
operator $H_{\beta_2}$. Computing (\ref{mass-app2}) in this case, 
we count three eigenvalues by virtue of conjugate points at 
$.68$, $2.00$, and $5.00$. 

We conclude with the following remark, addressing some details
that have been set aside during the discussion of this application. 

\begin{remark} \label{equivalence-remark2}
It's natural to view $H$ as an operator on a weighted Hilbert
space $L^2_{x^2} ((0, \infty), \mathbb{C})$ with inner product
\begin{equation*}
    \langle \phi, \psi \rangle_{x^2}
    = \int_0^{+\infty} x^2 \phi (x) \bar{\psi} (x) dx. 
\end{equation*}
With this specification, $H$ is self-adjoint on the domain 
\begin{equation*}
    \begin{aligned}
    \dom(H) &= \Big{\{}\phi \in L^2_{x^2} ((0,\infty),\mathbb{C}): 
    \phi, \phi' \in \AC_{\loc} ((0,\infty),\mathbb{C}), \\
    & \quad H\phi \in L^2_{x^2} ((0,\infty),\mathbb{C}), \,
    \lim_{x \to 0^+} \Big(\Phi (x; \lambda_0) (v_1 (\lambda_0) + \beta v_2 (\lambda_0)) \Big)^*
    J {\phi (x) \choose x^2 \phi' (x)}  = 0 
    \Big{\}}.
    \end{aligned}
\end{equation*}
Likewise, the operator $\mathcal{H}$ from Remark \ref{spectrum2-rigorously} is self-adjoint on the domain 
\begin{equation*}
    \begin{aligned}
    \dom(\mathcal{H}) &= \Big{\{}\psi \in L^2 ((0,\infty),\mathbb{C}): 
    \psi, \psi' \in \AC_{\loc} ((0,\infty),\mathbb{C}), \\
    & \quad \mathcal{H} \psi \in L^2 ((0,\infty),\mathbb{C}), \,
    \lim_{x \to 0^+} \Big(\Psi (x; \lambda_0) (v_1 (\lambda_0) + \beta v_2 (\lambda_0)) \Big)^*
    J {\psi (x) \choose \psi' (x)} = 0 
    \Big{\}},
    \end{aligned}
\end{equation*}
where $\Psi (x; \lambda)$ is a fundamental matrix associated with 
$\mathcal{H}$,
\begin{equation*}
    J \Psi' = \mathcal{B} (x; \lambda) \Psi; 
    \quad \Psi (1; \lambda) =
    \begin{pmatrix}
    1 & 0 \\
    1 & 1
    \end{pmatrix},
    \quad \mathcal{B} (x; \lambda) =
    \begin{pmatrix}
    \frac{\gamma}{x}+\lambda & 0 \\
    0 & 1
    \end{pmatrix}.
\end{equation*}

With these precise specifications, it's straightforward to verify 
that $H$ and $\mathcal{L}$ (the latter constructed as in 
Lemma \ref{self-adjoint-operator-lemma})
have precisely the same sets of essential spectrum, and 
also the same sets of discrete eigenvalues. In addition, 
these spectral sets also agree with their counterparts
for $\mathcal{H}$. 
\end{remark}

\bigskip
{\it Acknowledgments.} The authors are grateful to Yuri Latushkin for bringing 
\cite{GZ2017} to their attention, for suggesting that the problem could be 
approached with the Maslov index, and for several helpful conversations along
the way.

\end{document}